\documentclass{amsart}

\usepackage{amsmath,amsfonts,amssymb}
\usepackage{xcolor,mhsetup}
\usepackage[extra,safe]{tipa}
\usepackage{mathtools}
\usepackage[left=2cm,right=2cm,top=2cm,bottom=2cm]{geometry}
\usepackage[applemac]{inputenc}
\usepackage{mathrsfs}
\usepackage{hyperref}
\usepackage{tabularx}

\usepackage{graphicx}
\usepackage{diagbox}
\usepackage{multicol}
\usepackage{multirow}
\newcommand{\tabincell}[2]{\begin{tabular}{@{}#1@{}}#2\end{tabular}}
\usepackage{relsize}

\theoremstyle{definition}
\newtheorem{theorem}{Theorem}[section]

\newtheorem{definition}[theorem]{Definition}
\newtheorem{example}[theorem]{Example}
\newtheorem{lem}{Lemma}[section]

\theoremstyle{remark}

\numberwithin{equation}{section}

\def \rtheta{\overrightarrow{\theta}}

\def\my_c{c_\infty}

\def \z{{\mathbf{z}}}

\def \l{{\mathbf{l}}}

\newcommand{\mynewtheorem}[2]{
  \newaliascnt{#1}{dummy}
  \newtheorem{#1}[#1]{#2}
  \aliascntresetthe{#1}
  \expandafter\def\csname #1autorefname\endcsname{#2}
}

\newcommand{\be}{\begin{equation}}
\newcommand{\ee}{\end{equation}}
\newcommand{\bde}{\begin{displaymath}}
\newcommand{\ede}{\end{displaymath}}
\newcommand{\beq}{\begin{eqnarray*}}
\newcommand{\eeq}{\end{eqnarray*}}
\newcommand{\beqa}{\begin{eqnarray}}
\newcommand{\eeqa}{\end{eqnarray}}
\newcommand{\bel }{\left\{\begin{array}{ll}}
\newcommand{\eel}{\cr \end{array} \right.}

\newcommand{\seq}[1]{{\lbrace #1 \rbrace}}

\newcommand{\dcb}{\begin{array}{lll}}
\newcommand{\dce}{\end{array}}
\newcommand{\ebe}{\begin{enumerate}\setlength{\baselineskip}{13pt}\setlength{\parskip}{0pt}}
\newcommand{\dbe}{\end{enumerate}}

\def\s{\bold{s}}
\def\x{\bold{x}}
\def\z{\bold{z}}

\newcommand{\E}{\mathcal{E}}

\def\P{{\mathbb P}}

\def\E{\mathbb{E}}

\def\I{\mathsf{1}}

\newcommand \A[1]{{\bf (#1)}}

\def\N{{\mathbb{N}} }
\def\E{{\mathbb{E}}  }

\def\P{{\mathbb{P}}  }

\def\I{{\mathbf{1}}}

\def\bint#1^#2{\displaystyle{\int_{#1}^{#2}}}
\def\bsum#1^#2{\displaystyle{\sum_{#1}^{#2}}}

\def\xdt_#1{X_#1(\Delta t)}

\def\0{{\mathbf{0}}}

\def\x{{\mathbf{x}}}

\def\y{{\mathbf{y}}}

\begin{document}

\title[\normalsize \textit{P\MakeLowercase{robabilistic representation of} IBP \MakeLowercase{formulae for some stochastic volatility models with unbounded drift}}]{\Large P\MakeLowercase{robabilistic representation of integration by parts formulae for some stochastic volatility models with unbounded drift}}

\author{Junchao Chen}
\address{Junchao Chen, Universit\'e de Paris, Laboratoire de Probabilit\'es, Statistique et Mod\'elisation (LPSM), F-75013 Paris, France}
\email{juchen@lpsm.paris}

\author{Noufel Frikha}
\address{Noufel Frikha, Universit\'e de Paris, Laboratoire de Probabilit\'es, Statistique et Mod\'elisation (LPSM), F-75013 Paris, France.}
\email{frikha@lpsm.paris}

\author{Houzhi Li}
\address{Houzhi Li, Universit\'e de Paris, Laboratoire de Probabilit\'es, Statistique et Mod\'elisation (LPSM), F-75013 Paris, France}
\email{houzhi.li@lpsm.paris}


\maketitle
\begin{abstract}
In this paper, we establish a probabilistic representation as well as some integration by parts formulae for the marginal law at a given time maturity of some stochastic volatility model with unbounded drift. Relying on a perturbation technique for Markov semigroups, our formulae are based on a simple Markov chain evolving on a random time grid for which we develop a tailor-made Malliavin calculus. Among other applications, an unbiased Monte Carlo path simulation method stems from our formulas so that it can be used in order to numerically compute with optimal complexity option prices as well as their sensitivities with respect to the initial values or \emph{Greeks} in finance, namely the \emph{Delta} and \emph{Vega}, for a large class of non-smooth European payoff. Numerical results are proposed to illustrate the efficiency of the method.
\end{abstract}

\section{Introduction}
In this work, we consider a two dimensional stochastic volatility model given by the solution of the following stochastic differential equation (SDE for short) with dynamics 
\begin{equation}
\label{stochastic:volatility:model:sde}
\left\{
\begin{array}{ll}
S_t & = s_0+ \int_0^t r S_s \, ds + \int_0^t \sigma_S(Y_s) S_s \, dW_s,\\ 
 Y_t & = y_0 + \int_0^t b_Y(Y_s) \, ds + \int_0^t \sigma_Y(Y_s) \, dB_s, \\
  d\langle B, W\rangle_s & = \rho \, ds
\end{array}
\right.
\end{equation}

\noindent where the coefficients $b_Y,\, \sigma_S, \, \sigma_Y: \mathbb{R} \rightarrow \mathbb{R}$ are smooth functions, $W$ and $B$ are one-dimensional standard Brownian motions with correlation factor $\rho \in (-1, 1)$ both being defined on some probability space $(\Omega, \mathcal{F}, \mathbb{P})$ .

The aim of this article is to prove a probabilistic representation formula for two integration by parts (IBP) formulae for the marginal law of the process $(S,Y)$ at a given time maturity $T$. To be more specific, for a given starting point $(s_0,y_0) \in (0,\infty)\times \mathbb{R}$ and a given finite time horizon $T>0$, we establish two Bismut-Elworthy-Li (BEL) type formulae for the two following quantities
 \begin{equation}\label{two:derivatives:for:ibp:formula}
 \partial_{s_0} \mathbb{E}\left[h(S_T, Y_T)\right] \quad \textnormal{ and } \quad \partial_{y_0} \mathbb{E}\left[h(S_T, Y_T)\right]
 \end{equation}

\noindent where $h$ is a real-valued possibly non-smooth payoff function defined on $[0,\infty) \times \mathbb{R}$.

Such IBP formulae have attracted a lot of interest during the last decades both from a theoretical and a practical point of views as they can be further analyzed to derive properties related to the transition density of the underlying process or to develop Monte Carlo simulation algorithm among other practical applications, see e.g. Nualart \cite{Nuabook}, Malliavin and Thalmaier \cite{malliavinT} and the references therein. They are also of major interest for computing sensitivities, also referred as to \emph{Greeks} in finance, of arbitrage price of financial derivatives which is the keystone for hedging purpose, i.e. for protecting the value of a portfolio against some possible changes in sources of risk. The two quantities appearing in \eqref{two:derivatives:for:ibp:formula} corresponds respectively to the Delta and Vega of the European option with payoff $h(S_T, Y_T)$. For a more detailed discussion on this topic, we refer the interested reader to Fourni\'e and al. \cite{fournie:1},\cite{fournie:2} for IBP formulae related to European, Asian options and conditional expectations, Gobet and al.\cite{Gobet:Kohatsu:1}, \cite{Gobet:Kohatsu:2} for IBP formulae related to some barrier or lookback options. Let us importantly point out that, from a numerical point of view, the aforementioned IBP formulae will inevitably involve a time discretization procedure of the underlying process and Malliavin weights, thus introducing two sources of error given by a bias and a statistical error, as it is already the case for the computation of the price $\mathbb{E}[h(S_T,Y_T)]$.  

Relying on a perturbation argument for the Markov semigroup generated by the couple $(X,Y)$, we first establish a probabilistic representation formula for the marginal law $(S_T, Y_T)$ for a fixed prescribed maturity $T>0$ based on a simple Markov chain evolving along a random time grid given by the jump times of an independent renewal process. Such probabilistic representation formula was first derived in Bally and Kohatsu-Higa \cite{BK} for the marginal law of a multi-dimensional diffusion process and of some L\'evy driven SDEs with bounded drift, diffusion and jump coefficients. Still in the case of bounded coefficients, it was then further investigated in Labord\`ere and al. \cite{henry-labordere2017}, Agarwal and Gobet \cite{GobetA} for multi-dimensional diffusion processes and in Frikha and al. \cite{frikha2019} for one-dimensional killed processes. The major advantage of the aforementioned probabilistic formulae lies in the fact that an unbiased Monte Carlo simulation method directly stems from it. Thus, it may be used to numerically compute an option price with optimal complexity since its computation will be only affected by the statistical error. However, let us emphasize that in general the variance of the Monte Carlo estimator tends to be large or even infinite. In order to circumvent this issue, an importance sampling scheme based on the law of the jump times of the underlying renewal process has been proposed in Anderson and Kohatsu-Higa \cite{APKA} in the multi-dimensional diffusion framework and in \cite{frikha2019} for one-dimensional killed processes.

The main novelty of our approach in comparison with the aforementioned works is that we allow the drift coefficient $b_Y$ to be possibly unbounded as it is the case in most stochastic volatility models (Stein-Stein, Heston, ...). Such boundedness condition on the drift coefficient has appeared persistently in the previous contributions and is actually essential since basically it allows to remove the drift in the choice of the approximation process in order to derive the probabilistic representation formula. The key ingredient that we here develop in order to remove this restriction consists in choosing adequatly the approximation process around which the original perturbation argument of the Markov semigroup $(X,Y)$ is done by taking into account the transport of the initial condition by the deterministic ordinary differential equation (ODE) having unbounded coefficient\footnote{This dynamical system is obtained by removing the noise, that is, by setting $\sigma_Y\equiv 0$, from the dynamics of $Y$ in \eqref{stochastic:volatility:model:sde}.}. The approximation process, or equivalently the underlying Markov chain on which the probabilistic representation is based, is then obtained from the original dynamics \eqref{stochastic:volatility:model:sde} by \emph{freezing the coefficients $b_Y$, $\sigma_S$ and $\sigma_Y$ along the flow of this ODE}. We stress that the previous choice is here crucial since it provides the adequate approximation process on which some good controls can be established. To the best of our knowledge, this feature appears to be new in this context. 

Having this probabilistic representation formula at hand together with the tailor-made Malliavin calculus machinery for this well-chosen underlying Markov chain, in the spirit of the BEL formula established in \cite{frikha2019} for killed diffusion processes with bounded drift coefficient, we rely on a propagation of the spatial derivatives forward in time then perform local IBP formulas on each time interval of the random time grid and finally merge them in a suitable manner in order to establish the two BEL formulae for the two quantities \eqref{two:derivatives:for:ibp:formula}. Following the ideas developed in \cite{APKA}, we achieve finite variance for the Monte Carlo estimators obtained from the probabilistic representation formulas of the couple $(S_T, Y_T)$ and of both IBP formulae by selecting adequatly the law of the jump times of the renewal process. We finally provide some numerical tests illustrating our previous analysis.

The article is organized as follows. In Section \ref{section:preliminaries}, we introduce our assumptions on the coefficients, present the approximation process that will be the main building block for our perturbation argument as well as the Markov chain that will play a central role in our probabilistic representation for the marginal law of the process $(X,Y)$ and for our IBP formulae. In addition, we construct the taillor-made Malliavin calculus machinery related to the underlying Markov chain upon which both IBP formulae are made. In Section \ref{sec:probabilistic:representation}, relying on the Markov chain introduced in Section \ref{section:preliminaries}, we establish in Theorem \ref{theorem:probabilistic:formulation} the probabilistic representation formula for the coupled $(S_T,Y_T)$. In Section \ref{section:ibp:formulae}, we establish the BEL formulae for the two quantities appearing in \eqref{two:derivatives:for:ibp:formula}. The main result of this section is Theorem \ref{thm:ibp:formulae}. Some numerical results are presented in Section \ref{section:numerical:results}. The proofs of Theorem \ref{theorem:probabilistic:formulation} and of some other technical but important results are postponed to the appendix of Section \ref{section:appendix}.

\subsection*{Notations:} For a fixed time $T$ and positive integer $n$, we will use the following notation for time and space variables $\s_n=(s_1,\cdots ,s_n)$, $\x_n = (x_1, \cdots, x_n)$, the differentials $d\s_n = ds_1 \cdots ds_n$, $d\x_n = dx_1 \cdots dx_n$ and also introduce the simplex $\Delta_n(T) := \left\{ \mathbf{s}_n \in [0,T]^n: 0 \leq s_1 < \cdots s_{n} \leq T \right\}$. 

In order to deal with time-degeneracy estimates, we will often use the following space-time inequality:
\begin{equation}
\label{space:time:inequality}
\forall p, q >0, \, \forall x \in \mathbb{R}, \quad |x|^p e^{-q|x|^2} \leq (p/(2qe))^{p/2}.
\end{equation}

For two positive real numbers $\alpha$ and $\beta$, we define the Mittag-Leffler function $z \mapsto E_{{\alpha ,\beta }}(z)=\sum _{{k=0}}^{\infty } {z^{k}}/{\Gamma (\alpha k+\beta )}$. For a positive integer $d$, we denote by $\mathcal{C}^{\infty}_p(\mathbb{R}^d)$ the space of real-valued functions which are infinitely differentiable on $\mathbb{R}^d$ with derivatives of any order having polynomial growth.

\section{Preliminaries: assumptions, definition of the underlying Markov chain and related Malliavin calculus} \label{section:preliminaries}

\subsection{Assumptions}
Throughout the article, we work on a probability space $(\Omega, \mathcal{F}, \mathbb{P})$ which is assumed to be rich enough to support all random variables that we will consider in what follows. We will work under the following assumptions on the coefficients:
\begin{trivlist}
\item[\A{AR}] The coefficients $\sigma_S$ and $\sigma_Y$ are bounded and smooth, in particular $\sigma_S$ and $\sigma_Y$ belong to $\mathcal{C}^{\infty}_b(\mathbb{R})$. The drift coefficient $b_Y$ belongs to $\mathcal{C}^{\infty}(\mathbb{R})$ and admits bounded derivatives of any order greater than or equal to one. In particular, the drift coefficient $b_Y$ may be unbounded.
\item[\A{ND}] There exists $\kappa \geq1$ such that for all $x \in \mathbb{R}$, 
$$
\kappa^{-1} \leq a_S(x) \leq \kappa, \quad \kappa^{-1} \leq a_Y(x) \leq \kappa
$$
\noindent where $a_S = \sigma^2_S$ and $a_Y= \sigma^2_Y$. Therefore, without loss of generality, we will assume that both $\sigma_S$ and $\sigma_Y$ are positive function.
\end{trivlist}

Apply It\^o's Lemma to $X_t = \ln(S_t) $. We get
\begin{equation}
\label{new:sde:to:approximate}
\left\{
\begin{array}{ll}
X_t & = x_0+ \int_0^t \Big(r -\frac12 a_S(Y_s)\Big) \,  ds + \int_0^t \sigma_S(Y_s)  \, dW_s,\\ 
 Y_t & = y_0 + \int_0^t b_Y(Y_s) \, ds + \int_0^t \sigma_Y(Y_s) \, dB_s, \\
  d\langle B, W\rangle_s & = \rho \, ds,
\end{array}
\right.
\end{equation}

\noindent with $x_0 = \ln(s_0)$. Without loss of generality, we will thus work with the Markov semigroup associated to the process $(X,Y)$, namely $P_t h(x_0, y_0) = \mathbb{E}[h(X_t, Y_t)]$. 
\subsection{Choice of the approximation process}\label{subsection:choice:approximation:process}

As already mentioned in the introduction, our strategy here is based on a probabilistic representation of the marginal law, in the spirit of the unbiased simulation method introduced for diffusion processes by Bally and Kohatsu-Higa \cite{BK}, see also Labord\`ere and al. \cite{henry-labordere2017}, and investigated from a numerical perspective by Andersson and Kohatsu-Higa \cite{APKA}. We also mention the recent contribution of one the author with Kohatsu-Higa and Li \cite{frikha2019} for IBP formulae for the marginal law of one-dimensional killed diffusion processes.

However, at this stage, it is important to point out that our choice of approximation process significantly differs from the four aforementioned references. Indeed, in the previous contributions, the drift is assumed to be bounded and basically plays no role so that one usually removes it in the dynamics of the approximation process. In order to handle the unbounded drift term $b_Y$ appearing in the dynamics of the volatility process, one has to take into account the transport of the initial condition by the ODE obtained by removing the noise in the dynamics of $Y$. To be more specific, we denote by $(m_t(s, y))_{t\in [s,T]}$, $0\leq s \leq T$, the unique solution to the ODE $\dot{m}_t = b_Y(m_t)$ with initial condition $m_s=y$. Observe that by time-homogeneity of the coefficient $b_Y$, one has $m_t(s, y)= m_{t-s}(0, y)$. We will simplify the notation when $s=0$ and write $m_t(y_0)$ for $m_t(0, y_0)$. When there is no ambiguity, we will often omit the dependence with respect to the initial point $y_0$ and we only write $m_t$ for $m_t(y_0)$. We now introduce the approximation process $(\bar{X}, \bar{Y})$ defined by
\begin{equation}
\label{proxy:sde}
\left\{
\begin{array}{ll}
\bar{X}^{x_0}_t  &  = x_0 +  \int_0^t (r -\frac12 a_S(m_s)) \, ds + \int_0^t \sigma_S(m_s) \, dW_s,\\ 
\bar{Y}^{y_0}_t &   =  y_0 + \int_0^t b_Y(m_s) \, ds + \int_0^t \sigma_Y(m_s) \, dB_s, \\
  d\langle B, W\rangle_s & = \rho \, ds.
\end{array}
\right.
\end{equation}

We will make intensive use of the explicit form of the Markov semigroup $(\bar{P}_t)_{t\in [0,T]}$ defined for any bounded measurable map $h:\mathbb{R}^2 \rightarrow \mathbb{R}$ by $\bar{P}_{t} h(x_0, y_0) = \mathbb{E}[h(\bar{X}^{x_0}_t, \bar{Y}^{y_0}_t)]$.

\begin{lem}\label{definition:density:proxy} Let $(x_0,y_0) \in \mathbb{R}^2$, $\rho \in (-1,1)$ and $t \in (0,\infty)$. Then, for any bounded and measurable map $h : \mathbb{R}^2 \rightarrow \mathbb{R}$, it holds
\begin{align}
\bar{P}_{t} h(x_0, y_0) &  = \int_{\mathbb{R}^2} h(x, y) \, \bar{p}(t, x_0, y_0, x, y)  \, dx dy \label{proxy:semigroup}
\end{align}

\noindent with
\begin{align*}
\bar{p}(t, x_0, y_0, x, y) & = \frac{1}{2 \pi \sigma_{S, t} \sigma_{Y, t} \sqrt{1-\rho_t^2}} \exp\Big(-\frac12 \frac{(x-x_0 - (r t - \frac{1}{2} a_{S, t}))^2}{a_{S, t} (1-\rho_t^2)} -\frac12 \frac{(y-m_t)^2}{a_{Y, t }  (1-\rho_t^2)}\Big) \\
& \quad \times \exp\Big( \rho_t \frac{(x-x_0 - (rt-\frac{1}{2} a_{S, t}))(y-m_t)}{\sigma_{S, t} \sigma_{Y, t} (1-\rho_t^2)}  \Big)
\end{align*} 

\noindent where we introduced the notations 
\begin{align*}
a_{S, t} & = a_{S, t}(y_0) := \sigma^2_{S, t} := \int_0^t a_{S}(m_s(y_0)) \, ds,\\ 
a_{Y, t} & = a_{Y, t}(y_0) := \sigma^2_{Y, t} := \int_0^t a_{Y}(m_s(y_0)) \, ds, \\
\sigma_{S, Y, t} & = \sigma_{S, Y, t}(y_0) := \int_0^t (\sigma_{S} \sigma_Y)(m_s(y_0))  \, ds, \\
\rho_t & := \rho \sigma_{S, Y, t}/(\sigma_{S, t} \sigma_{Y, t}).
\end{align*}

Moreover, for any $t\in (0,T]$, there exists some positive constant $C:=C(T, \rho, a, r, \kappa)$ such that
\begin{equation}
\label{upper:bound:transition:density:approximation:semigroup}
\forall t\in (0,T], \quad \bar{p}(t, x_0, y_0, x, y) \leq C \bar{q}_{4\kappa}(t, x_0, y_0, x, y)
\end{equation}
\noindent where, for a positive parameter $c$, we introduced the density function 
\begin{equation}
\label{def:kernel:bar:q}
(x, y) \mapsto \bar{q}_{c}(t, x_0, y_0, x, y) := \frac{1}{2 \pi c t } \exp\Big(- \frac{(x-x_0)^2}{2ct} -  \frac{(y-m_{t})^2}{2c t}\Big).
\end{equation}


\end{lem}

\begin{proof} We write
\begin{align*}
(\bar{X}^{x_0}_t, \bar{Y}^{y_0}_t) = \Big(x_0 +  r t -\frac12 a_{S, t}  + \int_0^t \sigma_{S}(m_s) \, dW_s , m_t +  \int_0^t \sigma_Y(m_s) \, \Big(\rho dW_s+  \sqrt{1-\rho^2}  d\widetilde{W}_s\Big) \Big)
\end{align*}

\noindent where $\widetilde{W}$ is a one-dimensional standard Brownian motion independent of $W$. We thus deduce that $(\bar{X}^{x_0}_t, \bar{Y}^{y_0}_t) \sim \mathcal{N}(\mu(t, x_0, y_0), \Sigma_{t})$ with
$$
\mu(t, x_0, y_0) = \Big(x_0 +  r t -\frac12 a_{S, t} ,  m_{t}\Big) \quad  \textnormal{and} \quad \Sigma_{t} =
 \left(
\begin{matrix}
a_{S, t}&   \rho \sigma_{S, Y, t} \\
 \rho \sigma_{S, Y, t}  & a_{Y, t}
\end{matrix}
\right).
$$

The expression of the transition density then readily follows. Now, from \A{ND}, it is readily seen that $a_{S, t}, a_{Y, t} \leq \kappa t$ so that using the inequalities $(a-b)^2 \geq \frac12 a^2 - b^2$ and $\rho^2_t \leq \rho^2$, it follows
\begin{align*}
\bar{p}(t, x_0, y_0, x, y) & = \frac{1}{2 \pi \sigma_{S, t}\sigma_{Y, t} \sqrt{1-\rho_t^2}} \exp\Big(-\frac12 \frac{(x-x_0 - (rt-\frac{1}{2} a_{S, t}))^2}{a_{S, t} (1-\rho_t^2)} -\frac12 \frac{(y-m_t)^2}{a_{Y, t}  (1-\rho_t^2)}\Big) \\
& \quad \times \exp\Big( \frac{\rho_t}{1-\rho_t^2} \frac{(x-x_0 - (rt-\frac{1}{2} a_{S, t}))(y-m_t)}{\sigma_{S, t}\sigma_{Y, t}}  \Big)\\
& \leq C \frac{1}{2 \pi (2\kappa) t } \exp\Big(- (4\kappa)^{-1} \frac{(x-x_0)^2}{2t} - (4\kappa)^{-1} \frac{(y-m_{t})^2}{2 t}\Big) \\
& =: C \bar{q}_{4\kappa}(t, x_0, y_0, x, y)
\end{align*}

\noindent for some positive constants $C:=C(T, \lambda, \rho, a, r, \kappa)$.

\end{proof}

We will also use the notation $(\bar{X}^{s, x}_t, \bar{Y}^{s, y}_{t})_{t\geq s}$ for the approximation process starting from $(x, y)$ at time $s$ and with coefficients frozen along the deterministic flow $\left\{m_{t}(s, y) = m_{t-s}(y) , t\geq s\right\}$. Note that the corresponding Markov semigroup satisfies $\bar{P}_{s, t}h(x, y) := \mathbb{E}\left[h(\bar{X}^{s, y}_t, \bar{Y}^{s, y}_t)\right] = \mathbb{E}\left[h(\bar{X}^{y}_{t-s}, \bar{Y}^{y}_{t-s})\right] = \bar{P}_{t-s}h(x, y)$.

\subsection{Markov chain on random time grid}

The first tool that we will employ is a renewal process $N$ that we now introduce.
\begin{definition}\label{counting:process}
Let $\tau:=(\tau_n)_{n\geq1}$ be a sequence of random variables such that $(\tau_{n} - \tau_{n-1})_{n\geq 1}$, with the convention $\tau_0=0$, are i.i.d. with positive density function $f$ and cumulant distribution function $t\mapsto F(t) = \int_{-\infty}^{t} f(s) \, ds$. Then, the renewal process $N:=(N_t)_{t\geq0}$ with jump times $\tau$ is defined by $N_t:=\sum_{n\geq1} \textbf{1}_{\left\{ \tau_n \leq t\right\}}$.
\end{definition}

It is readily seen that, for any $t>0$, $\left\{ N_{t} = n \right\} = \left\{\tau_n \leq t < \tau_{n+1} \right\}$ and by an induction argument that we omit, one may prove that the joint distribution of $(\tau_1, \cdots, \tau_n)$ is  given by
$$
\mathbb{P}(\tau_1 \in ds_1, \cdots, \tau_n \in ds_n) = \prod_{j=0}^{n-1} f(s_{i+1}-s_i) \textbf{1}_{\left\{ 0 < s_1 < \cdots < s_n\right\}}
$$

\noindent which in turn implies
\begin{align}
\mathbb{E}[\textbf{1}_{\left\{ N_{t} =n \right\}} \Phi(\tau_1, \cdots, \tau_n)] & = \mathbb{E}[\textbf{1}_{\left\{\tau_n \leq t < \tau_{n+1}\right\}} \Phi(\tau_1, \cdots, \tau_n)] \nonumber \\
& = \int_{t}^{\infty} \int_{\Delta_n(t)} \Phi(s_1, \cdots, s_n) \prod_{j=0}^{n} f(s_{j+1}-s_j) \, d\bold{s}_{n+1} \nonumber
\end{align}

\noindent with the convention $s_0= 0$. Hence, by Fubini's theorem, it holds
\begin{align}
\mathbb{E}[\textbf{1}_{\left\{ N_{t} =n \right\}} \Phi(\tau_1, \cdots, \tau_n)] & = \int_{\Delta_n(t)} \Phi(s_1, \cdots, s_n) (1- F(t-s_{n})) \prod_{j=0}^{n-1} f(s_{j+1}-s_j) \, d\bold{s}_{n} \label{probabilistic:representation:time:integrals}
\end{align}

\noindent for any measurable map $\Phi: \Delta_n(t) \rightarrow \mathbb{R}$ satisfying $\mathbb{E}[\textbf{1}_{\left\{ N_{t} =n \right\}} |\Phi(\tau_1, \cdots, \tau_n)|] < \infty$.
  
 Usual choices that we will consider are the followings.

\begin{example}\label{counting:process:examples}

\begin{enumerate}
\item If the density function $f$ is given by $f(t) = \lambda e^{-\lambda t} \textbf{1}_{[0,\infty)}(t)$ for some positive parameter $\lambda$, then $N$ is a Poisson process with intensity $\lambda$.
  
\item If the density function $f$ is given by $f(t) = \frac{1-\alpha}{\bar{\tau}^{1-\alpha}}\frac{1}{t^{\alpha}} \textbf{1}_{[0, \bar{\tau}]}(t)$ for some parameters $(\alpha, \bar{\tau})\in (0,1) \times (0,\infty)$, then $N$ is a renewal process with $[0,\bar{\tau}]$-valued $Beta(1-\alpha, 1)$ jump times.
\item More generally, if the density function $f$ is given by $f(t) = \frac{\bar{\tau}^{1-\alpha-\beta}}{B(\alpha,\beta)}\frac{1}{t^{1-\alpha}(\bar{\tau}-t)^{1-\beta}} \textbf{1}_{[0, \bar{\tau}]}(t)$ for some parameters $(\alpha, \beta, \bar{\tau})\in (0,1)^2 \times (0,\infty)$, then $N$ is a renewal process with $[0,\bar{\tau}]$-valued $Beta(\alpha, \beta)$ jump times.
 \end{enumerate}
\end{example}

Given a sequence $Z = (Z^{1}_{n}, Z^{2}_n)_{n\geq1}$ of i.i.d. random vector with law $\mathcal{N}(0, I_2)$ and a renewal process $N$ independent of $Z$ with jump times $(\tau_i)_{i\geq0}$, we set $\zeta_{i} = \tau_i \wedge T$, with the convention $\zeta_0= 0$, and we consider the two-dimensional Markov chain $(\bar{X}, \bar{Y})$ with $(\bar{X}_0, \bar{Y}_0) = (x_0, y_0)$ at time $0$ (evolving on the random time grid $(\zeta_i)_{i\geq0}$) and with dynamics for any $0\leq i \leq N_T$
\begin{equation}\label{euler:scheme}
\left\{
\begin{array}{rl}
\bar{X}_{i+1} & = \bar{X}_{i} +  \Big(r (\zeta_{i+1}-\zeta_i)  -\frac12 a_{S,i}\Big) + \sigma_{S, i} Z^{1}_{i+1},\\
\bar{Y}_{i+1} & = m_i + \sigma_{Y,i} \left( \rho_i Z^1_{i+1} + \sqrt{1-\rho_i^2} Z^{2}_{i+1}\right),
\end{array}
\right.
\end{equation}

\noindent where we introduced the notations 
\begin{align*}
a_{S, i} & := \sigma^2_{S, i} := a_{S, \zeta_{i+1}-\zeta_i}(\bar{Y}_i) = \int_0^{\zeta_{i+1}-\zeta_i} a_S(m_s(\bar{Y}_i)) \, ds, \\
a_{Y, i} &  := \sigma^2_{Y, i} := a_{Y, \zeta_{i+1}-\zeta_i}(\bar{Y}_i) = \int_0^{\zeta_{i+1}-\zeta_i} a_Y(m_s(\bar{Y}_i)) \, ds, \\  
\sigma_{S, Y, i} & : = \int_0^{\zeta_{i+1}-\zeta_i} (\sigma_S \sigma_Y)(m_s(\bar{Y}_i)) \, ds, \\
\rho_i & := \rho_{\zeta_{i+1}-\zeta_i}(\bar{Y}_i) = \rho \frac{\sigma_{S, Y, i}}{\sigma_{S, i} \sigma_{Y,i}},\\ 
m_{i} & := m_{\zeta_{i+1}-\zeta_i}(\bar{Y}_{i}). 
\end{align*}

We will denote by $\sigma'_{S, i}$ the first derivative of $y\mapsto \sigma_{S, i}(y)$ taken at $\bar{Y}_i$ and proceed similarly for the quantities $\sigma'_{Y, i}, \, \sigma'_{S, Y, i}, \, \rho'_i$ and $m'_i$.
We define the filtration $\mathcal{G} = (\mathcal{G}_i)_{i\geq0}$ where $\mathcal{G}_{i} = \sigma(Z^1_j, Z^{2}_j, \, 0\leq j \leq i)$, for $i\geq1$ and $\mathcal{G}_0$ stands for the trivial $\sigma$-field. We assume that the filtration $\mathcal{G}$ satisfies the usual conditions. For an integer $n$, we will use the notations $\zeta^{n}=(\zeta_0, \cdots, \zeta_n)$ and $\tau^{n}=(\tau_0, \cdots, \tau_n)$.

\subsection{Tailor-made Malliavin calculus for the Markov chain $(\bar{X}, \bar{Y})$.}\label{sec:tailor:mall:calculus}

In this section we introduce a tailor-made Malliavin calculus for the underlying Markov chain $(\bar{X}, \bar{Y})$ defined by \eqref{euler:scheme} which will be employed in order to establish our IBP formulae. Instead of using an infinite dimensional calculus as it is usually done in the literature, see e.g. Nualart \cite{Nuabook}, the approach developed below is based on a finite dimensional calculus for which the dimension is given by the number of jumps of the underlying renewal process involved in the Markov chain $(\bar{X}, \bar{Y})$.

\begin{definition} Let $n \in \mathbb{N}$. For any $i \in \left\{ 0, \cdots, n\right\}$, we define the set $\mathbb{S}_{i, n}(\bar{X}, \bar{Y})$, as the space of random variables $H$ such that
\begin{itemize}
\item $H = h(\bar{X}_{i}, \bar{Y}_i, \bar{X}_{i+1}, \bar{Y}_{i+1}, \zeta^{n+1})$, on the set $\left\{ N_T = n \right\}$, where we recall $\zeta^{n+1} := (\zeta_0, \cdots, \zeta_{n+1}) = (0, \zeta_1, \cdots, \zeta_{n}, T)$.
\item For any $\s_{n+1} \in \Delta_{n+1}(T)$, the map $h(.,.,.,., \s_{n+1}) \in \mathcal{C}^{\infty}_p(\mathbb{R}^4)$. 
\end{itemize}

\end{definition}

For a r.v. $H \in \mathbb{S}_{i, n}(\bar{X}, \bar{Y})$, we will often abuse the notations and write
$$
H \equiv H(\bar{X}_{i}, \bar{Y}_i, \bar{X}_{i+1}, \bar{Y}_{i+1}, \zeta^{n+1})
$$  

\noindent that is the same symbol $H$ may denote the r.v. or the function in the set $\mathbb{S}_{i, n}(\bar{X}, \bar{Y})$. One can easily define the flow derivatives for $H \in \mathbb{S}_{i, n}(\bar{X}, \bar{Y})$ as follows
\begin{align}
\partial_{\bar{X}_{i+1}} H & = \partial_3 h(\bar{X}_{i}, \bar{Y}_i, \bar{X}_{i+1}, \bar{Y}_{i+1}, \zeta^{n+1}), \nonumber \\
\partial_{\bar{Y}_{i+1}} H & = \partial_4 h(\bar{X}_{i}, \bar{Y}_i, \bar{X}_{i+1}, \bar{Y}_{i+1}, \zeta^{n+1}), \nonumber \\
\partial_{\bar{X}_{i}} H & = \partial_1 h(\bar{X}_{i}, \bar{Y}_i, \bar{X}_{i+1}, \bar{Y}_{i+1}, \zeta^{n+1}) + \partial_3 h(\bar{X}_{i}, \bar{Y}_i, \bar{X}_{i+1}, \bar{Y}_{i+1}, \zeta^{n+1})  \partial_{\bar{X}_i} \bar{X}_{i+1}, \nonumber \\
\partial_{\bar{Y}_{i}} H & = \partial_2 h(\bar{X}_{i}, \bar{Y}_i, \bar{X}_{i+1}, \bar{Y}_{i+1}, \zeta^{n+1}) + \partial_3 h(\bar{X}_{i}, \bar{Y}_i, \bar{X}_{i+1}, \bar{Y}_{i+1}, \zeta^{n+1})  \partial_{\bar{Y}_i} \bar{X}_{i+1} + \partial_4 h(\bar{X}_{i}, \bar{Y}_i, \bar{X}_{i+1}, \bar{Y}_{i+1}, \zeta^{n+1})  \partial_{\bar{Y}_i} \bar{Y}_{i+1}, \nonumber
\end{align}
\noindent and from the dynamics \eqref{euler:scheme}
\begin{align}
 \partial_{\bar{X}_i} \bar{X}_{i+1} & = 1, \nonumber \\
 \partial_{\bar{Y}_i} \bar{Y}_{i+1} & = m'_{i} + \sigma'_{Y, i} \Big(\rho_i Z^1_{i+1} + \sqrt{1-\rho_i^2} Z^2_{i+1}\Big) + \sigma_{Y, i} \frac{\rho'_i}{\sqrt{1-\rho^2_i}}\left(\sqrt{1-\rho^2_i} Z^1_{i+1} - \rho_i Z^{2}_{i+1} \right), \label{partial:flow:deriv:Y:bar:process} \\
\partial_{\bar{Y}_i} \bar{X}_{i+1}   & = - \frac12 a'_{S,i}  + \sigma'_{S, i}Z^{1}_{i+1} = - \frac12 a'_{S, i} + \frac{ \sigma'_{S, i}}{\sigma_{S, i}} \Big(\bar{X}_{i+1}-\bar{X}_i - \Big(r(\zeta_{i+1}-\zeta_i) -\frac12 a_{S, i}\Big) \Big). \label{partial:Y:flow:deriv:X:bar:process}
\end{align}

We now define the integral and derivative operators for $H \in \mathbb{S}_{i, n}(\bar{X}, \bar{Y})$, as
\begin{align}
\mathcal{I}^{(1)}_{i+1}(H) & = H \Big[\frac{Z^{1}_{i+1}}{\sigma_{S,i}(1-\rho_i^2)} - \frac{\rho_i}{1-\rho_i^2} \frac{\rho_i Z^1_{i+1} + \sqrt{1-\rho_i^2} Z^2_{i+1}}{\sigma_{S, i}} \Big] - \mathcal{D}^{(1)}_{i+1} H, \label{integral:operator:1} \\
\mathcal{I}^{(2)}_{i+1}(H) &= H \Big[\frac{\rho_i Z^1_{i+1} + \sqrt{1-\rho_i^2} Z^{2}_{i+1}}{\sigma_{Y, i} (1-\rho_i^2)} - \frac{\rho_i}{1-\rho_i^2} \frac{ Z^{1}_{i+1} }{\sigma_{Y, i}} \Big] - \mathcal{D}^{(2)}_{i+1} H, \label{integral:operator:2}\\
\mathcal{D}^{(1)}_{i+1} H & = \partial_{\bar{X}_{i+1}} H, \label{derivative:operator:1} \\
\mathcal{D}^{(2)}_{i+1} H & = \partial_{\bar{Y}_{i+1}} H.\label{derivative:operator:2}
\end{align}

Note that due to the above definitions and assumption \A{H}, it is readily checked that $\mathcal{I}^{(1)}_{i+1}(H), \mathcal{I}^{(2)}_{i+1}(H), \mathcal{D}^{(1)}_{i+1} H $ and $\mathcal{D}^{(2)}_{i+1} H$ are elements of $\mathbb{S}_{i, n}(\bar{X}, \bar{Y})$ so that we can define iterations of the above operators. Namely, by induction, for a multi-index $\alpha = (\alpha_1, \cdots, \alpha_p)$ of length $p$ with $\alpha_i \in \left\{ 1, 2\right\}$ and $\alpha_{p+1} \in \left\{1, 2\right\}$, we define
$$
\mathcal{I}^{(\alpha, \alpha_{p+1})}_{i+1}(H) = \mathcal{I}^{(\alpha_{p+1})}_{i+1}(\mathcal{I}^{(\alpha)}_{i+1}(H)), \quad \mathcal{D}^{(\alpha, \alpha_{p+1})}_{i+1} H = \mathcal{D}^{(\alpha_{p+1})}_{i+1}(\mathcal{D}^{(\alpha)}_{i+1} H)
$$
\noindent with the intuitive notation $(\alpha, \alpha_{p+1}) = (\alpha_1, \cdots, \alpha_{p+1})$.

Throughout the article, we will use the following notation for a certain type of conditional expectation that will be frequently employed. For any $X\in L^{1}(\mathbb{P})$ and any $i \in \left\{ 0, \cdots, n\right\}$,
$$
\mathbb{E}_{i, n}[X] = \mathbb{E}[X | \mathcal{G}_i, \tau^{n+1}, N_T= n]
$$

\noindent where we recall that we employ the notation $\tau^{n+1}=(\tau_0, \cdots, \tau_{n+1})$. Having the above definitions and notations at hand, the following duality formula is satisfied: for any non-empty multi-index $\alpha$ of length $p$, with $\alpha_i \in \left\{1, 2\right\}$, for any $i \in \left\{1, \cdots, p\right\}$, $p$ being a positive integer, it holds
\begin{equation}
\mathbb{E}_{i,n}\Big[\mathcal{D}^{(\alpha)}_{i+1} f(\bar{X}_{i+1}, \bar{Y}_{i+1}) H \Big] = \mathbb{E}_{i,n}\Big[ f(\bar{X}_{i+1}, \bar{Y}_{i+1}) \mathcal{I}^{(\alpha)}_{i+1}(H) \Big]. \label{duality:formula}
\end{equation}

 In order to obtain explicit norm estimates for random variables in $\mathbb{S}_{i,n}(\bar{X}, \bar{Y})$, it is useful to define for $H \in \mathbb{S}_{i,n}(\bar{X}, \bar{Y})$, $i\in \left\{0, \cdots, n\right\}$ and $p\geq1$
$$
\|H\|^{p}_{p, i, n} = \mathbb{E}_{i,n}[|H|^p].
$$

We will also employ a chain rule formula for the integral operators defined above.
\begin{lem}\label{lem:chain:rule:formula} Let $H \equiv H(\bar{X}_i, \bar{Y}_i, \bar{X}_{i+1}, \bar{Y}_{i+1}, \zeta^{n+1}) \in \mathbb{S}_{i, n}(\bar{X}, \bar{Y})$, for some $i\in \left\{0, \cdots, n\right\}$. The following chain rule formulae hold for any $(\alpha_1,\alpha_2) \in \left\{1,2\right\}^2$
\begin{align}
\partial_{\bar{X}_{i}}\mathcal{I}^{(\alpha_1)}_{i+1}(H) & = \mathcal{I}^{(\alpha_1)}_{i+1}(\partial_{\bar{X}_i} H) , \quad \partial_{\bar{X}_{i}}\mathcal{I}^{(\alpha_1,\alpha_2)}_{i+1}(H) = \mathcal{I}^{(\alpha_1,\alpha_2)}_{i+1}(\partial_{\bar{X}_i} H).  \label{chain:rule:formula:integral:operators:first:coordinate:deriv}
\end{align}
Moreover, one has
\begin{align}
\partial_{\bar{Y}_{i}}\mathcal{I}^{(1)}_{i+1}(H) &=  \mathcal{I}^{(1)}_{i+1}(\partial_{\bar{Y}_i}H) - \frac{\sigma'_{S, i}}{\sigma_{S, i}} \mathcal{I}^{(1)}_{i+1}(H) -  \frac{\rho'_i}{1-\rho^2_i} \frac{\sigma_{Y, i}}{\sigma_{S, i}}  \mathcal{I}^{(2)}_{i+1}(H), 
\label{chain:rule:formula:first:order:integral:operator:1:second:coordinate:deriv} \\
 \partial_{\bar{Y}_i} \mathcal{I}^{(2)}_{i+1}(H) &= \mathcal{I}^{(2)}_{i+1}(\partial_{\bar{Y}_i}H) - \left(\frac{\sigma'_{Y,i}}{\sigma_{Y,i}} - \frac{\rho'_i \rho_i}{1-\rho^2_i} \right) \mathcal{I}^{(2)}_{i+1}(H),\label{chain:rule:formula:first:order:integral:operator:2:second:coordinate:deriv} \\
 \partial_{\bar{Y}_{i}}\mathcal{I}^{(1,1)}_{i+1}(H) & = \mathcal{I}^{(1,1)}_{i+1}(\partial_{\bar{Y}_i}H) - 2 \frac{\sigma'_{S,i}}{\sigma_{S,i}} \mathcal{I}^{(1,1)}_{i+1}(H) -  \frac{\rho'_i}{1-\rho^2_i}  \frac{\sigma_{Y, i}}{\sigma_{S, i}}  \left( \mathcal{I}^{(1, 2)}_{i+1}(H) +  \mathcal{I}^{(2,1)}_{i+1}(H)\right) , \label{chain:rule:formula:second:order:integral:operator:1:second:coordinate:deriv}\\
  \partial_{\bar{Y}_{i}}\mathcal{I}^{(2,2)}_{i+1}(H) & =  \mathcal{I}^{(2,2)}_{i+1}(\partial_{\bar{Y}_i}H) - 2 \left(\frac{\sigma'_{Y,i}}{\sigma_{Y,i}} - \frac{\rho'_i \rho_i}{1-\rho^2_i} \right) \mathcal{I}^{(2, 2)}_{i+1}(H),\label{chain:rule:formula:second:order:integral:operator:2:second:coordinate:deriv}\\
    \partial_{\bar{Y}_{i}}\mathcal{I}^{(1,2)}_{i+1}(H) & = \mathcal{I}^{(1,2)}_{i+1}(\partial_{\bar{Y}_i}H) - \left(\frac{\sigma'_{S,i}}{\sigma_{S,i}} + \frac{\sigma'_{Y,i}}{\sigma_{Y,i}} - \frac{\rho'_i \rho_i}{1-\rho^2_i} \right) \mathcal{I}^{(1, 2)}_{i+1}(H) -  \frac{\rho'_i}{1-\rho^2_i} \frac{\sigma_{Y, i}}{\sigma_{S, i}}  \mathcal{I}^{(2, 2)}_{i+1}(H).
    \label{chain:rule:formula:second:order:integral:operator:cross:second:coordinate:deriv}
\end{align}

\end{lem}

\begin{proof} Observe that from the very definitions \eqref{integral:operator:1} and \eqref{integral:operator:2}, one directly gets 
$$
\partial_{\bar{X}_i} \mathcal{I}^{(1)}_{i+1}(1) = \partial_{\bar{X}_i} \mathcal{I}^{(2)}_{i+1}(1) = 0
$$
\noindent while, also by direct computation, we obtain 
\begin{align*}
\partial_{\bar{Y}_i}  \mathcal{I}^{(1)}_{i+1}(1) & =  -\frac{\sigma'_{S,i}}{\sigma_{S,i}} \mathcal{I}^{(1)}_{i+1}(1) - \frac{\rho'_i}{1-\rho^2_i} \frac{\sigma_{Y, i}}{\sigma_{S, i}} \mathcal{I}^{(2)}_{i+1}(1), \\ \partial_{\bar{Y}_i}  \mathcal{I}^{(2)}_{i+1}(1) & = - \left(\frac{\sigma'_{Y,i}}{\sigma_{Y,i}} - \frac{\rho'_i \rho_i}{1-\rho^2_i}\right)\mathcal{I}^{(2)}_{i+1}(1).
\end{align*}
We thus deduce
\begin{align*}
\partial_{\bar{X}_i} \mathcal{I}^{(\alpha_1)}_{i+1}(H) & = \partial_{\bar{X}_i} H \mathcal{I}^{(\alpha_1)}_{i+1}(1) + H \partial_{\bar{X}_i} \mathcal{I}^{(\alpha_1)}_{i+1}(1) - \partial_{\bar{X}_i}\mathcal{D}^{(\alpha_1)}_{i+1}H \\
& = \partial_{\bar{X}_i} H\mathcal{I}^{(\alpha_1)}_{i+1}(1) - \mathcal{D}^{(\alpha_1)}_{i+1}( \partial_{\bar{X}_i} H) \\
& = \mathcal{I}^{(\alpha_1)}_{i+1}(\partial_{\bar{X}_i} H)
\end{align*}
\noindent where we used the fact $\mathcal{D}^{(\alpha_1)}_{i+1}\partial_{\bar{X}_i} H = \partial_{\bar{X}_i} \mathcal{D}^{(\alpha_1)}_{i+1} H $ which easily follows by direct computation. As a consequence, it is readily seen
\begin{align*}
\partial_{\bar{X}_{i}}\mathcal{I}^{(\alpha_1,\alpha_2)}_{i+1}(H) &  = \partial_{\bar{X}_i} \mathcal{I}^{(\alpha_2)}_{i+1}(\mathcal{I}^{(\alpha_1)}_{i+1}(H)) = \mathcal{I}^{(\alpha_2)}_{i+1}( \partial_{\bar{X}_i}\mathcal{I}^{(\alpha_1)}_{i+1}(H))  = \mathcal{I}^{(\alpha_2)}_{i+1}(\mathcal{I}^{(\alpha_1)}_{i+1}(\partial_{\bar{X}_i}H))= \mathcal{I}^{(\alpha_1,\alpha_2)}_{i+1}(\partial_{\bar{X}_i}H).
\end{align*}

This concludes the proof of \eqref{chain:rule:formula:integral:operators:first:coordinate:deriv}. The chain rule formulae \eqref{chain:rule:formula:first:order:integral:operator:1:second:coordinate:deriv}, \eqref{chain:rule:formula:first:order:integral:operator:2:second:coordinate:deriv}, \eqref{chain:rule:formula:second:order:integral:operator:1:second:coordinate:deriv}, \eqref{chain:rule:formula:second:order:integral:operator:2:second:coordinate:deriv} and \eqref{chain:rule:formula:second:order:integral:operator:cross:second:coordinate:deriv} follow from similar arguments. Let us prove \eqref{chain:rule:formula:first:order:integral:operator:1:second:coordinate:deriv} and \eqref{chain:rule:formula:first:order:integral:operator:2:second:coordinate:deriv}. The proofs of \eqref{chain:rule:formula:second:order:integral:operator:1:second:coordinate:deriv}, \eqref{chain:rule:formula:second:order:integral:operator:2:second:coordinate:deriv} and \eqref{chain:rule:formula:second:order:integral:operator:cross:second:coordinate:deriv} are omitted. Observe first that in general $\mathcal{D}^{(\alpha_1)}_{i+1}\partial_{\bar{Y}_i} H \neq \partial_{\bar{Y}_i} \mathcal{D}^{(\alpha_1)}_{i+1} H $. Indeed, by standard computations, it holds 
\begin{align*}
\partial_{\bar{Y}_i} \mathcal{D}^{(1)}_{i+1}H & = \partial_{\bar{Y}_i} \partial_{\bar{X}_{i+1}} h(\bar{X}_i, \bar{Y}_i, \bar{X}_{i+1}, \bar{Y}_{i+1}, \zeta^{n}) \\
& = \partial^2_{2, 3} h(\bar{X}_i, \bar{Y}_i, \bar{X}_{i+1}, \bar{Y}_{i+1}, \zeta^{n}) + \partial^2_3 h(\bar{X}_i, \bar{Y}_i, \bar{X}_{i+1}, \bar{Y}_{i+1}, \zeta^{n}) \partial_{\bar{Y}_i}\bar{X}_{i+1} + \partial^2_{4,3} h(\bar{X}_i, \bar{Y}_i, \bar{X}_{i+1}, \bar{Y}_{i+1}, \zeta^{n}) \partial_{\bar{Y}_i}\bar{Y}_{i+1},\\
\mathcal{D}^{(1)}_{i+1} \partial_{\bar{Y}_i}H &= \partial_{\bar{X}_{i+1}} \partial_{\bar{Y}_i} h(\bar{X}_i, \bar{Y}_i, \bar{X}_{i+1}, \bar{Y}_{i+1}, \zeta^{n}) \\
&= \partial_{\bar{X}_{i+1}} (\partial_{2} h(\bar{X}_i, \bar{Y}_i, \bar{X}_{i+1}, \bar{Y}_{i+1}, \zeta^{n}) + \partial_3 h(\bar{X}_i, \bar{Y}_i, \bar{X}_{i+1}, \bar{Y}_{i+1}, \zeta^{n}) \partial_{\bar{Y}_i}\bar{X}_{i+1} + \partial_{4} h(\bar{X}_i, \bar{Y}_i, \bar{X}_{i+1}, \bar{Y}_{i+1}, \zeta^{n}) \partial_{\bar{Y}_i}\bar{Y}_{i+1}) \\
&= \partial^2_{3,2} h(\bar{X}_i, \bar{Y}_i, \bar{X}_{i+1}, \bar{Y}_{i+1}, \zeta^{n}) + \partial^2_3 h(\bar{X}_i, \bar{Y}_i, \bar{X}_{i+1}, \bar{Y}_{i+1}, \zeta^{n}) \partial_{\bar{Y}_i}\bar{X}_{i+1} + \partial_3 h(\bar{X}_i, \bar{Y}_i, \bar{X}_{i+1}, \bar{Y}_{i+1}, \zeta^{n}) \partial_{\bar{X}_{i+1}} \partial_{\bar{Y}_i}\bar{X}_{i+1} \\
& \quad + \partial^2_{3,4} h(\bar{X}_i, \bar{Y}_i, \bar{X}_{i+1}, \bar{Y}_{i+1}, \zeta^{n}) \partial_{\bar{Y}_i}\bar{Y}_{i+1} +  \partial_{4} h(\bar{X}_i, \bar{Y}_i, \bar{X}_{i+1}, \bar{Y}_{i+1}, \zeta^{n}) \partial_{\bar{X}_{i+1}}\partial_{\bar{Y}_i}\bar{Y}_{i+1} \\
&= \partial_{\bar{Y}_i} \mathcal{D}^{(1)}_{i+1}H + \partial_3 h(\bar{X}_i, \bar{Y}_i, \bar{X}_{i+1}, \bar{Y}_{i+1}, \zeta^{n}) \partial_{\bar{X}_{i+1}} \partial_{\bar{Y}_i}\bar{X}_{i+1} +\partial_{4} h(\bar{X}_i, \bar{Y}_i, \bar{X}_{i+1}, \bar{Y}_{i+1}, \zeta^{n}) \partial_{\bar{X}_{i+1}}\partial_{\bar{Y}_i}\bar{Y}_{i+1} \\
&= \partial_{\bar{Y}_i} \mathcal{D}^{(1)}_{i+1}H + \mathcal{D}^{(1)}_{i+1}H \partial_{\bar{X}_{i+1}} \partial_{\bar{Y}_i}\bar{X}_{i+1} +  \mathcal{D}^{(2)}_{i+1}H \partial_{\bar{X}_{i+1}}\partial_{\bar{Y}_i}\bar{Y}_{i+1} \\
& = \partial_{\bar{Y}_i} \mathcal{D}^{(1)}_{i+1}H + \frac{\sigma'_{S,i}}{\sigma_{S,i}} \mathcal{D}^{(1)}_{i+1}H + \frac{\rho'_i}{1-\rho^2_i} \frac{\sigma_{Y, i}}{\sigma_{S, i}}  \mathcal{D}^{(2)}_{i+1}H
\end{align*}
\noindent where we used the two identities $\partial_{\bar{X}_{i+1}} \partial_{\bar{Y}_i}\bar{X}_{i+1} =  \frac{\sigma'_{S, i}}{\sigma_{S, i}}$ and $\partial_{\bar{X}_{i+1}}\partial_{\bar{Y}_i}\bar{Y}_{i+1} = \frac{\rho'_i}{1-\rho^2_i} \frac{\sigma_{Y, i}}{\sigma_{S, i}} $ which readily stems from \eqref{partial:flow:deriv:Y:bar:process}, \eqref{partial:Y:flow:deriv:X:bar:process} and the dynamics \eqref{euler:scheme}.

From \eqref{integral:operator:1} and the previous identity, we thus obtain
\begin{align*}
\partial_{\bar{Y}_{i}}\mathcal{I}^{(1)}_{i+1}(H) &  = \partial_{\bar{Y}_i}\mathcal{I}^{(1)}_{i+1}(1) H + \mathcal{I}^{(1)}_{i+1}(1) \partial_{\bar{Y}_i}H - \partial_{\bar{Y}_i}\mathcal{D}^{(1)}_{i+1}H \\
& =  - \frac{\sigma'_{S,i}}{\sigma_{S,i}} \mathcal{I}^{(1)}_{i+1}(1) H -  \frac{\rho'_i}{1-\rho^2_i} \frac{\sigma_{Y, i}}{\sigma_{S, i}} \mathcal{I}^{(2)}_{i+1}(1)H  + \mathcal{I}^{(1)}_{i+1}(1) \partial_{\bar{Y}_i}H -  \mathcal{D}^{(1)}_{i+1} \partial_{\bar{Y}_i}H  + \frac{\sigma'_{S,i}}{\sigma_{S,i}} \mathcal{D}^{(1)}_{i+1}H \\
&  \quad +\frac{\rho'_i}{1-\rho^2_i} \frac{\sigma_{Y, i}}{\sigma_{S, i}} \mathcal{D}^{(2)}_{i+1}H\\
& = -\frac{\sigma'_{S,i}}{\sigma_{S,i}} \left(\mathcal{I}^{(1)}_{i+1}(1) H - \mathcal{D}^{(1)}_{i+1}H\right) + \mathcal{I}^{(1)}_{i+1}(1) \partial_{\bar{Y}_i}H -  \mathcal{D}^{(1)}_{i+1} \partial_{\bar{Y}_i}H -  \frac{\rho'_i}{1-\rho^2_i} \frac{\sigma_{Y, i}}{\sigma_{S, i}} \left( \mathcal{I}^{(2)}_{i+1}(1)H - \mathcal{D}^{(2)}_{i+1}H \right) \\
& =  \mathcal{I}^{(1)}_{i+1}(\partial_{\bar{Y}_i}H) - \frac{\sigma'_{S, i}}{\sigma_{S, i}} \mathcal{I}^{(1)}_{i+1}(H) -  \frac{\rho'_i}{1-\rho^2_i} \frac{\sigma_{Y, i}}{\sigma_{S, i}}  \mathcal{I}^{(2)}_{i+1}(H).
\end{align*}
Similarly, after some algebraic manipulations using \eqref{euler:scheme} and \eqref{partial:flow:deriv:Y:bar:process}, we get $\partial_{\bar{Y}_{i+1}} \partial_{\bar{Y}_i} \bar{Y}_{i+1} = \frac{\sigma'_{Y, i}}{\sigma_{Y, i}} - \frac{\rho'_i \rho_i}{1-\rho_i^2}$ so that
\begin{align*}
\mathcal{D}^{(2)}_{i+1} \partial_{\bar{Y}_i}H = \partial_{\bar{Y}_i} \mathcal{D}^{(2)}_{i+1}H + \mathcal{D}^{(2)}_{i+1}H \partial_{\bar{Y}_{i+1}} \partial_{\bar{Y}_i}\bar{Y}_{i+1} = \partial_{\bar{Y}_i} \mathcal{D}^{(2)}_{i+1}H + \left(\frac{\sigma'_{Y,i}}{\sigma_{Y, i}}- \frac{\rho'_i \rho_i}{1-\rho_i^2}\right) \mathcal{D}^{(2)}_{i+1}H 
\end{align*}
\noindent so that, omitting some technical details, we get
\begin{align*}
\partial_{\bar{Y}_i} \mathcal{I}^{(2)}_{i+1}(H) & =  \mathcal{I}^{(2)}_{i+1}(\partial_{\bar{Y}_i}H) - \left(\frac{\sigma'_{Y,i}}{\sigma_{Y,i}} - \frac{\rho'_i \rho_i}{1-\rho^2_i} \right) \mathcal{I}^{(2)}_{i+1}(1) H  + \left(\frac{\sigma'_{Y,i}}{\sigma_{Y, i}} - \frac{\rho'_i \rho_i}{1-\rho^2_i} \right)\mathcal{D}^{(2)}_{i+1}H\\
& = \mathcal{I}^{(2)}_{i+1}(\partial_{\bar{Y}_i}H) - \left(\frac{\sigma'_{Y,i}}{\sigma_{Y,i}} - \frac{\rho'_i \rho_i}{1-\rho^2_i} \right) \mathcal{I}^{(2)}_{i+1}(H).
\end{align*}
The identities  \eqref{chain:rule:formula:second:order:integral:operator:1:second:coordinate:deriv}, \eqref{chain:rule:formula:second:order:integral:operator:2:second:coordinate:deriv} and \eqref{chain:rule:formula:second:order:integral:operator:cross:second:coordinate:deriv} eventually follows from \eqref{chain:rule:formula:first:order:integral:operator:1:second:coordinate:deriv} and \eqref{chain:rule:formula:first:order:integral:operator:2:second:coordinate:deriv} using some simple algebraic computations.
\end{proof}

We conclude this section by introducing the following space of random variables which satisfy some time regularity estimates.
\begin{definition} Let $\ell \in \mathbb{Z}$ and $n \in \mathbb{N}$. For any $i\in \left\{0, \cdots, n\right\}$, we define the space $\mathbb{M}_{i, n}(\bar{X}, \bar{Y}, \ell/2)$ as the set of finite random variables $H \in \mathbb{S}_{i, n}(\bar{X}, \bar{Y})$ satisfying the following time regularity estimate: for any $p\geq1$, for any $c>0$, there exists some positive constants $C:=C(T), \, c'$, $T\mapsto C(T)$ being non-decreasing and $c'$ being independent of $T$, such that for any $(x_{i}, y_{i}, x_{i+1}, y_{i+1}, \s_{n+1})  \in \mathbb{R}^4 \times \Delta_{n+1}(T)$,
\begin{align}
|H(x_{i}, y_{i}, x_{i+1}, y_{i+1}, \s_{n+1})|^p & \bar{q}_c(s_{i+1} - s_{i}, x_{i}, y_{i}, x_{i+1}, y_{i+1}) \label{time:regularity:estimate:space:inequality} \\
& \leq C (s_{i+1}-s_{i})^{p \ell/2} \bar{q}_{c'}(s_{i+1} - s_{i}, x_{i}, y_{i}, x_{i+1}, y_{i+1}) \nonumber
\end{align}
\noindent where the density function $\mathbb{R}^2 \ni (x_{i+1}, y_{i+1})\mapsto  \bar{q}_{c}(s_{i+1} - s_{i}, x_{i}, y_{i}, x_{i+1}, y_{i+1})$ is defined in Lemma \ref{definition:density:proxy}.
\end{definition}

We again remark that since the space $\mathbb{M}_{i, n}(\bar{X}, \bar{Y}, \ell/2)$ is a subset of $ \mathbb{S}_{i, n}(\bar{X}, \bar{Y})$, when we say that a random variable $\mathbb{M}_{i, n}(\bar{X}, \bar{Y}, \ell/2)$ this statement is always understood on the set $\left\{N_T=n\right\}$. 

Before proceeding, let us provide a simple example of some random variables that belong to the aforementioned space. From \eqref{integral:operator:1} and the dynamics \eqref{euler:scheme} of the Markov chain $(\bar{X}, \bar{Y})$, it holds
\begin{align*}
\mathcal{I}^{(1)}_{i+1}(1) & = \Big[\frac{\bar{X}_{i+1}-\bar{X}_{i}-(r(\zeta_{i+1}-\zeta_i)-\frac12 a_{S, i})}{a_{S,i}(1-\rho_i^2)} - \frac{\rho_i}{1-\rho_i^2} \frac{\bar{Y}_{i+1}-m_{i}}{\sigma_{S,i} \sigma_{Y,i}} \Big],\\
\mathcal{I}^{(1,1)}_{i+1}(1) &= (\mathcal{I}^{(1)}_{i+1}(1))^2 - \mathcal{D}^{1}_{i+1}(\mathcal{I}^{(1)}_{i+1}(1)) = (\mathcal{I}^{(1)}_{i+1}(1))^2 - \frac{1}{a_{S,i}(1-\rho_i^2)},
\end{align*}

\noindent so that, $\mathcal{I}^{(1)}_{i+1}(1)$ and $\mathcal{I}^{(1,1)}_{i+1}(1)$ belong to $\mathbb{S}_{i,n}(\bar{X}, \bar{Y})$. Moreover, under \A{ND}, for any $p\geq1$, it holds
$$
\Big|\mathcal{I}^{(1)}_{i+1}(1)(x_{i}, y_{i}, x_{i+1}, y_{i+1}, \s_{n+1})\Big|^p \leq C\Big(1+ \frac{|x_{i+1}-x_{i}|^p}{(s_{i+1}-s_{i})^p} +  \frac{|y_{i+1}-m_{i}(y_{i})|^p}{(s_{i+1}-s_{i})^p}\Big)
$$
\noindent and similarly,
$$
\Big|\mathcal{I}^{(1,1)}_{i+1}(1)(x_{i}, y_{i}, x_{i+1}, y_{i+1}, \s_{n+1})\Big|^p \leq C\Big(1+ \frac{1}{(s_{i+1}-s_{i})^p} + \frac{|x_{i+1}-x_{i}|^{2p}}{(s_{i+1}-s_{i})^{2p}} +  \frac{|y_{i+1}-m_{i}(y_{i})|^{2p}}{(s_{i+1}-s_{i})^{2p}}\Big).
$$

Hence, from the space-time inequality \eqref{space:time:inequality}, for any $c>0$ and any $c'>c$, it holds
$$
\Big|\mathcal{I}^{(1)}_{i+1}(1)(x_{i}, y_{i}, x_{i+1}, y_{i+1}, \s_{n+1})\Big|^p \bar{q}_c(s_{i+1} - s_{i}, x_{i}, y_{i}, x_{i+1}, y_{i+1}) \leq C  (s_{i+1}-s_{i})^{-p/2} \bar{q}_{c'}(s_{i+1} - s_{i}, x_{i}, y_{i}, x_{i+1}, y_{i+1})
$$
\noindent and
$$
\Big|\mathcal{I}^{(1,1)}_{i+1}(1)(x_{i}, y_{i}, x_{i+1}, y_{i+1}, \s_{n+1})\Big|^p \bar{q}_c(s_{i+1} - s_{i}, x_{i}, y_{i}, x_{i+1}, y_{i+1}) \leq C (s_{i+1}-s_{i})^{-p} \bar{q}_{c'}(s_{i+1} - s_{i}, x_{i}, y_{i}, x_{i+1}, y_{i+1})
$$

\noindent for some positive constant $C:=C(T)$, $T\mapsto C(T)$ being non-decreasing. We thus conclude that $\mathcal{I}^{(1)}_{i+1}(1) \in \mathbb{M}_{i,n}(\bar{X},\bar{Y}, -1/2)$ and $\mathcal{I}^{(1,1)}_{i+1}(1) \in \mathbb{M}_{i,n}(\bar{X},\bar{Y}, -1)$ for any $i\in \left\{0, \cdots, n\right\}$.

A straightforward generalization of the above example is the following property that will be frequently used in the sequel. We omit its proof.
\begin{lem}\label{lem:time:degeneracy:property}
Let $n\in \mathbb{N}$ and $i \in \left\{0, \cdots, n\right\}$. Assume that $H \in \mathbb{M}_{i,n}(\bar{X}, \bar{Y}, \ell/2)$ and $\mathcal{D}^{(\alpha_1)}_i H \in \mathbb{M}_{i,n}(\bar{X}, \bar{Y}, \ell'/2)$, for some $(\ell, \ell') \in \mathbb{Z}^2$ and some $\alpha_1 \in \left\{1,2\right\}$. Then, it holds $\mathcal{I}^{(\alpha_1)}_{i}(H) \in \mathbb{M}_{i,n}(\bar{X}, \bar{Y}, ((\ell-1)\wedge \ell')/2)$. Additionally, if $\mathcal{D}^{(\alpha_2)}_i H \in \mathbb{M}_{i,n}(\bar{X}, \bar{Y}, \ell'/2)$ and $\mathcal{D}^{(\alpha_1, \alpha_2)}_i H \in \mathbb{M}_{i,n}(\bar{X}, \bar{Y}, \ell''/2)$, for some $\ell'' \in \mathbb{Z}$ and $\alpha_2\in \left\{1, 2\right\}$, then it holds $\mathcal{I}^{(\alpha_1, \alpha_2)}_{i}(H) \in \mathbb{M}_{i,n}(\bar{X}, \bar{Y}, ((\ell-2)\wedge (\ell'-1)\wedge \ell'')/2)$. Finally, if $H_1 \in \mathbb{M}_{i,n}(\bar{X}, \bar{Y}, \ell_1/2)$ and  $H_2 \in \mathbb{M}_{i,n}(\bar{X}, \bar{Y}, \ell_2/2)$ for some $(\ell_1, \ell_2) \in \mathbb{Z}^2$ then $H_1 H_2  \in \mathbb{M}_{i,n}(\bar{X}, \bar{Y}, (\ell_1+\ell_2)/2) $ and $(\zeta_{i+1}-\zeta_{i}) \mathcal{I}^{(\alpha_1)}_{i}(H_1) \in \mathbb{M}_{i,n}(\bar{X}, \bar{Y}, (\ell_1+1)/2)$.
\end{lem}

Finally, we importantly emphasize that if $H \in \mathbb{M}_{i, n}(\bar{X}, \bar{Y}, \ell/2)$ for some $n\in \mathbb{N}$, $i\in \left\{0, \cdots, n\right\}$ and $\ell \in \mathbb{Z}$, then, its conditional $L^{p}(\mathbb{P})$-moment is finite and also satisfies a time regularity estimate. More precisely, for any $p\geq1$, it holds
\begin{equation}
\label{conditional:lp:moment:space:M:I:N}
\|H\|_{p, i, n} \leq C (\zeta_{i+1}-\zeta_{i})^{\ell/2}
\end{equation}
\noindent for some positive constant $C:=C(T)$, $T\mapsto C(T)$ being non-decreasing.
Indeed, using the fact that the sequence $Z$ is independent of $N$ as well as the upper-estimate \eqref{upper:bound:transition:density:approximation:semigroup} of Lemma \ref{definition:density:proxy} and finally \eqref{time:regularity:estimate:space:inequality}, one directly gets
\begin{align*}
\|H\|^{p}_{p, i, n} & = \E_{i,n}\Big[|H(\bar{X}_{i}, \bar{Y}_{i}, \bar{X}_{i+1}, \bar{Y}_{i+1}, \zeta^{n+1})|^p \Big| \bar{X}_{i}, \bar{Y}_{i}, \tau^{n+1}, N_T=n \Big] \\
& = \int_{\mathbb{R}^2} |H(\bar{X}_{i}, \bar{Y}_{i}, x_{i+1}, y_{i+1}, \zeta^{n+1})|^p \bar{p}(\zeta_{i+1}-\zeta_{i}, \bar{X}_{i}, \bar{Y}_{i}, x_{i+1}, y_{i+1}) \, dx_{i+1} dy_{i+1} \\
& \leq C \int_{\mathbb{R}^2} |H(\bar{X}_{i}, \bar{Y}_{i}, x_{i+1}, y_{i+1}, \zeta^{n+1})|^p \bar{q}_{4\kappa}(\zeta_{i+1}-\zeta_{i}, \bar{X}_{i}, \bar{Y}_{i}, x_{i+1}, y_{i+1}) \, dx_{i+1} dy_{i+1} \\
& \leq C (\zeta_{i+1}-\zeta_{i})^{p\ell/2}
\end{align*}
\noindent so that \eqref{conditional:lp:moment:space:M:I:N} directly follows. The previous conditional $L^{p}(\mathbb{P})$-moment estimate will be used at several places in the sequel. 

\section{Probabilistic representation for the couple $(S_T, Y_T)$.}\label{sec:probabilistic:representation}
In this section, we establish a probabilistic representation for the marginal law $(S_T, Y_T)$, or equivalently, for the law of $(X_T, Y_T)$ which is based on the Markov chain $(\bar{X}, \bar{Y})$ introduced in the previous section. For $\gamma>0$, we denote by $\mathcal{B}_\gamma(\mathbb{R}^2)$ the set of Borel measurable map $h: \mathbb{R}^2 \rightarrow \mathbb{R}$ satisfying the following exponential growth assumption at infinity, namely, for some positive constant $C$, for any $ (x, y)\in \mathbb{R}^2$, 
\begin{equation}
\label{growth:assumption:test:function}
|h(x, y)| \leq  C \exp(\gamma (|x|^2+|y|^2)).
\end{equation}

\begin{theorem}\label{theorem:probabilistic:formulation} Let $T>0$.  Under assumptions \A{AR} and \A{ND}, the law of the couple $(X_T, Y_T)$ given by the unique solution to the SDE \eqref{new:sde:to:approximate} at time $T$ starting from $(x_0=\ln(s_0), y_0)$ at time $0$ satisfies the following probabilistic representation: there exists a positive constant $c:=c(T, b_Y, \kappa)$ such that for any  $0\leq \gamma< c^{-1}$ and any $h\in\mathcal{B}_\gamma(\mathbb{R}^2)$, it holds
\begin{equation}
\mathbb{E}[h(X_T, Y_T)]= \mathbb{E}\Big[h(\bar{X}_{N_T+1}, \bar{Y}_{N_T+1}) \prod_{i=1}^{N_T+1} \theta_{i}\Big] \label{probabilistic:formulation:theorem}
\end{equation}
\noindent where the random variables $\theta_i \in \mathbb{S}_{i-1, n}(\bar{X}, \bar{Y})$ are defined by
\begin{align}
\theta_{i} & =  (f(\zeta_{i}-\zeta_{i-1}))^{-1}\Big[ \mathcal{I}^{(1,1)}_{i}(c^{i}_S) -  \mathcal{I}^{(1)}_{i}(c^{i}_S)  + \mathcal{I}^{(2,2)}_{i}(c^{i}_Y) +  \mathcal{I}^{(2)}_{i}(b^{i}_Y) +\mathcal{I}^{(1,2)}_{i}(c^{i}_{Y, S}) \Big], \, 1 \leq i \leq N_T, \label{definition:first:theta:theorem:probabilistic:representation} \\
\theta_{N_{T}+1} & = (1-F(T-\zeta_{N_{T}}))^{-1}, \label{definition:theta:theorem:probabilistic:representation}
\end{align}

\noindent with 
\begin{align*}
c^{i}_S & := \frac12 \Big(a_S(\bar{Y}_{i}) - a_S(m_{i-1})\Big),  \\
c^{i}_Y & := \frac12 \Big(a_Y(\bar{Y}_{i}) - a_Y(m_{i-1})\Big),  \\
b^{i}_Y & := b_Y(\bar{Y}_{i}) - b_Y(m_{i-1}), \\
c^{i}_{Y, S} & := \rho  ((\sigma_S\sigma_Y)(\bar{Y}_{i}) - (\sigma_S\sigma_Y)(m_{i-1})).
\end{align*}

Assume furthermore that $N$ is a renewal process with $Beta(\alpha, 1)$ jump times. For any $p\geq1$ satisfying $p(\frac12-\alpha) \leq 1-\alpha$, for any $\gamma$ such that $0\leq p \gamma < c^{-1}$ and any $h\in \mathcal{B}_\gamma(\mathbb{R}^2)$, the random variable appearing inside the expectation in the right-hand side of \eqref{probabilistic:formulation:theorem} admits a finite $L^{p}(\mathbb{P})$-moment. In particular, if $\alpha=1/2$ then for any $p\geq1$, for any $h\in \mathcal{B}_\gamma(\mathbb{R}^2)$ with $0\leq p \gamma < c^{-1}$, the $L^{p}(\mathbb{P})$-moment is finite.
 
\end{theorem}

The proof of Theorem \ref{theorem:probabilistic:formulation} is postponed to Appendix \ref{proof:probabilistic:representation:marginal:law}.

\section{Integration by parts formulae}\label{section:ibp:formulae}
In this section, we establish two IBP formulae for the law of the couple $(S_T, Y_T)$. More precisely, we are interested in providing a Bismut-Elworthy-Li formula for the two quantities 
$$
\partial_{s_0} \mathbb{E}[h(S_T, Y_T)], \quad \partial_{y_0} \mathbb{E}[h(S_T, Y_T)].
$$

Our strategy is divided into two steps as follows:

\smallskip

\emph{Step 1:} The first step was performed with the probabilistic representation established in Theorem \ref{theorem:probabilistic:formulation} for the couple $(X_T, Y_T)$ involving the two-dimensional Markov chain $(\bar{X}, \bar{Y})$ evolving on a time grid governed by the jump times of the renewal process $N$. Introducing $h(x, y)= f(e^{x}, y)$ and assuming that $f$ is of polynomial growth at infinity, it is sufficient to consider the two quantities
$$
\partial_{s_0} \mathbb{E}\Big[h(\bar{X}_{N_T+1}, \bar{Y}_{N_T+1}) \prod_{i=1}^{N_T+1} \theta_{i}\Big], \quad \partial_{y_0}\mathbb{E}\Big[h(\bar{X}_{N_T+1}, \bar{Y}_{N_T+1}) \prod_{i=1}^{N_T+1} \theta_{i}\Big]
$$
 
 \noindent for $h \in \mathcal{B}_\gamma(\mathbb{R}^2)$ for some $\gamma>0$ recalling that $x_0=\ln(s_0)$.
 
 \smallskip
 
 \emph{Step 2:} At this stage, one might be tempted to perform a standard IBP formula as presented in Nualart \cite{Nuabook} on the whole time interval $[0,T]$. However, such a strategy is likely to fail. The main reason is that the Skorokhod integral of the product of weights $\prod_{i=1}^{N_T} \theta_i$ will inevitably involve the Malliavin derivative of $\theta_i$ which will in turn raise some integrability issues of the resulting Malliavin weight. The key idea that we use in order to circumvent this issue consists in performing local IBP formulae on each of the random intervals $[\zeta_{i}, \zeta_{i+1}]$, $i=0, \cdots, N_T$, that is, by using the noise of the Markov chain on this specific time interval and then by combining all these local IBP formulae in a suitable way.

\smallskip

To implement successfully our strategy, two main ingredients are needed. Our first ingredient consists in transferring the partial derivatives $\partial_{s_0}$ and $\partial_{y_0}$ on the expectation forward in time from the first time interval $[0, \zeta_1]$ to the interval on which we perform the local IBP formula, say $[\zeta_i, \zeta_{i+1}]$. Our second ingredient consists in combining these various local IBP formulae in an adequate manner. Roughly speaking, we will consider a weighted sum of each IBP formula, the weight being precisely the length of the corresponding time interval.

\subsection{The transfer of derivative formula}

\begin{lem}\label{lem:transfer:derivative}
Let $h \in \mathcal{C}^1_p(\mathbb{R}^2)$ and $n \in \mathbb{N}$. The maps $\mathbb{R}^2 \ni (x, y) \mapsto \mathbb{E}_{i, n}\Big[h(\bar{X}_{i+1}, \bar{Y}_{i+1}) \theta_{i+1} | (\bar{X}_i, \bar{Y}_i)=(x, y)\Big]$, $i \in \left\{0, \cdots, n\right\}$, belong to $\mathcal{C}^1_p(\mathbb{R}^2)$ $a.s.$ Moreover, the following transfer of derivative formulae hold 
\begin{align}
\partial_{s_0} \mathbb{E}_{0, n}\Big[h(\bar{X}_{1}, \bar{Y}_{1}) \theta_{1} \Big]   = \mathbb{E}_{0, n}\Big[\partial_{\bar{X}_{1}}h(\bar{X}_{1}, \bar{Y}_{1}) \frac{\theta_{1}}{s_0}\Big]\label{transfer:derivative:x:first:interval}
\end{align}
\noindent while for $1\leq i \leq n$, 
\begin{align}
\partial_{\bar{X}_i} \mathbb{E}_{i, n}\Big[h(\bar{X}_{i+1}, \bar{Y}_{i+1}) \theta_{i+1} \Big]  & = \mathbb{E}_{i, n}\Big[\partial_{\bar{X}_{i+1}}h(\bar{X}_{i+1}, \bar{Y}_{i+1}) \theta_{i+1}\Big]. \label{transfer:derivative:x:next:interval}
\end{align}

Similarly, the following transfer of derivative formulae hold: for any $0\leq i\leq n-1$
\begin{align}
\partial_{\bar{Y}_i} \mathbb{E}_{i, n}\Big[h(\bar{X}_{i+1}, \bar{Y}_{i+1}) \theta_{i+1} \Big]  & = \mathbb{E}_{i, n}\Big[\partial_{\bar{Y}_{i+1}} h(\bar{X}_{i+1}, \bar{Y}_{i+1}) {\rtheta}^{e, Y}_{i+1}\Big] +  \mathbb{E}_{i, n}\Big[\partial_{\bar{X}_{i+1}} h(\bar{X}_{i+1}, \bar{Y}_{i+1}) {\rtheta}^{e, X}_{i+1}\Big] \nonumber \\
& \quad + \mathbb{E}_{i,n}\Big[h(\bar{X}_{i+1}, \bar{Y}_{i+1}) \rtheta^{c}_{i+1}\Big] \label{transfer:derivative:y:next:interval}
\end{align}
\noindent with 
\begin{align*}
\rtheta^{e, Y}_{i+1} & = (f(\zeta_{i+1}-\zeta_{i}))^{-1}\Big[ \mathcal{I}^{(1,1)}_{i+1}(d^{i+1}_S) + \mathcal{I}^{(2,2)}_{i+1}(d^{i+1}_Y) + \mathcal{I}^{(1)}_{i+1}(e^{Y, i+1}_S) + \mathcal{I}^{(2)}_{i+1}(e^{Y, i+1}_Y) + \mathcal{I}^{(1,2)}_{i+1}(d^{i+1}_{Y,S}) \Big] ,\\
\rtheta^{e, X}_{i+1} & = (f(\zeta_{i+1}-\zeta_{i}))^{-1} \mathcal{I}^{(1)}_{i+1}(e^{X, i+1}_S) ,\\
\rtheta^{c}_{i+1} & = \mathcal{I}^{(1)}_{i+1}\Big(\partial_{\bar{Y}_i} \bar{X}_{i+1} \theta_{i+1} - \rtheta^{e, X}_{i+1}\Big) + \partial_{\bar{Y}_i} \theta_{i+1} \\
& +  \mathcal{I}^{(2)}_{i+1}\Big(m'_i \theta_{i+1} - \rtheta^{e, Y}_{i+1} + \Big(\sigma'_{Y, i} \Big(\rho_i Z^{1}_{i+1}+ \sqrt{1-\rho_i^2}Z^{2}_{i+1}\Big) + \sigma_{Y,i} \frac{\rho'_i}{\sqrt{1-\rho^2_i}}\Big(\sqrt{1-\rho_i^2} Z^{1}_{i+1} - \rho_i Z^{2}_{i+1}\Big)\Big)\theta_{i+1}\Big),\\ 
d^{i+1}_S & = m'_i c^{i+1}_S,\\
d^{i+1}_Y & = m'_i c^{i+1}_Y, \\
d^{i+1}_{Y, S} & = m'_i c^{i+1}_{Y, S}, \\
e^{Y, i+1}_S & = -m'_i c^{i+1}_S + \partial_{\bar{Y}_i} c^{i+1}_{Y,S}, \\
e^{Y, i+1}_Y & = m'_i b^{i+1}_Y + \partial_{\bar{Y}_i} c^{i+1}_Y,\\
e^{X, i+1}_S & =\partial_{\bar{Y}_i} c^{i+1}_S .
\end{align*}

For $i=n$, one also has
\begin{align}
\partial_{\bar{Y}_n} \mathbb{E}_{n, n}\Big[h(\bar{X}_{n+1}, \bar{Y}_{n+1}) \theta_{n+1} \Big]  & = \mathbb{E}_{n, n}\Big[\partial_{\bar{Y}_{n+1}} h(\bar{X}_{n+1}, \bar{Y}_{n+1}) {\rtheta}^{e, Y}_{n+1}\Big] + \mathbb{E}_{n, n}\Big[\partial_{\bar{X}_{n+1}} h(\bar{X}_{n+1}, \bar{Y}_{n+1}) {\rtheta}^{e, X}_{n+1}\Big]\nonumber \\
& \quad + \mathbb{E}_{n,n}\Big[h(\bar{X}_{n+1}, \bar{Y}_{n+1}) \rtheta^{c}_{n+1}\Big] \label{transfer:derivative:y:last:interval}
\end{align}
\noindent with
\begin{align*}
{\rtheta}^{e, Y}_{n+1} & = (1-F(T-\zeta_n))^{-1} \Big(m'_{n} + \sigma'_{Y,n}  \Big(\rho_n Z^{1}_{n+1}+ \sqrt{1-\rho_n^2} Z^{2}_{n+1}\Big) + \sigma_{Y, n} \frac{\rho'_n}{\sqrt{1-\rho^2_n}} \Big(\sqrt{1-\rho^2_n} Z^{1}_{n+1} -\rho_n Z^{2}_{n+1} \Big)\Big),\\
{\rtheta}^{e, X}_{n+1} & = (1-F(T-\zeta_n))^{-1} \Big(-\frac12 a'_{S,n} + \sigma'_{S,n}Z^{1}_{n+1} \Big),
\end{align*}
\noindent and we set $ \rtheta^{c}_{n+1} = 0$ for notational convenience.

 Finally, the weight sequences $(\rtheta^{e, Y}_{i})_{1\leq i \leq n+1}$, $(\rtheta^{e, X}_{i})_{1\leq i \leq n+1}$ and $(\rtheta^{c}_{i})_{1\leq i \leq n+1}$ defined above satisfy 
 $$
 f(\zeta_{i}-\zeta_{i-1}) \rtheta^{e, Y}_{i}, \,  f(\zeta_{i}-\zeta_{i-1})\rtheta^{c}_i \in \mathbb{M}_{i-1, n}(\bar{X}, \bar{Y}, -1/2),  \quad f(\zeta_{i}-\zeta_{i-1})\rtheta^{e, X}_i  \in \mathbb{M}_{i-1, n}(\bar{X}, \bar{Y}, 0), \, \quad i \in \left\{1, \cdots, n\right\}
 $$
 \noindent and $(1-F(T-\zeta_n)) \rtheta^{e, Y}_{n+1}\in \mathbb{M}_{n, n}(\bar{X}, \bar{Y}, 0)$, $(1-F(T-\zeta_n)) \rtheta^{e, X}_{n+1}, \in \mathbb{M}_{n, n}(\bar{X}, \bar{Y}, 1/2)$.
\end{lem}

The proof of Lemma \ref{lem:transfer:derivative} is postponed to Appendix \ref{proof:lem:transfer:derivative}. The transfer of derivative procedure starts on the first time interval $[0,\zeta_{1}]$ according to formulae \eqref{transfer:derivative:x:first:interval} and \eqref{transfer:derivative:y:next:interval} (for $i=0$). It expresses the fact that the flow derivatives $\partial_{s_0}$ and $\partial_{y_0}$ of the conditional expectations on the left-hand side of the equations are transferred to derivative operators $\partial_{\bar{X}_1}$ and $\partial_{\bar{Y}_1}$ on the test function $h$ appearing on the right-hand side. Remark that the first derivatives of $h$ have been written ubiquitously as $\partial_{\bar{X}_{i+1}}h(\bar{X}_{i+1}, \bar{Y}_{i+1})$ and $\partial_{\bar{Y}_{i+1}} h(\bar{X}_{i+1}, \bar{Y}_{i+1})$.

Then, by the Markov property satisfied by the process $(\bar{X}, \bar{Y})$, the function $h$ appearing inside the (conditional) expectations on the right-hand side of \eqref{transfer:derivative:x:first:interval} and \eqref{transfer:derivative:y:next:interval} (for $i=0$) will be given by the conditional expectation appearing on the left-hand side of the same equations but for $i=1$. The transfer of derivative formulae for the following time intervals are obtained by induction using \eqref{transfer:derivative:x:next:interval} and \eqref{transfer:derivative:y:next:interval} up to the last time interval.  Doing so, we obtain various transfer of derivative formulae by transferring successively the derivative operators through all intervals forward in time.

\subsection{The integration by parts formulae}

We first define the weights that will be used in our IBP formulae. For an integer $n$, on the set $\left\{N_T=n\right\}$, for any $k \in \left\{ 1, \cdots, n+1\right\}$ and any $j \in \left\{ 1, \cdots,  k\right\}$, we define
\begin{align*}
\rtheta^{\mathcal{I}^{(1),n+1}_k} & :=  \prod_{i=k+1}^{n+1} \theta_i \times \mathcal{I}^{(1)}_{k}(\theta_k)  \times \prod_{i=1}^{k-1} \theta_i,\\
\rtheta^{C^{n+1}_j} & :=  \prod_{i=j+1}^{n+1} \theta_i \times \rtheta_{j}^c \times \prod_{i=1}^{{j}-1} \rtheta^{e, Y}_i,  \\
\rtheta^{\mathcal{I}^{(2),n+1}_k} &:= \prod_{i=k+1}^{n+1} \theta_{i} \times \mathcal{I}^{(2)}_{k}(\rtheta^{e, Y}_k) \times \prod_{i=1}^{k-1} \rtheta^{e, Y}_{i},\\
\rtheta^{\mathcal{I}^{(1), n+1}_k}_j &:=  \prod_{i=k+1}^{n+1} \theta_i \times \mathcal{I}^{(1)}_{k}(\theta_k) \times \prod_{i=j+1}^{k-1} \theta_i \times \rtheta_{j}^{e, X}  \times \prod_{i=1}^{j-1}   \rtheta^{e, Y}_i, \quad  j=1, \cdots, k-1, \\
\rtheta^{\mathcal{I}^{(1), n+1}_k}_k &:= \prod_{i=k+1}^{n+1} \theta_i \times   \mathcal{I}^{(1)}_{k}(\rtheta_{k}^{e, X})  \times \prod_{i=1}^{k-1}   \rtheta^{e, Y}_i.
\end{align*}

With the above definitions at hand, we are now able to state our IBP formulae.
\begin{theorem}\label{thm:ibp:formulae} Let $T>0$.  Under assumptions \A{AR} and \A{ND}, the law of the couple $(X_T, Y_T)$, given by the unique solution to the SDE \eqref{new:sde:to:approximate} at time $T$ starting from $(x_0=\ln(s_0), y_0)$ at time $0$, satisfies the following Bismut-Elworthy-Li IBP formulae: there exists some positive constant $c:=c(T, b_Y, \kappa)$ such that for any $0\leq \gamma< c^{-1}$ and any  $h \in \mathcal{B}_{\gamma}(\mathbb{R}^2)$, for any $(s_0, y_0) \in (0, \infty) \times \mathbb{R}$, it holds
\begin{align}\label{IBP:s0}
s_0 T \partial_{s_0} \mathbb{E}\Big[h(X_T, Y_T)\Big] = \mathbb{E}\Big[h(\bar{X}_{N_T+1}, \bar{Y}_{N_T+1}) \sum_{k=1}^{N_T+1} (\zeta_{k}-\zeta_{k-1}) \rtheta^{\mathcal{I}^{(1), N_T+1}_k}  \Big]
\end{align}
\noindent and
\begin{align}\label{IBP:y0}
T \partial_{y_0} \mathbb{E}\Big[h(X_T, Y_T)\Big] & = \E\left[h(\bar{X}_{N_T+1}, \bar{Y}_{N_T+1}) \sum_{k = 1}^{N_T+1} (\zeta_{k}-\zeta_{k-1}) \Big( \rtheta^{\mathcal{I}^{(2),N_T+1}_k} + \sum_{j=1}^{k} \rtheta^{C^{N_T+1}_j}+\rtheta^{\mathcal{I}^{(1), N_T+1}_k}_j \Big) \right].
\end{align}

Moreover, if $N$ is a renewal process with $Beta(\alpha, 1)$ jump times, then, for any $p\geq1$ satisfying $p(\frac12-\alpha) \leq 1-\alpha$, for any $\gamma$ such that $0\leq p \gamma < c^{-1}$ and any $h \in \mathcal{B}_\gamma(\mathbb{R}^2)$, the random variables appearing inside the expectation in the right-hand side of \eqref{IBP:s0} and \eqref{IBP:y0} admit a finite $L^{p}(\mathbb{P})$-moment. In particular, if $\alpha=1/2$ then for any $p\geq1$, for any $h\in \mathcal{B}_{\gamma}(\mathbb{R}^2)$ with $0\leq p \gamma < c^{-1}$, the $L^{p}(\mathbb{P})$-moment is finite.
\end{theorem}

\begin{proof}
We only prove the IBP formula \eqref{IBP:y0}. The proof of \eqref{IBP:s0} follows by completely analogous (and actually more simple) arguments and is thus omitted. 

\bigskip

\noindent \emph{Step 1: proof of the IBP formula \eqref{IBP:y0} for $h \in \mathcal{C}^{1}_b(\mathbb{R}^2)$.}\\
Let $h\in \mathcal{C}^1_b(\mathbb{R}^2)$. From Theorem \ref{theorem:probabilistic:formulation} and Fubini's theorem, we write  
\begin{align}
	\E[h(X_{T}, Y_T)]  & = \sum_{n\geq0} \E\Big[\E\Big[ h(\bar{X}_{n+1}, \bar{Y}_{n+1})  \prod_{i=1}^{n+1}  \theta_{i} \vert \tau^{n+1} \Big] \,\I_{\{N_T = n\}}\Big] \label{decomposition:number:jump:law:XT:YT}
\end{align}
\noindent where we used the fact that $\left\{ N_T= n \right\} = \left\{ \tau_{n+1} >T\right\} \cap \left\{ \tau_n \leq T\right\}$.
 In most of the arguments below, we will work on the set $\{N_T=n\} $. In order to perform our induction argument forward in time through the Markov chain structure, we define for $ k\in \left\{0, \cdots, n+1\right\}$ the functions 
\begin{align*}
	  H_{k}(\bar X_k, \bar Y_k)
	:= \E_{k,n}\Big[h(\bar{X}_{n+1}, \bar{Y}_{n+1})\prod_{i=k+1}^{n+1}  \theta_{i}\Big] = \E\Big[h(\bar{X}_{n+1}, \bar{Y}_{n+1})\prod_{i=k+1}^{n+1}    \theta_{i}| \bar{X}_k, \bar{Y}_k, \tau^{n+1}, N_T=n\Big].
\end{align*}
\noindent We also let $ H_{n+1}(\bar X_{n+1}, \bar{Y}_{n+1}) := h(\bar{X}_{n+1}, \bar{Y}_{n+1})$. Note that we omit the dependence w.r.t the sequence $\tau^{n+1}$ in the definition of the (random) maps $(H_k)_{0\leq k \leq n+1}$. From the above definition and using \A{ND}, \A{AR}, it follows that the map $H_k$ belongs to $\mathcal{C}^{1}_p(\mathbb{R}^2)$ $a.s.$ for any $0\leq k \leq n+1$. Moreover, from the tower property of conditional expectation the following relation is satisfied for any $ k\in \left\{ 0, \cdots, n\right\}$
	\begin{align}
	 H_k(\bar X_k, \bar Y_k)= \mathbb{E}_{k,n}[  H_{k+1}(\bar{X}_{k+1}, \bar {Y}_{k+1})  \theta_{k+1} ].\label{recur2}
	\end{align}

Now, using first the Lebesgue differentiation theorem and then iterating the transfer of derivative formula \eqref{transfer:derivative:y:next:interval} in Lemma \ref{lem:transfer:derivative}, we obtain\footnote{As before, we use the convention $\sum_{\emptyset} \cdots = 0$, $\prod_{\emptyset} \cdots =1$.} for any $k\in \left\{1, \cdots, n\right\}$,  
	\begin{align}
	\label{transfer:derivative:k:times}
\partial_{y_0}\E\Big[  h(\bar{X}_{n+1}, \bar Y_{n+1}) \prod_{i=1}^{n+1}  \theta_{i} \Big|\, \tau^{n+1}\Big]  &= \partial_{y_0} \E\Big[ \E_{0,n}\Big[H_1(\bar{X}_1, \bar{Y}_1) \theta_1 \Big] \Big| \tau^{n+1}\Big] \nonumber\\
& = \E\Big[  \partial_{y_0}\E_{0,n}\Big[H_1(\bar{X}_1, \bar{Y}_1) \theta_1 \Big] \Big| \tau^{n+1}\Big] \nonumber\\
	&= \E\Big[\mathcal{D}^{(2)}_k H_{k}(\bar{X}_{k}, \bar Y_k) \prod_{i=1}^{k} \rtheta^{e, Y}_{i} \Big|\, \tau^{n+1}\Big] +\sum_{j=1}^{k}\E\Big[H_{j}(\bar{X}_{j}, \bar Y_j) \rtheta_{j}^c \prod_{i=1}^{{j}-1}   \rtheta^{e, Y}_i \Big|\, \tau^{n+1}\Big]  \nonumber \\
	&  \quad + \sum_{j=1}^{k}\E\Big[\mathcal{D}^{(1)}_{j}H_{j}(\bar{X}_{j}, \bar Y_j) \rtheta_{j}^{e, X} \prod_{i=1}^{{j}-1}   \rtheta^{e, Y}_i \Big|\, \tau^{n+1}\Big].
	\end{align}

To further simplify the first term appearing on the right-hand side of \eqref{transfer:derivative:k:times}, we use the tower property of conditional expectation (w.r.t $\E_{k-1,n}[.]$) and the integration by parts formula \eqref{duality:formula}. For any $k \in \left\{ 1, \cdots, n\right\}$, we obtain 
\begin{align*}
	\E\Big[\mathcal{D}^{(2)}_k  H_{k}(\bar{X}_{k}, \bar{Y}_k) \rtheta^{e, Y}_{k} \,\Big| \mathcal{G}_{{k-1}},\tau^{n+1}\Big] & = \E\Big[ H_{k}(\bar{X}_{k}, \bar{Y}_k)\mathcal{I}^{(2)}_{k}(\rtheta^{e, Y}_{k}) \,\Big| \mathcal{G}_{{k-1}},\tau^{n+1}\Big].
\end{align*}

We also simplify the third term appearing on the right-hand side of \eqref{transfer:derivative:k:times}, by using the transfer of derivatives formula \eqref{transfer:derivative:x:next:interval} up to the time interval $[\zeta_{k-1}, \zeta_k]$. For any $j \in \left\{1, \cdots, k\right\}$, it holds
$$
\E\Big[\mathcal{D}^{(1)}_{j}H_{j}(\bar{X}_{j}, \bar Y_j) \rtheta_{j}^{e, X} \prod_{i=1}^{{j}-1}   \rtheta^{e, Y}_i \Big|\, \tau^{n+1}\Big] = \E\Big[\mathcal{D}^{(1)}_{k}H_{k}(\bar{X}_{k}, \bar{Y}_k) \prod_{i=j+1}^{k} \theta_i \rtheta_{j}^{e, X} \prod_{i=1}^{j-1}   \rtheta^{e, Y}_i \Big|\, \tau^{n+1}\Big] ,
$$

\noindent so that, if $j \in \left\{1, \cdots, k-1\right\}$, taking conditional expectation (using again $\E_{k-1,n}[.]$) and then performing an IBP formula on the last time interval $[\zeta_{k-1}, \zeta_k]$ yield 
$$
\E\Big[\mathcal{D}^{(1)}_{k}H_{k}(\bar{X}_{k}, \bar{Y}_k) \prod_{i=j+1}^{k} \theta_i \rtheta_{j}^{e, X} \prod_{i=1}^{j-1}   \rtheta^{e, Y}_i \Big|\, \tau^{n+1}\Big] = \E\Big[ H_{k}(\bar{X}_{k}, \bar{Y}_k) \mathcal{I}^{(1)}_{k}(\theta_k) \prod_{i=j+1}^{k-1} \theta_i \rtheta_{j}^{e, X} \prod_{i=1}^{j-1}   \rtheta^{e, Y}_i \Big|\, \tau^{n+1}\Big],
$$
\noindent while if $j=k$, we obtain
$$
\E\Big[\mathcal{D}^{(1)}_{k}H_{k}(\bar{X}_{k}, \bar{Y}_k) \prod_{i=j+1}^{k} \theta_i \rtheta_{j}^{e, X} \prod_{i=1}^{j-1}   \rtheta^{e, Y}_i \Big|\, \tau^{n+1}\Big] = \E\Big[ H_{k}(\bar{X}_{k}, \bar{Y}_k) \mathcal{I}^{(1)}_{k}(   \rtheta_{k}^{e, X}  )\prod_{i=1}^{k-1}   \rtheta^{e, Y}_i \Big|\, \tau^{n+1}\Big].
$$

Coming back to \eqref{transfer:derivative:k:times} and using the definition of the maps $(H_k)_{0\leq k \leq n+1}$, we thus deduce
\begin{align}
\label{transfer:derivative:k:times:second:formula}
\partial_{y_0}\E\Big[  h(\bar{X}_{n+1}, \bar Y_{n+1}) \prod_{i=1}^{n+1}  \theta_{i} \Big|\, \tau^{n+1}\Big] 
	&= \E\Big[ H_{k}(\bar{X}_{k}, \bar Y_k) \mathcal{I}^{(2)}_{k}(\rtheta^{e, Y}_{k}) \times \prod_{i=1}^{k-1} \rtheta^{e, Y}_{i} \Big|\, \tau^{n+1}\Big] +\sum_{j=1}^{k}\E\Big[H_{j}(\bar{X}_{j}, \bar Y_j) \rtheta_{j}^c\times  \prod_{i=1}^{{j}-1}   \rtheta^{e, Y}_i \Big|\, \tau^{n+1}\Big]  \nonumber \\
	& \quad + \sum_{j=1}^{k - 1}\E\Big[ H_{k}(\bar{X}_{k}, \bar{Y}_k) \mathcal{I}^{(1)}_{k}(\theta_k) \times \prod_{i=j+1}^{k-1} \theta_i  \times \rtheta_{j}^{e, X} \prod_{i=1}^{j-1}   \rtheta^{e, Y}_i \Big|\, \tau^{n+1} \Big] \nonumber \\
	& \quad + \E\Big[ H_{k}(\bar{X}_{k}, \bar{Y}_k) \mathcal{I}^{(1)}_{k}(   \rtheta_{k}^{e, X}  )\prod_{i=1}^{k-1}   \rtheta^{e, Y}_i \Big|\, \tau^{n+1}\Big] \nonumber \\
	& = \E\Big[ h(\bar{X}_{n+1}, \bar{Y}_{n+1}) \prod_{i=k+1}^{n+1} \theta_i \times  \mathcal{I}^{(2)}_{k}(\rtheta^{e, Y}_{k}) \times \prod_{i=1}^{k-1} \rtheta^{e, Y}_{i} \Big|\, \tau^{n+1}\Big] \nonumber \\
	& \quad  +\sum_{j=1}^{k}\E\Big[h(\bar{X}_{n+1}, \bar{Y}_{n+1})\prod_{i=j+1}^{n+1} \theta_i \times  \rtheta_{j}^c \times \prod_{i=1}^{{j}-1}   \rtheta^{e, Y}_i \Big|\, \tau^{n+1}\Big]  \nonumber \\
	& \quad + \sum_{j=1}^{k-1}\E\Big[ h(\bar{X}_{n+1}, \bar{Y}_{n+1}) \prod_{i=k+1}^{n+1} \theta_i \times \mathcal{I}^{(1)}_{k}(\theta_k) \times\prod_{i=j+1}^{k-1} \theta_i \rtheta_{j}^{e, X} \prod_{i=1}^{j-1}   \rtheta^{e, Y}_i \Big|\, \tau^{n+1}\Big]  \nonumber \\
	& \quad + \E\Big[  h(\bar{X}_{n+1}, \bar{Y}_{n+1}) \prod_{i=k+1}^{n+1} \theta_i \times  \mathcal{I}^{(1)}_{k}(   \rtheta_{k}^{e, X}  )\prod_{i=1}^{k-1}   \rtheta^{e, Y}_i \Big|\, \tau^{n+1}\Big]. 
\end{align}

	In the case $k = n+1$, using the transfer of derivative formula \eqref{transfer:derivative:y:last:interval} of Lemma \ref{lem:transfer:derivative} on the last time interval and then performing the IBP formula \eqref{duality:formula}, we obtain the representation
\begin{align}
	 \partial_{y_0} \E\Big[h(\bar{X}_{n+1}, \bar{Y}_{n+1}) \prod_{i=1}^{n+1} \theta_{i} \Big|\, \tau^{n+1}\Big] &=  \E\Big[\mathcal{D}^{(2)}_{n+1}h(\bar{X}_{n+1}, \bar{Y}_{n+1})\prod_{i=1}^{n+1}  \rtheta^{e, Y}_i \Big|\, \tau^{n+1}\Big] + \sum_{j=1}^{n+1}\E\Big[ H_{j}(\bar{X}_{j}, \bar{Y}_j) \rtheta_{j}^c \prod_{i=1}^{{j}-1} \rtheta^{e, Y}_i \Big|\, \tau^{n+1}\Big] \nonumber\\
	& \qquad + \sum_{j=1}^{n+1}\E\Big[\mathcal{D}^{(1)}_{j}H_{j}(\bar{X}_{j}, \bar Y_j) \rtheta_{j}^{e, X} \prod_{i=1}^{{j}-1}   \rtheta^{e, Y}_i \Big|\, \tau^{n+1}\Big]\nonumber \\
	& =  \E\Big[h(\bar{X}_{n+1}, \bar{Y}_{n+1}) \mathcal{I}^{(2)}_{n+1}(\rtheta^{e, Y}_{n+1} ) \prod_{i=1}^{n}  \rtheta^{e, Y}_i \Big|\, \tau^{n+1}\Big] + \sum_{j=1}^{n+1}\E\Big[ H_{j}(\bar{X}_{j}, \bar{Y}_j) \rtheta_{j}^c \prod_{i=1}^{{j}-1} \rtheta^{e, Y}_i \Big|\, \tau^{n+1}\Big] \nonumber\\
	& \qquad + \sum_{j=1}^{n+1}\E\Big[\mathcal{D}^{(1)}_{n+1} h(\bar{X}_{n+1}, \bar Y_{n+1}) \prod_{i=j+1}^{n+1} \theta_{i} \rtheta_{j}^{e, X} \prod_{i=1}^{{j}-1}   \rtheta^{e, Y}_i \Big|\, \tau^{n+1}\Big] \nonumber\\
	& =  \E\Big[h(\bar{X}_{n+1}, \bar{Y}_{n+1}) \mathcal{I}^{(2)}_{n+1}(\rtheta^{e, Y}_{n+1} ) \times \prod_{i=1}^{n}  \rtheta^{e, Y}_i \Big|\, \tau^{n+1}\Big] \nonumber \\
	& \quad + \sum_{j=1}^{n+1}\E\Big[ h(\bar{X}_{n+1}, \bar{Y}_{n+1}) \prod_{i= j+1}^{n+1}\theta_i \times \rtheta_{j}^c \times \prod_{i=1}^{{j}-1} \rtheta^{e, Y}_i \Big|\, \tau^{n+1}\Big] \nonumber\\
	& \quad + \sum_{j=1}^{n}\E\Big[ h(\bar{X}_{n+1}, \bar Y_{n+1}) \mathcal{I}^{(1)}_{n+1}(\theta_{n+1}) \times \prod_{i=j+1}^{n} \theta_{i} \times \rtheta_{j}^{e, X} \times \prod_{i=1}^{{j}-1}   \rtheta^{e, Y}_i \Big|\, \tau^{n+1}\Big]  \nonumber \\
	& \quad + \E\Big[ h(\bar{X}_{n+1}, \bar Y_{n+1}) \mathcal{I}^{(1)}_{n+1}(\rtheta_{n+1}^{e, X}) \times \prod_{i=1}^{n}   \rtheta^{e, Y}_i \Big|\, \tau^{n+1}\Big] \label{ibp:vol:process:last:time:interval}
\end{align}	

\noindent where, for the last term appearing in the right-hand side of the above identities, we employed the transfer of derivative formula \eqref{transfer:derivative:x:next:interval} up to the last time interval and then performed an IBP formula. 
			
Now, the key point in order to establish the  IBP formula \eqref{IBP:y0} is to combine in a suitable way the identities \eqref{transfer:derivative:k:times:second:formula} and \eqref{ibp:vol:process:last:time:interval}. For each $k\in \left\{ 0, \cdots, n\right\}$, we multiply the above formulae by the length of the interval on which the local IBP formula is performed, namely we multiply by $\zeta_{k}- \zeta_{k-1}$ both sides of \eqref{transfer:derivative:k:times:second:formula}, $k=1, \cdots, n-1$, and we multiply by $T-\zeta_n$ both sides of \eqref{ibp:vol:process:last:time:interval}. We then sum them over all $k$. Recalling that $\sum_{k=1}^{n+1} \zeta_{k}-\zeta_{k-1}= T -\zeta_0 = T$, we deduce 
\begin{align*}
&  T\partial_{y_0}\E\Big[h(\bar{X}_{n+1}, \bar{Y}_{n+1}) \prod_{i=1}^{n+1} \theta_{i} \Big|\, \tau^{n+1}\Big]\\
	=&\sum_{k = 1}^{n+1} (\zeta_{k}-\zeta_{k-1})\E\Big[h(\bar{X}_{n+1}, \bar{Y}_{n+1}) \prod_{i=k+1}^{n+1} \theta_{i} \times \mathcal{I}^{(2)}_{k}(\rtheta^{e, Y}_k) \times \prod_{i=1}^{k-1} \rtheta^{e, Y}_{i} \Big|\, \tau^{n+1}\Big]\\
	 & + \sum_{k = 1}^{n+1}(\zeta_{k}-\zeta_{k-1})\sum_{j=1}^{k}\E\Big[h(\bar{X}_{n+1}, \bar{Y}_{n+1}) \prod_{i=j+1}^{n+1} \theta_i \times \rtheta_{j}^c \times \prod_{i=1}^{{j}-1} \rtheta^{e, Y}_i \Big|\, \tau^{n+1}\Big]\\
	 & + \sum_{k = 1}^{n+1}(\zeta_{k}-\zeta_{k-1}) \Big( \sum_{j=1}^{k-1}\E\Big[h(\bar{X}_{n+1}, \bar{Y}_{n+1}) \prod_{i=k+1}^{n+1} \theta_i \times \mathcal{I}^{(1)}_{k}(\theta_k) \times \prod_{i=j+1}^{k-1} \theta_i \times \rtheta_{j}^{e, X}  \times \prod_{i=1}^{j-1}   \rtheta^{e, Y}_i\Big|\, \tau^{n+1}\Big] \\
	 & +  \E\Big[ h(\bar{X}_{n+1}, \bar Y_{n+1}) \prod_{i=k+1}^{n+1} \theta_i \times \mathcal{I}^{(1)}_{k}(\rtheta_{k}^{e, X}) \times \prod_{i=1}^{n}   \rtheta^{e, Y}_i \Big|\, \tau^{n+1}\Big] \Big) \\
	 & =  \E\Big[h(\bar{X}_{n+1}, \bar{Y}_{n+1}) \sum_{k = 1}^{n+1} (\zeta_{k}-\zeta_{k-1}) \Big( \rtheta^{\mathcal{I}^{(2),n+1}_k} + \sum_{j=1}^{k} \rtheta^{C^{n+1}_j}+\rtheta^{\mathcal{I}^{(1), n+1}_k}_j\Big) \Big| \tau^{n+1} \Big].
\end{align*}

We now provide a sharp upper-estimate for the above quantity. From Lemma \ref{estimate:theta:time:loca:density} and Lemma \ref{lem:transfer:derivative}, it follows that $f(\zeta_{i}-\zeta_{i-1}) \theta_i, \,f(\zeta_{i}-\zeta_{i-1})  \rtheta^{e, Y}_{i}, \, f(\zeta_{i}-\zeta_{i-1}) \rtheta^{c}_i  \in \mathbb{M}_{i-1, n}(\bar{X}, \bar{Y}, -1/2)$ and $f(\zeta_{i}-\zeta_{i-1}) \rtheta^{e, X}_i \in  \mathbb{M}_{i-1, n}(\bar{X}, \bar{Y}, 0)$ for any $i\in \left\{1, \cdots, n\right\}$. Moreover, from the very definition of the weights $\theta_i$, $ \rtheta^{e, X}_i$ and $\rtheta^{e, Y}_i$, after some simple but cumbersome computations that we omit, one has $f(\zeta_{i}-\zeta_{i-1})  \mathcal{D}^{(1)}_i(\theta_i), \, f(\zeta_{i}-\zeta_{i-1}) \mathcal{D}^{(2)}_i(\rtheta^{e, Y}_i) \in \mathbb{M}_{i-1, n}(\bar{X}, \bar{Y}, -1)$ and $f(\zeta_{i}-\zeta_{i-1})  \mathcal{D}^{(1)}_i( \rtheta^{e, X}_i ) \in \mathbb{M}_{i-1, n}(\bar{X}, \bar{Y}, -1/2)$ so that from Lemma \ref{lem:time:degeneracy:property} we conclude $f(\zeta_{i}-\zeta_{i-1})  (\zeta_i -\zeta_{i-1}) \mathcal{I}^{(1)}_i(\theta_i) \in \mathbb{M}_{i-1, n}(\bar{X}, \bar{Y}, 0)$, $f(\zeta_{i}-\zeta_{i-1})  (\zeta_i -\zeta_{i-1}) \mathcal{I}^{(2)}_i(\rtheta^{e, Y}_i) \in \mathbb{M}_{i-1, n}(\bar{X}, \bar{Y}, 0)$ and $f(\zeta_{i}-\zeta_{i-1})  (\zeta_i -\zeta_{i-1}) \mathcal{I}^{(1)}_i(\rtheta^{e, X}_i) \in \mathbb{M}_{i-1, n}(\bar{X}, \bar{Y}, 1/2) $. Hence, from the boundedness of $h$, the tower property of conditional expectation and \eqref{conditional:lp:moment:space:M:I:N}, it holds
\begin{align*}
\Big| &  (\zeta_{k}-\zeta_{k-1}) \E\Big[h(\bar{X}_{n+1}, \bar{Y}_{n+1}) \prod_{i=k+1}^{n+1} \theta_{i} \times \mathcal{I}^{(2)}_{k}(\rtheta^{e, Y}_k) \times \prod_{i=1}^{k-1} \rtheta^{e, Y}_{i} \Big|\, \tau^{n+1}\Big] \Big| \\
& \leq C^{n+1} (1-F(T-\zeta_{n}))^{-1} \prod^{n}_{i=k+1} (f(\zeta_{i}-\zeta_{i-1}))^{-1} (\zeta_{i}-\zeta_{i-1})^{-\frac12} (f(\zeta_{k}-\zeta_{k-1}))^{-1} \prod_{i=1}^{k-1}  (f(\zeta_{i}-\zeta_{i-1}))^{-1} (\zeta_{i}-\zeta_{i-1})^{-\frac12}
\end{align*}
 \noindent so that using the identity \eqref{probabilistic:representation:time:integrals}
 \begin{align*}
 \sum_{n\geq0} & \mathbb{E}\Big[\sum_{k=1}^{n+1}\Big| (\zeta_{k}-\zeta_{k-1}) \E\Big[h(\bar{X}_{n+1}, \bar{Y}_{n+1}) \prod_{i=k+1}^{n+1} \theta_{i} \times \mathcal{I}^{(2)}_{k}(\rtheta^{e, Y}_k) \times \prod_{i=1}^{k-1} \rtheta^{e, Y}_{i} \Big|\, \tau^{n+1}\Big] \Big| \I_\seq{N_T=n} \Big] \\
 & \leq \sum_{n\geq0} C^{n+1} \sum_{k=1}^{n+1} \mathbb{E}\Big[ (1-F(T-\zeta_{n}))^{-1} (f(\zeta_{k}-\zeta_{k-1}))^{-1}\prod^{n}_{i =1, i\neq k} (f(\zeta_{i}-\zeta_{i-1}))^{-1} (\zeta_{i}-\zeta_{i-1})^{-\frac12} \I_\seq{N_T=n} \Big] \\
 & \leq \sum_{n\geq0} C^{n+1} \sum_{k=1}^{n+1} \int_{\Delta_{n}(T)}   \prod_{i=1, i\neq k}^{n} (s_{i}-s_{i-1})^{-1/2} \, d\s_{n}\\
 & \leq  \sum_{n\geq0} (n+1) C^{n+1} T^{(n+1)/2} \frac{\Gamma^{n}(1/2)}{\Gamma(1+n/2)} <\infty.
 \end{align*}
 
 From similar arguments that we omit, it follows
 \begin{align*}
 \Big| & (\zeta_{k}-\zeta_{k-1}) \sum_{j=1}^{k}\E\Big[h(\bar{X}_{n+1}, \bar{Y}_{n+1}) \Big(\rtheta^{C^{n+1}_j}+\rtheta^{\mathcal{I}^{(1), n+1}_k}_j\Big) \Big|\, \tau^{n+1}\Big] \Big| \\
 & \leq C^{n+1} (\zeta_{k}-\zeta_{k-1}) \sum_{j=1}^{k} (1-F(T-\zeta_n))^{-1} \prod_{i=1}^{n} (f(\zeta_i-\zeta_{i-1}))^{-1}(\zeta_{i}-\zeta_{i-1})^{-1/2} [ 1 + \I_\seq{i=k} (\zeta_{i}-\zeta_{i-1})^{-1/2}  ].
 \end{align*}
 
 \noindent so that using again the identity \eqref{probabilistic:representation:time:integrals}
 \begin{align*}
  \sum_{n\geq0} &  \mathbb{E}\Big[ \sum_{k=1}^{n+1} \Big|  (\zeta_{k}-\zeta_{k-1}) \sum_{j=1}^{k}\E\Big[h(\bar{X}_{n+1}, \bar{Y}_{n+1}) \Big(\rtheta^{C^{n+1}_j}+\rtheta^{\mathcal{I}^{(1), n+1}_k}_j\Big) \Big|\, \tau^{n+1}\Big] \Big| \I_\seq{N_T=n} \Big] \\
  & \leq \sum_{n\geq0 } C^{n+1}\sum_{k=1}^{n+1}  \E\big[ (\zeta_{k}-\zeta_{k-1}) \sum_{j=1}^{k} (1-F(T-\zeta_n))^{-1} \prod_{i=1}^{n} (f(\zeta_i-\zeta_{i-1}))^{-1}(\zeta_{i}-\zeta_{i-1})^{-1/2} [ 1 + \I_\seq{i=k} (\zeta_{i}-\zeta_{i-1})^{-1/2}  ] \I_\seq{N_T=n}\big]  \\
  & \leq \sum_{n\geq 0}  C^{n+1} (n+1)(n+2) T^{(n+1)/2}\frac{\Gamma^{n}(1/2)}{\Gamma(1+n/2)}< \infty.
 \end{align*}
 
%
 
The preceding estimates combined with \eqref{decomposition:number:jump:law:XT:YT} and the Lebesgue dominated convergence theorem allows to conclude that $y_0\mapsto \E[h(X_{T}, Y_T)]$ is continuously differentiable with
\begin{align*}
	T \partial_{y_0} \E[h(X_{T}, Y_T)]  & = T \partial_{y_0} \E\Big[  h(\bar{X}_{N_T+1}, \bar{Y}_{N_T+1}) \prod_{i=1}^{N_T+1} \theta_{i} \Big] \\
	& = \sum_{n\geq0} \E\Big[T \partial_{y_0}\E\Big[h(\bar{X}_{n+1}, \bar{Y}_{n+1})  \prod_{i=1}^{n+1} \theta_{i} \Big| \tau^{n+1} \Big] \,\I_{\{N_T = n\}} \Big]\\
	& =  \sum_{n\geq0} \E\Big[ \E\Big[h(\bar{X}_{n+1}, \bar{Y}_{n+1}) \sum_{k = 1}^{n+1} (\zeta_{k}-\zeta_{k-1}) \Big( \rtheta^{\mathcal{I}^{(2),n+1}_k} + \sum_{j=1}^{k} \rtheta^{C^{n+1}_j}+\rtheta^{\mathcal{I}^{(1), n+1}_k}_j\Big) \Big| \tau^{n+1} \Big] \,\I_{\{N_T = n\}} \Big]\\
	& = \E\Big[h(\bar{X}_{N_T+1}, \bar{Y}_{N_T+1}) \sum_{k = 1}^{N_T+1} (\zeta_{k}-\zeta_{k-1}) \Big( \rtheta^{\mathcal{I}^{(2),N_T+1}_k} + \sum_{j=1}^{k} \rtheta^{C^{n+1}_j}+\rtheta^{\mathcal{I}^{(1), N_T+1}_k}_j\Big) \Big]
\end{align*}

\noindent where we used Fubini's theorem for the last equality. This completes the proof of the IBP formula \eqref{IBP:y0} for $h \in \mathcal{C}^1_{b}(\mathbb{R}^2)$.  

\bigskip

\noindent \emph{Step 2: Extension to $h \in \mathcal{B}_\gamma(\mathbb{R}^2)$ for some positive $\gamma$.}\\

We now extend the two IBP formulae that we have established in the previous step to the case of a test function $h \in \mathcal{B}_{\gamma}(\mathbb{R}^2)$ for some sufficiently small $\gamma>0$. Let us note that under assumption \A{H}, from Kusuoka and Stroock \cite{kusuoka:stroock:mc2}, Corollary (3.25) and  the upper-estimate (3.27) therein, the process $(X_t, Y_t)_{t\geq0}$ admits a smooth transition density $(t, x_0, y_0, x ,y) \mapsto p(t, x_0, y_0, x, y) \in \mathcal{C}^{\infty}((0,\infty) \times \mathbb{R}^2 \times \mathbb{R}^2)$ and for any $h \in \mathcal{C}^1_{b}(\mathbb{R}^2)$, it holds
$$
\partial^{\alpha}_{s_0} \partial^{\beta}_{y_0} \mathbb{E}[h(X_T, Y_T)]  = \int_{\mathbb{R}^2} h(x, y) \, \partial^{\alpha}_{s_0} \partial^{\beta}_{y_0} p(T, x_0, y_0, x, y) \, dx dy
$$ 
\noindent for any $T>0$ and any integers $\alpha$ and $\beta$. 

We then proceed as in step 2 of the proof of Theorem \ref{theorem:probabilistic:formulation}. Namely, we prove that 
\begin{align}
T \partial_{y_0} &\mathbb{E}[h(X_T, Y_T)] \notag  \\
&= \E\Big[  h(\bar{X}_{N_T+1}, \bar{Y}_{N_T+1})  \sum_{k = 1}^{N_T+1} (\zeta_{k}-\zeta_{k-1}) \Big( \rtheta^{\mathcal{I}^{(2),N_T+1}_k} + \sum_{j=1}^{k} \rtheta^{C^{n+1}_j}+\rtheta^{\mathcal{I}^{(1), N_T+1}_k}_j\Big) \Big] \nonumber\\
& = \int_{\mathbb{R}^2} h(x, y) \, \E\Big[\bar{p}(T-\zeta_{N_T}, \bar{X}_{N_T}, \bar{Y}_{N_T}, x, y) \sum_{k = 1}^{N_T+1} (\zeta_{k}-\zeta_{k-1}) \Big( \rtheta^{\mathcal{I}^{(2),N_T+1}_k} + \sum_{j=1}^{k} \rtheta^{C^{n+1}_j}+\rtheta^{\mathcal{I}^{(1), N_T+1}_k}_j\Big) \Big] \, dx dy \label{gradient:heat:kernel:probabilistic:representation}
\end{align}

\noindent for any $h\in \mathcal{C}^1_b(\mathbb{R}^2)$.

Indeed, since $f(\zeta_{i}-\zeta_{i-1}) \theta_i \in \mathbb{M}_{i-1, n}(\bar{X}, \bar{Y}, -1/2)$ and $f(\zeta_{k}-\zeta_{k-1}) \mathcal{I}^{(2)}_{k}(\rtheta^{e, Y}_k) \in \mathbb{M}_{k-1,n}(\bar{X}, \bar{Y}, -1)$,  for some $c:=c(T, b_Y, \kappa)>4\kappa$, it holds 
\begin{align}
 \E\Big[& \bar{p}(T-\zeta_{n}, \bar{X}_{n}, \bar{Y}_{n}, x, y) \sum_{k = 1}^{n+1} (\zeta_{k}-\zeta_{k-1})\Big| \rtheta^{\mathcal{I}^{(2),n+1}_k}\Big| \Big| \tau^{n+1} \Big]  \nonumber \\
 & \leq C^{n+1} \int_{(\mathbb{R}^2)^{n}} \bar{q}_{4\kappa}(T- \zeta_n, x_n, y_n, x, y) \, \sum_{k=1}^{n+1} (\zeta_{k}-\zeta_{k-1}) (1-F(T-\zeta_n))^{-1} \prod_{i=k+1}^{n} (f(\zeta_{i}-\zeta_{i-1}))^{-1} (\zeta_i - \zeta_{i-1})^{-1/2} \nonumber \\
 & \quad \times (f(\zeta_k - \zeta_{k-1}))^{-1} (\zeta_k - \zeta_{k-1})^{-1} \prod_{i=1}^{k-1} (f(\zeta_{i}-\zeta_{i-1}))^{-1} (\zeta_{i}-\zeta_{i-1})^{-1/2} \prod_{i=1}^{n} \bar{q}_{4\kappa}(\zeta_{i}-\zeta_{i-1}, x_{i-1}, y_{i-1}, x_i, y_i) \, d\x_n d\y_n \nonumber \\
 & \leq C^{n+1} \bar{q}_{c}(T, x_0, y_0, x, y) \sum_{k=1}^{n+1}   (1-F(T-\zeta_n))^{-1} \prod_{i=1}^{n}(f(\zeta_{i}-\zeta_{i-1}))^{-1} \prod_{i=1, i \neq k}^{n} (\zeta_i - \zeta_{i-1})^{-1/2} \label{upper:bound:gradient:kernel:first:part}
\end{align}

\noindent where, for the first inequality we used the upper-estimate \eqref{upper:bound:transition:density:approximation:semigroup} and for the last inequality we used Lemma \ref{lemma:semigroup:property}. From similar arguments, one gets
\begin{align}
&  \E\Big[ \bar{p}(T-\zeta_{n}, \bar{X}_{n}, \bar{Y}_{n}, x, y) \sum_{k = 1}^{n+1} (\zeta_{k}-\zeta_{k-1})  \sum_{j=1}^{k} \Big|\rtheta^{C^{n+1}_j} \Big| + \Big|\rtheta^{\mathcal{I}^{(1), N_T+1}_k}_j\Big| \Big| \tau^{n+1} \Big] \nonumber\\
 & \leq C^{n+1} \bar{q}_{c}(T, x_0, y_0, x, y) \sum_{k=1}^{n+1} (\zeta_{k}-\zeta_{k-1}) \sum_{j=1}^{k} (1-F(T-\zeta_n))^{-1} \prod_{i=1}^{n} (f(\zeta_i-\zeta_{i-1}))^{-1}(\zeta_{i}-\zeta_{i-1})^{-1/2} [ 1 + \I_\seq{i=k} (\zeta_{i}-\zeta_{i-1})^{-1/2}  ].\label{upper:bound:gradient:kernel:second:part}
\end{align}
 

Now, from the upper-bounds \eqref{upper:bound:gradient:kernel:first:part} and \eqref{upper:bound:gradient:kernel:second:part} as well as the identity \eqref{probabilistic:representation:time:integrals}, we conclude
\begin{align}
 \sum_{n\geq0} &\E\Big[  \E\Big[ \bar{p}(T-\zeta_{n}, \bar{X}_{n}, \bar{Y}_n, x, y)  \sum_{k = 1}^{n+1} (\zeta_{k}-\zeta_{k-1}) \Big(  \Big|\rtheta^{\mathcal{I}^{(2),n+1}_k} \Big|+ \sum_{j=1}^{k} \Big|\rtheta^{C^{n+1}_j} \Big|+ \Big|\rtheta^{\mathcal{I}^{(1), n+1}_k}_j\Big|\Big) \Big| \tau^{n+1}\Big] \I_\seq{N_T=n} \Big] \nonumber\\
& \leq  \bar{q}_{c}(T, x_0, y_0, x, y) \sum_{n\geq0 } C^{n+1} \E\Big[\sum_{k=1}^{n+1}   (1-F(T-\zeta_n))^{-1} \prod_{i=1}^{n}(f(\zeta_{i}-\zeta_{i-1}))^{-1} \prod_{i=1, i \neq k}^{n} (\zeta_i - \zeta_{i-1})^{-1/2} \nonumber \\
& \quad + \sum_{k=1}^{n+1} (\zeta_{k}-\zeta_{k-1}) \sum_{j=1}^{k} (1-F(T-\zeta_n))^{-1} \prod_{i=1}^{n} (f(\zeta_i-\zeta_{i-1}))^{-1}(\zeta_{i}-\zeta_{i-1})^{-1/2} [ 1 + \I_\seq{i=k} (\zeta_{i}-\zeta_{i-1})^{-1/2}  ]\Big] \nonumber\\
& \leq \bar{q}_{c}(T, x_0, y_0, x, y) \sum_{n\geq0} C^{n+1} [(n+1)+ (n+1)(n+2)/2] T^{(n+1)/2}\frac{\Gamma^{n}(1/2)}{\Gamma(1+n/2)} \nonumber\\
& =  C T^{1/2}\bar{q}_{c}(T, x_0, y_0, x, y). \label{upper:grandient:density}
\end{align}

From the preceding inequality and Fubini's theorem, we thus get
\begin{align}
& \Big| \E\Big[\bar{p}(T-\zeta_{N_T}, \bar{X}_{N_T}, \bar{Y}_{N_T}, x, y) \sum_{k = 1}^{N_T+1} (\zeta_{k}-\zeta_{k-1}) \Big( \rtheta^{\mathcal{I}^{(2),N_T+1}_k} + \sum_{j=1}^{k} \rtheta^{C^{N_T+1}_j}+\rtheta^{\mathcal{I}^{(1), N_T+1}_k}_j\Big) \Big]\Big|\nonumber\\
& \leq C T^{1/2}\bar{q}_{c}(T, x_0, y_0, x, y) \label{upper:bound:gradient:density:prob:representation}
\end{align}
\noindent for some positive constant $C:=C(T)$ such that $T\mapsto C(T)$ is non-decreasing. Applying again Fubini's theorem allows to complete the proof of \eqref{gradient:heat:kernel:probabilistic:representation}. Hence,
\begin{align*}
T\partial_{y_0}& \mathbb{E}[h(X_T, Y_T)]  \\
&  = \int_{\mathbb{R}^2} h(x, y) \ \partial_{y_0} p(T, x_0, y_0, x, y) \, dx dy \\
& =  \int_{\mathbb{R}^2} h(x, y) \, \E\Big[\bar{p}(T-\zeta_{N_T}, \bar{X}_{N_T}, \bar{Y}_{N_T}, x, y) \sum_{k = 1}^{N_T+1} (\zeta_{k}-\zeta_{k-1}) \Big( \rtheta^{\mathcal{I}^{(2),N_T+1}_k} + \sum_{j=1}^{k} \rtheta^{C^{N_T+1}_j}+\rtheta^{\mathcal{I}^{(1), N_T+1}_k}_j\Big) \Big] \, dxdy
\end{align*}
\noindent for any $h\in \mathcal{C}^1_b(\mathbb{R}^2)$. A monotone class argument allows to conclude that the preceding identity is still valid for any bounded and measurable map $h$ defined over $\mathbb{R}^2$ and a standard approximation argument allows to extend it to $h\in \mathcal{B}_\gamma(\mathbb{R}^2)$ for any $0\leq \gamma < c^{-1}$, $c$ being the positive constant appearing in \eqref{upper:bound:gradient:density:prob:representation}. We eventually conclude from the preceding identity, \eqref{upper:grandient:density} combined with Fubini's theorem that
$$
T \partial_{y_0} \mathbb{E}[h(X_T, Y_T)] = \E\Big[ h(\bar{X}_{N_T+1}, \bar{Y}_{N_T+1}) \sum_{k = 1}^{N_T+1} (\zeta_{k}-\zeta_{k-1}) \Big( \rtheta^{\mathcal{I}^{(2),N_T+1}_k} + \sum_{j=1}^{k} \rtheta^{C^{N_T+1}_j}+\rtheta^{\mathcal{I}^{(1), N_T+1}_k}_j\Big) \Big] 
$$
\noindent for any $h\in \mathcal{B}_\gamma(\mathbb{R}^2)$ such that $0\leq \gamma < c^{-1}$.

\bigskip

\noindent \emph{Step 3: $L^{p}(\mathbb{P})$-moments for a renewal process with Beta jump times.}\\

	From the above formula, the proof of the ${L}^p(\mathbb{P})$-moment estimate when $N$ is a renewal process with Beta jump times follows by similar arguments as those employed at step 3 of the proof of Theorem \ref{theorem:probabilistic:formulation}. We omit the remaining technical details.

\end{proof}


\section{Numerical Results}\label{section:numerical:results}

In this section, we provide some numerical results for the unbiased Monte Carlo algorithm that stems from the probabilistic representation formula established in Theorem \ref{theorem:probabilistic:formulation} and the Bismut-Elworthy-Li formulae of Theorem \ref{thm:ibp:formulae} for the couple $(S_T, Y_T)$ that allows to compute the Delta and the Vega related to the option price of the vanilla option with payoff $h(S_T)$. We here consider the unique strong solution associated to the SDE \eqref{stochastic:volatility:model:sde} for three different models corresponding to three different diffusion coefficient function $\sigma_S$ and two different options, namely Call and digital Call options with maturity $T$ and strike $K$, with payoff functions $h(x, y) = (\exp(x)-K)_+$ and $h(x, y) = \textbf{1}_{\left\{ \exp(x) \geq K\right\}}$ respectively. For these three models, the drift function of the volatility process is defined by $b_Y(x) = \lambda_Y( \mu - x)$ and we fix the parameters as follows: $T=0.5$, $r=0.03$,  $K=1.5$, $x_0= \ln(s_0) = 0.4$, $Y_0 = 0.2$, \ $\sigma_Y(.) \equiv \sigma_Y =0.2$, $\lambda_Y = 0.5$, $\mu=0.3$ and $\rho=0.6$. We also consider two type of renewal process $N$: a Poisson process with intensity parameter $\lambda = 0.5$ and a renewal process with $Beta(1-\alpha, 1)$ jump times with parameters $\alpha=0.1$ and $\bar{\tau} = 2$.

\subsection{Black-Scholes Model}
We first consider the simple (toy) example corresponding to the Black-Scholes dynamics
$$
dS_t = r S_t \, dt + \sigma_S S_t \, dW_t, \quad dY_t = b_Y(Y_t) \, dt + \sigma_Y(Y_t) dB_t, \quad d\langle B, W\rangle_t = \rho dt, \, \rho \in (-1,1).
$$

\noindent with constant diffusion coefficient function $\sigma_S(.) \equiv \sigma_S>0$. The law of $(S_T, Y_T)$ can be computed explicitly so that analytical formulas are available for the price, Delta and Vega. Note that the discount factor $e^{-rT}$ has been added in our probabilistic representation formula for comparison purposes. In this example, we importantly remark that the dynamics of the Euler scheme writes 
\begin{equation}\label{euler:scheme:bs:example}
\left\{
\begin{array}{rl}
\bar{X}_{i+1} & = \bar{X}_{i} +  \Big(r   -\frac12 a_{S}\Big) (\zeta_{i+1}-\zeta_i) + \sigma_{S} \sqrt{\zeta_{i+1}-\zeta_i} Z^{1}_{i+1},\\
\bar{Y}_{i+1} & = m_i + \sigma_{Y} \sqrt{\zeta_{i+1}-\zeta_i} \left( \rho Z^1_{i+1} + \sqrt{1-\rho^2} Z^{2}_{i+1}\right),
\end{array}
\right.
\end{equation}
\noindent with $m_{i} = m_{\zeta_{i+1}-\zeta_i}(\bar{Y}_{i}) =  \mu + (\bar{Y}_i - \mu) e^{-\lambda (\zeta_{i+1} - \zeta_i)}$. Also, the weights $(\theta_i)_{1\leq i \leq N_T+1}$ in the probabilistic representation \eqref{probabilistic:formulation:theorem} of Theorem \ref{theorem:probabilistic:formulation} greatly simplifies, namely
\begin{align*}
\theta_{i}  =  (f(\zeta_{i}-\zeta_{i-1}))^{-1}   \mathcal{I}^{(2)}_{i}(b^{i}_Y), \, 1 \leq i \leq N_T, \quad \mbox{ and } \quad \theta_{N_{T}+1} & = (1 - F(T - \zeta_{N_T}))^{-1}.
\end{align*}

We perform $M_1 = 10^7$ Monte Carlo path simulations to approximate the price as well as the two Greeks and compare them with the corresponding values obtained using the standard Monte Carlo method combined with an Euler-Maruyama approximation scheme for the dynamics \eqref{stochastic:volatility:model:sde} with $M_2 = 160 000 $ Monte Carlo simulations paths and mesh size $\delta =T/n$ where $n = 200$. The Delta and Vega are obtained using the Monte Carlo finite difference approach combined with the Euler-Maruyama discretization scheme, that is, denoting by $E^{n}_{M_2}(s_0, y_0)$ the Monte Carlo estimator with Euler-Maruyama scheme, we compute $(E^{n}_{M_2}(s_0+\varepsilon, y_0) - E^{n}_{M_2}(s_0, y_0))/\varepsilon$ and $(E^{n}_{M_2}(s_0, y_0+\varepsilon) - E^{n}_{M_2}(s_0, y_0))/\varepsilon$ respectively with $\varepsilon = 10^{-2}$. The numerical results for the three different quantities are summarized in Table \ref{table:bs:price}, Table \ref{table:bs:delta}, Table \ref{table:bs:vega} respectively. The first column provides the value of the parameter $\sigma_S$. The second column stands for the value of the price, Delta or Vega obtained by the corresponding Black-Scholes formula. The third and fourth columns correspond to the value obtained by the Euler-Maruyama discretization scheme together with its $95\%$ confidence interval. The fifth and sixth (resp. seventh and eighth ) columns provide the estimated value with its $95\%$ confidence interval by our method in the case of Exponential sampling (resp. Beta sampling). 
We observe a good behaviour of the unbiased estimators for all three quantities and for all the values of the parameter $\sigma_S$.
 
 \begin{table}[htp]
	\centering
	\resizebox{\textwidth}{20mm}{
	\begin{tabular}{|c|c|c|c|c|c|c|c|} 
		\hline  
		\multicolumn{1}{|c|}{\multirow{2}*{ $ \sigma_S $ }}  &
		\multicolumn{1}{c|}{\multirow{2}*{  \tabincell{c}{ B-S \\  formula  }  }}  &
		\multicolumn{2}{c|}{ Euler Scheme  } & 
		  \multicolumn{2}{c|}{   Exponential  sampling  }& 
		  \multicolumn{2}{c|}{   Beta  sampling  }  \\
		   \cline{3-8} 
		&  &  Price   &  95\% CI  &    Price   &  95\% CI &    Price   &  95\% CI   \\
		\hline
		0.25 & 0.111804 &  0.111467& [0.110699,  0.112235] & 0.112285 & [0.111781, 0.112789] & 0.112159 & [0.111734, 0.112584]  \\
		\hline
		0.3 & 0.132621 &0.13293  &  [0.132337, 0.133524]&0.133054  & [0.132394, 0.133713]  & 0.132954 &[0.132482, 0.133425]  \\
		\hline
		0.4 &0.174152 & 0.17392 & [0.173103, 0.174737] & 0.175346 &[0.174557, 0.176135]  & 0.174584 &[0.173912, 0.175255] \\
		\hline
		 0.6&0.256572 & 0.258063 & [0.256727, 0.259399] & 0.255934 & [0.254592, 0.257277] & 0.256514 & [0.255419, 0.257608] \\
		\hline
	\end{tabular}
	} 
	\caption{Comparison between the unbiased Monte Carlo estimation and the Monte Carlo Euler-Maruyama scheme for the price of a Call option in the Black-Scholes model for different values of $\sigma_S$.}
	\label{table:bs:price}
\end{table}

 \begin{table}[htp]
	\centering
	\resizebox{\textwidth}{20mm}{
	\begin{tabular}{|c|c|c|c|c|c|c|c|} 
		\hline  
		\multicolumn{1}{|c|}{\multirow{2}*{ $\sigma_S$ }}  &
		\multicolumn{1}{c|}{\multirow{2}*{  \tabincell{c}{ B-S \\  formula  }  }}  &
		\multicolumn{2}{c|}{ Euler Scheme  } & 
		  \multicolumn{2}{c|}{   Exponential  sampling  }& 
		  \multicolumn{2}{c|}{   Beta  sampling  }  \\
		   \cline{3-8} 
		&  &  Delta    &  95\% CI  &   Delta    &  95\% CI &    Delta   &  95\% CI   \\
		\hline
		0.25 & 0.556589&  0.555686 & [0.553178, 0.558194] &0.554613 & [0.551336, 0.557891] & 0.556314 & [0.553141, 0.559488]  \\
		\hline
		0.3 & 0.560018 & 0.561099 &[0.559455, 0.562742] &  0.557398&[0.553517, 0.56128]  &  0.557561 & [0.554848, 0.560274] \\
		\hline
		0.4 &  0.569512 &  0.570293 &[0.568533, 0.572053]  & 0.569098  & [0.565706, 0.572489] & 0.56731 & [0.564279, 0.570341] \\
		\hline
		 0.6&0.592743 & 0.594988 &[0.592957, 0.59702] & 0.586245  & [0.582428, 0.590062]  & 0.588015 & [0.584572, 0.591457]\\
		\hline
	\end{tabular}
	} 
	\caption{Comparison between the unbiased Monte Carlo estimation and the Monte Carlo Euler-Maruyama scheme for the Delta of a Call option in the Black-Scholes model for different values of $\sigma_S$.}
	\label{table:bs:delta}
\end{table}

 \begin{table}[htp]
	\centering
	\resizebox{\textwidth}{20mm}{
	\begin{tabular}{|c|c|c|c|c|c|} 
		\hline  
		\multicolumn{1}{|c|}{\multirow{2}*{ $\sigma_S$ }}  &
		\multicolumn{1}{c|}{\multirow{2}*{  \tabincell{c}{ B-S \\  formula  }  }}  &
		  \multicolumn{2}{c|}{   Exponential  sampling  }& 
		  \multicolumn{2}{c|}{   Beta  sampling  }  \\
		   \cline{3-6} 
		&  &     Vega   &  95\% CI &    Vega   &  95\% CI   \\
		\hline
		0.25 & 0 & 0.000745386 & [-0.00102979, 0.00252057] & -0.000438032 & [-0.00211468, 0.00123862] \\
		\hline
		0.3 & 0 & -0.0013932 & [-0.0036299, 0.000843502] & -0.000491083 & [-0.00249688, 0.00151471] \\
		\hline
		0.4 & 0 & 0.00331309 & [0.000258292, 0.00636788] & -0.00117019 & [-0.00393975, 0.00159938]\\
		\hline
		 0.6& 0 & -0.00286877 & [-0.00777679, 0.00203925] & -0.0027807 & [-0.00718374, 0.00162235] \\
		\hline
	\end{tabular}
	} 
	\caption{Comparison between the unbiased Monte Carlo estimation for the Vega of a Call option in the Black-Scholes model for different values of $\sigma_S$.}
	\label{table:bs:vega}
\end{table}

\subsection{A Stein-Stein type model}
In this second example, we consider a Stein-Stein type model where the diffusion coefficient function for the spot price is an affine function, namely $\sigma_S(x) = \sigma_1 x + \sigma_2$ where $\sigma_1$ and $\sigma_2$ are two positive constants. Note carefully that $\sigma_S$ is not uniformly elliptic and bounded so that \A{AR} and \A{ND} are not satisfied. However, we heuristically choose $\sigma_1$ and $\sigma_2$ so that $\sigma_S(Y_t)$ is bounded and strictly positive with high probability. Also, analytical expressions for the coefficients are available, namely
{\small
\begin{align*}
a_{S,i} &= \int_0^{\zeta_{i+1}-\zeta_i} \Big[\sigma_1 \big(\mu + (\bar{Y}_i - \mu) e^{-\lambda_Y s}\big) + \sigma_2\Big]^2 \, ds, \\
&= (\sigma_1 \mu + \sigma_2)^2 (\zeta_{i+1}-\zeta_i) + \sigma_1^2 (\bar{Y}_i - \mu)^2 \frac{1 - e^{- 2 \lambda_Y (\zeta_{i+1} - \zeta_i)}}{2 \lambda_Y} + 2 \sigma_1 (\sigma_1 \mu + \sigma_2) (\bar{Y}_i - \mu) \frac{1 - e^{- \lambda_Y (\zeta_{i+1} - \zeta_i)}}{\lambda_Y}, \\
a'_{S,i} &= \sigma_1^2 (\bar{Y}_i - \mu) \frac{1 - e^{- 2 \lambda_Y (\zeta_{i+1} - \zeta_i)}}{\lambda_Y} + 2\sigma_1  (\sigma_1 \mu + \sigma_2) \frac{1 - e^{- \lambda_Y (\zeta_{i+1} - \zeta_i)}}{\lambda_Y}, \\
\rho_i &= \rho \frac{\int_0^{\zeta_{i+1}-\zeta_i} \Big[ \alpha \big(\mu + (\bar{Y}_i - \mu) e^{-\lambda s}\big) + C \Big] \, ds}{\sigma_{S,i} \sqrt{\zeta_{i+1}-\zeta_i}} = \rho \frac{\alpha (\bar{Y}_i - \mu) (1 - e^{- \lambda (\zeta_{i+1} - \zeta_i)}) / \lambda + (\sigma_1 \mu + \sigma_2) (\zeta_{i+1}-\zeta_i)}{\sigma_{S,i}  \sqrt{\zeta_{i+1}-\zeta_i}}, \\
\rho'_i &= \rho \frac{\sigma_{S,i} \big( \sigma_1 (1 - e^{- \lambda_Y (\zeta_{i+1} - \zeta_i)}) / \lambda_Y \big) - \sigma'_{S,i} \big( \sigma_1 (\bar{Y}_i - \mu) (1 - e^{- \lambda_Y (\zeta_{i+1} - \zeta_i)}) / \lambda_Y + (\sigma_1 \mu + \sigma_2) (\zeta_{i+1}-\zeta_i) \big)}{a_{S,i}  \sqrt{\zeta_{i+1}-\zeta_i}}.
\end{align*} 
}

The parameters for the unbiased Monte Carlo method and the Monte Carlo method combined with an Euler-Maruyama approximation scheme are chosen as in the first example. The numerical results related to the price, Delta and Vega are provided in Table \ref{table:stein:price:call}, Table \ref{table:stein:delta:call}, Table \ref{table:stein:vega:call} respectively for the Call option and in Table \ref{table:stein:price:digital}, Table \ref{table:stein:delta:digital}, Table \ref{table:stein:vega:digital} for the digital Call option. In spite of the fact that the main assumptions are not satisfied, we again observe a good performance of the unbiased estimators for all three quantities and for all the values of the parameters $\sigma_1$, $\sigma_2$, except for the computation of the Vega of a Call option for large values of $\sigma_1$ and $\sigma_2$.
 \begin{table}[htp]
	\centering
	\resizebox{\textwidth}{20mm}{
	\begin{tabular}{|c|c|c|c|c|c|c|c|} 
		\hline  
		\multicolumn{1}{|c|}{\multirow{2}*{ $ \sigma_1 $ }}  &
		\multicolumn{1}{c|}{\multirow{2}*{ $ \sigma_2 $ }}  &
		\multicolumn{2}{c|}{ Euler Scheme  } & 
		  \multicolumn{2}{c|}{   Exponential  sampling  }& 
		  \multicolumn{2}{c|}{   Beta  sampling  }  \\
		   \cline{3-8} 
		&  & Price &    95\% CI  &    Price   &  95\% CI &    Price   &  95\% CI   \\
		\hline
		0.1 & 0.15  & 0.0790885 & [0.0784919, 0.0796851] &0.0794717  &[0.0785927, 0.0803508] & 0.0791559 & [0.0786344, 0.0796774] \\
		\hline
	0.2&	0.25 &  0.129665 & [0.128602, 0.130728] & 0.129818 & [0.128292, 0.131345] & 0.129055 & [0.126215, 0.131895] \\
		\hline
	0.3&	0.4 &   0.202324 & [0.200507, 0.20414] & 0.200155 &[0.199442, 0.200868] & 0.200371 &[0.199675, 0.201066]\\
		\hline
	0.4&	 0.5& 0.249249 & [0.246866, 0.251632] & 0.249114  & [0.248217, 0.250012] & 0.249279 & [0.248237, 0.250322]  \\
		\hline
	\end{tabular}
	} 
	\caption{Comparison between the unbiased Monte Carlo estimation for the price of a Call option in the Stein-Stein type model for different values of the parameters $\sigma_1$ and $\sigma_2$.}
	\label{table:stein:price:call}
\end{table}
 \begin{table}[htp]
	\centering
	\resizebox{\textwidth}{20mm}{
	\begin{tabular}{|c|c|c|c|c|c|c|c|} 
		\hline  
		\multicolumn{1}{|c|}{\multirow{2}*{ $ \sigma_1 $ }}  &
		\multicolumn{1}{c|}{\multirow{2}*{ $ \sigma_2 $ }}  &
		\multicolumn{2}{c|}{ Euler Scheme  } & 
		  \multicolumn{2}{c|}{   Exponential  sampling  }& 
		  \multicolumn{2}{c|}{   Beta  sampling  }  \\
		   \cline{3-8} 
		&  & Delta &    95\% CI  &    Delta   &  95\% CI &    Delta   &  95\% CI   \\
		\hline
		0.1 & 0.15 &0.545257 & [0.542601, 0.547914]& 0.542838 &  [0.534227, 0.551448]&  0.540304 & [0.535519, 0.545088] \\
		\hline
	0.2&	0.25 &  0.548642 & [0.54574, 0.551544] & 0.541611 & [0.533812, 0.549409] & 0.535165 & [0.51798, 0.552351] \\
		\hline
	0.3&	0.4 & 0.566919 & [0.563629, 0.570208]& 0.555688 &[0.548827, 0.56255]  & 0.558808  &[0.553545, 0.56407] \\
		\hline
	0.4&	 0.5&0.579445 & [0.575861, 0.583028] & 0.569003 &[0.559972, 0.578034]  &  0.568666 &[0.561328, 0.576004] \\
		\hline
	\end{tabular}
	} 
	\caption{Comparison between the unbiased Monte Carlo estimation for the Delta of a Call option in the Stein-Stein type model for different values of the parameters $\sigma_1$ and $\sigma_2$.}
	\label{table:stein:delta:call}
\end{table}
 \begin{table}[htp]
	\centering
	\resizebox{\textwidth}{20mm}{
	\begin{tabular}{|c|c|c|c|c|c|c|c|} 
		\hline  
		\multicolumn{1}{|c|}{\multirow{2}*{ $ \sigma_1 $ }}  &
		\multicolumn{1}{c|}{\multirow{2}*{ $ \sigma_2 $ }}  &
		\multicolumn{2}{c|}{ Euler Scheme  } & 
		  \multicolumn{2}{c|}{   Exponential  sampling  }& 
		  \multicolumn{2}{c|}{   Beta  sampling  }  \\
		   \cline{3-8} 
		&  & Vega &    95\% CI  &    Vega   &  95\% CI &    Vega  &  95\% CI   \\
		\hline
		0.1 & 0.15 &0.0370801 & [0.0367679, 0.0373923] & 0.0340152 & [0.0317984, 0.036232] &0.0350342  &  [0.0333134, 0.036755] \\
		\hline
	0.2&	0.25 &0.0738723 & [0.0731769, 0.0745676]& 0.0704958 & [0.0662385, 0.074753] & 0.0652897  & [0.0612527, 0.0693267]  \\
		\hline
	0.3&	0.4 &  0.11114 & [0.109907, 0.112373]& 0.0899367 & [0.0830599, 0.0968136] & 0.10303  & [0.0912086, 0.114851] \\
		\hline
	0.4&	 0.5& 0.14385 & [0.142055, 0.145645] & 0.122496 & [0.109106, 0.135885] & 0.133235  & [0.125356, 0.141114] \\
		\hline
	\end{tabular}
	} 
	\caption{Comparison between the unbiased Monte Carlo estimation for the Vega of a Call option in the Stein-Stein type model for different values of the parameters $\sigma_1$ and $\sigma_2$.}
	\label{table:stein:vega:call}
\end{table}

 \begin{table}[htp]
	\centering
	\resizebox{\textwidth}{23mm}{
	\begin{tabular}{|c|c|c|c|c|c|c|c|} 
		\hline  
		\multicolumn{1}{|c|}{\multirow{2}*{ $ \sigma_1 $ }}  &
		\multicolumn{1}{c|}{\multirow{2}*{ $ \sigma_2 $ }}  &
		\multicolumn{2}{c|}{ Euler Scheme  } & 
		  \multicolumn{2}{c|}{   Exponential  sampling  }& 
		  \multicolumn{2}{c|}{   Beta  sampling  }  \\
		   \cline{3-8} 
		&  & Price &    95\% CI  &    Price   &  95\% CI &    Price   &  95\% CI   \\
		\hline
		0 &  0.3  &  0.469652 & [0.467241, 0.472063] & 0.469251 & [0.468668, 0.469834]  &  0.468994 & [0.468439, 0.469548]  \\
		\hline
		0.1 & 0.15  &  0.491121 & [0.488708, 0.493535]&  0.489251 & [0.488158, 0.490344]  &  0.489883 & [0.489044, 0.490723] \\
		\hline
	0.2&	0.25 &  0.459518 & [0.45711, 0.461926] &  0.458577 & [0.457324, 0.459829]  &  0.458555 & [0.457574, 0.459535] \\
		\hline
	0.3&	0.4 &   0.430451 & [0.428057, 0.432845] &  0.428559 & [0.42734, 0.429779]  &  0.428941 & [0.428074, 0.429807 \\
		\hline
	0.4&	 0.5 &  0.408908 & [0.406529, 0.411286] &0.40788 & [0.404907, 0.410852] &0.409511 & [0.408526, 0.410496] \\
		\hline
	\end{tabular}
	} 
	\caption{Comparison between the unbiased Monte Carlo estimation for the price of a digital Call option in the Stein-Stein type model for different values of the parameters $\sigma_1$ and $\sigma_2$.}
	\label{table:stein:price:digital}
\end{table}

 \begin{table}[htp]
	\centering
	\resizebox{\textwidth}{23mm}{
	\begin{tabular}{|c|c|c|c|c|c|c|c|} 
		\hline  
		\multicolumn{1}{|c|}{\multirow{2}*{ $ \sigma_1 $ }}  &
		\multicolumn{1}{c|}{\multirow{2}*{ $ \sigma_2 $ }}  &
		\multicolumn{2}{c|}{ Euler Scheme  } & 
		  \multicolumn{2}{c|}{   Exponential  sampling  }& 
		  \multicolumn{2}{c|}{   Beta  sampling  }  \\
		   \cline{3-8} 
		&  & Delta &    95\% CI  &    Delta   &  95\% CI &    Delta   &  95\% CI   \\
		\hline
		0 &  0.3  &  1.22307 & [1.19252, 1.25363] &  1.24579 & [1.24326, 1.24833]  &  1.24408 & [1.24165, 1.24651] \\
		\hline
		0.1 & 0.15  &  2.17706 & [2.13691, 2.21721] &  2.17577 & [2.16695, 2.18459]  &  2.18049 & [2.17398, 2.18701]  \\
		\hline
	0.2&	0.25 &  1.29839 & [1.26695, 1.32984] &  1.26832 & [1.26267, 1.27397]  &  1.26854 & [1.26428, 1.27279] \\
		\hline
	0.3&	0.4 &  0.776519 & [0.752002, 0.801036] &  0.792788 & [0.789598, 0.795978]  &  0.793139 & [0.790702, 0.795577] \\
		\hline
	0.4&	 0.5 &0.606688 & [0.58496, 0.628416]&0.618031 & [0.610424, 0.625638] & 0.621753 & [0.619061, 0.624446] \\
		\hline
	\end{tabular}
	} 
	\caption{Comparison between the unbiased Monte Carlo estimation for the Delta of a digital Call option in the Stein-Stein type model for different values of the parameters $\sigma_1$ and $\sigma_2$.}
	\label{table:stein:delta:digital}
\end{table}

 \begin{table}[htp]
	\centering
	\resizebox{\textwidth}{23mm}{
	\begin{tabular}{|c|c|c|c|c|c|c|c|} 
		\hline  
		\multicolumn{1}{|c|}{\multirow{2}*{ $ \sigma_1 $ }}  &
		\multicolumn{1}{c|}{\multirow{2}*{ $ \sigma_2 $ }}  &
		\multicolumn{2}{c|}{ Euler Scheme  } & 
		  \multicolumn{2}{c|}{   Exponential  sampling  }& 
		  \multicolumn{2}{c|}{   Beta  sampling  }  \\
		   \cline{3-8} 
		&  & Vega &    95\% CI  &    Vega   &  95\% CI &    Vega  &  95\% CI   \\
		\hline
		0 &  0.3  &  0 & [0, 0] &  0.000481062 & [-0.00499082, 0.00595295]  &  -0.000755278 & [-0.00588091, 0.00437035] \\
		\hline
		0.1 & 0.15  &  -0.0200101 & [-0.0308873, -0.00913292] &  -0.0249364 & [-0.0346885, -0.0151842]  &  -0.0286496 & [-0.0358769, -0.0214223] \\
		\hline
	0.2&	0.25 &-0.0246278 & [-0.0366948, -0.0125608] &  -0.032211 & [-0.0436935, -0.0207285]  &  -0.0311689 & [-0.0392428, -0.023095] \\
		\hline
	0.3&	0.4 & -0.0354025 & [-0.04987, -0.0209349]&  -0.0422004 & [-0.0518535, -0.0325472]  &  -0.0413018 & [-0.0489346, -0.0336691] \\
		\hline
	0.4&	 0.5 &  -0.0492556 & [-0.0663201, -0.0321911] & -0.0512594 & [-0.0638074, -0.0387114] & -0.0517876 & [-0.0597881, -0.0437871] \\
		\hline
	\end{tabular}
	} 
	\caption{Comparison between the unbiased Monte Carlo estimation for the Vega of a digital Call option in the Stein-Stein type model for different values of the parameters $\sigma_1$ and $\sigma_2$.}
	\label{table:stein:vega:digital}
\end{table}

\vspace{5\baselineskip}
\subsection{A model with a periodic diffusion coefficient function}
In our last example, the volatility of spot price takes the following form $\sigma_S(x) = \sigma_1 \cos(x) + \sigma_2$ where $\sigma_1$ and $\sigma_2$ are two positive constants such that $\sigma_2-\sigma_1>0$ in order to ensure that \A{ND} is satisfied. Here, the coefficients appearing in the dynamics \eqref{euler:scheme} write{\small
\begin{align*}
a_{S,i} &= \int_0^{\zeta_{i+1}-\zeta_i} \Big[\sigma_1 \cos \big(\mu + (\bar{Y}_i - \mu) e^{-\lambda_Y s}\big) + \sigma_2\Big]^2 \, ds, \\
a'_{S,i} &= - 2 \alpha \int_0^{\zeta_{i+1}-\zeta_i} e^{- \lambda_Y s} \sin \big(\mu + (\bar{Y}_i - \mu) e^{-\lambda_Y s}\big) \Big[\sigma_1 \cos \big(\mu + (\bar{Y}_i - \mu) e^{-\lambda_Y s}\big) + \sigma_2\Big] \, ds, \\
\rho_i &= \rho \frac{\int_0^{\zeta_{i+1}-\zeta_i} \Big[ \sigma_1 \cos \big(\mu + (\bar{Y}_i - \mu) e^{-\lambda_Y s}\big) + \sigma_2 \Big] \, ds}{\sigma_{S,i} \sqrt{\zeta_{i+1}-\zeta_i}}, \\
\rho'_i &= - \rho \frac{ \sigma_1 \sigma_{S,i}  \int_0^{\zeta_{i+1}-\zeta_i} e^{- \lambda_Y s} \sin \big(\mu + (\bar{Y}_i - \mu) e^{-\lambda_Y s}\big) ds + \sigma'_{S,i} \int_0^{\zeta_{i+1}-\zeta_i} \Big[ \sigma_1 \cos \big(\mu + (\bar{Y}_i - \mu) e^{-\lambda_Y s}\big) + \sigma_2 \Big] \, ds}{a_{S,i} \sqrt{\zeta_{i+1}-\zeta_i}}
\end{align*}
}

\noindent and no analytical expressions are available. However, a simple numerical integration method can be employed for the computation of the above integrals. We here use Simpson's 3/8 rule which for a real-valued $\mathcal{C}^{4}([0,T])$ function $g$ writes as follows
\begin{align*}
\forall t \in [0,T], \quad \int_0^t g(s)ds \approx \frac{t}{8} \left(g(0) + 3g\left(\frac{t}{3}\right) +  3g\left(\frac{2t}{3}\right) + g(t)\right)
\end{align*}
\noindent with an error given by $ g^{(4)}(t') T^5/6480$ for some $t' \in [0,T]$.

The parameters of the unbiased Monte Carlo method and the Monte Carlo Euler-Maruyama scheme remain unchanged. The numerical results related to the price, Delta and Vega are provided in Table \ref{table:cosinus:price:call}, Table \ref{table:cosinus:delta:call}, Table \ref{table:cosinus:vega:call} respectively for the Call option and in Table \ref{table:cosinus:price:digital}, Table \ref{table:cosinus:delta:digital}, Table \ref{table:cosinus:vega:digital} for the digital Call option. Here again, the unbiased estimators perform very well for all range of values of the parameters.

 \begin{table}[htp]
	\centering
	\resizebox{\textwidth}{20mm}{
	\begin{tabular}{|c|c|c|c|c|c|c|c|} 
		\hline  
		\multicolumn{1}{|c|}{\multirow{2}*{ $ \sigma_1 $ }}  &
		\multicolumn{1}{c|}{\multirow{2}*{ $ \sigma_2 $ }}  &
		\multicolumn{2}{c|}{ Euler Scheme  } & 
		  \multicolumn{2}{c|}{   Exponential  sampling  }& 
		  \multicolumn{2}{c|}{   Beta  sampling  }  \\
		   \cline{3-8} 
		&  & Price &    95\% CI  &    Price   &  95\% CI &    Price   &  95\% CI   \\
		\hline
		0.1 & 0.15 & 0.110649 & [0.109801, 0.111497]&0.111245  & [0.110746, 0.111745] &  0.111163 & [0.11071, 0.111617] \\
		\hline
	0.2&	0.25 & 0.193525 & [0.191897, 0.195154] & 0.19476 & [0.19378, 0.19574] & 0.193705 & [0.192939, 0.19447] \\
		\hline
	0.3&	0.4 & 0.294275 & [0.291444, 0.297106] &  0.294418& [0.292502, 0.296333] &  0.294724 &[0.293178, 0.296269] \\
		\hline
	0.4&	 0.5&  0.371509 & [0.367579, 0.375439] & 0.3739 & [0.371509, 0.376292] & 0.373974 & [0.372141, 0.375806] \\
		\hline
	\end{tabular}
	} 
	\caption{Comparison between the unbiased Monte Carlo estimation for the price of a Call option in the model with $\sigma_S(x) = \sigma_1 \cos(x) + \sigma_2$ for different values of the parameters $\sigma_1$ and $\sigma_2$.}
	\label{table:cosinus:price:call}
\end{table}

\begin{table}[htp]
	\centering
	\resizebox{\textwidth}{20mm}{
	\begin{tabular}{|c|c|c|c|c|c|c|c|} 
		\hline  
		\multicolumn{1}{|c|}{\multirow{2}*{ $ \sigma_1 $ }}  &
		\multicolumn{1}{c|}{\multirow{2}*{ $ \sigma_2 $ }}  &
		\multicolumn{2}{c|}{ Euler Scheme  } & 
		  \multicolumn{2}{c|}{   Exponential  sampling  }& 
		  \multicolumn{2}{c|}{   Beta  sampling  }  \\
		   \cline{3-8} 
		&  & Delta &    95\% CI  &    Delta  &  95\% CI &    Delta  &  95\% CI   \\
		\hline
		0.1 & 0.15 & 0.556917 & [0.554118, 0.559717] &0.560077  & [0.556733, 0.563422] & 0.555364 & [0.552636, 0.558092] \\
		\hline
	0.2&	0.25 & 0.577937 & [0.574727, 0.581148]&  0.579704 &  [0.575622, 0.583785] &  0.577287 & [0.574331, 0.580243] \\
		\hline
	0.3&	0.4 &0.605788 & [0.601947, 0.60963] & 0.604575 & [0.602771, 0.606379] &0.601188  & [0.599354, 0.603021] \\
		\hline
	0.4&	 0.5&0.62865 & [0.624204, 0.633096] &0.623698  & [0.618519, 0.628878] & 0.626259 & [0.622246, 0.630271] \\
		\hline
	\end{tabular}
	} 
	\caption{Comparison between the unbiased Monte Carlo estimation for the Delta of a Call option in the model with $\sigma_S(x) = \sigma_1 \cos(x) + \sigma_2$ for different values of the parameters $\sigma_1$ and $\sigma_2$.}
	\label{table:cosinus:delta:call}
\end{table}

\begin{table}[htp]
	\centering
	\resizebox{\textwidth}{20mm}{
	\begin{tabular}{|c|c|c|c|c|c|c|c|} 
		\hline  
		\multicolumn{1}{|c|}{\multirow{2}*{ $ \sigma_1 $ }}  &
		\multicolumn{1}{c|}{\multirow{2}*{ $ \sigma_2 $ }}  &
		\multicolumn{2}{c|}{ Euler Scheme  } & 
		  \multicolumn{2}{c|}{   Exponential  sampling  }& 
		  \multicolumn{2}{c|}{   Beta  sampling  }  \\
		   \cline{3-8} 
		&  & Vega &    95\% CI  &    Vega   &  95\% CI &    Vega   &  95\% CI   \\
		\hline
		0.1 & 0.15 & -0.00773549 & [-0.00782648, -0.00764449] & -0.00805159 & [-0.00985368, -0.0062495] & -0.00846248 &[-0.0101504, -0.00677453]  \\
		\hline
	0.2&	0.25 &-0.0156691 & [-0.0158849, -0.0154532]& -0.0161045 & [-0.0194751, -0.0127339] &  -0.0137565 &[-0.0169305, -0.0105825] \\
		\hline
	0.3&	0.4 &-0.0235822 & [-0.0240098, -0.0231547]  & -0.0177797 & [-0.0236385, -0.0119209]  &  -0.0232616 &[-0.0288379, -0.0176852] \\
		\hline
	0.4&	 0.5&-0.030774 & [-0.0314484, -0.0300996]&-0.031729 & [-0.0405267, -0.0229313] &  -0.0327252 & [-0.0402293, -0.0252211] \\
		\hline
	\end{tabular}
	} 
	\caption{Comparison between the unbiased Monte Carlo estimation for the Vega of a Call option in the model with $\sigma_S(x) = \sigma_1 \cos(x) + \sigma_2$ for different values of the parameters $\sigma_1$ and $\sigma_2$.}
	\label{table:cosinus:vega:call}
\end{table}

 \begin{table}[htp]
	\centering
	\resizebox{\textwidth}{23mm}{
	\begin{tabular}{|c|c|c|c|c|c|c|c|} 
		\hline  
		\multicolumn{1}{|c|}{\multirow{2}*{ $ \sigma_1 $ }}  &
		\multicolumn{1}{c|}{\multirow{2}*{ $ \sigma_2 $ }}  &
		\multicolumn{2}{c|}{ Euler Scheme  } & 
		  \multicolumn{2}{c|}{   Exponential  sampling  }& 
		  \multicolumn{2}{c|}{   Beta  sampling  }  \\
		   \cline{3-8} 
		&  & Price &    95\% CI  &    Price   &  95\% CI &    Price   &  95\% CI   \\
		\hline
		0 &  0.3  &  0.470206 & [0.467795, 0.472617] &  0.468756 & [0.468174, 0.469338]  &  0.46883 & [0.468273, 0.469387] \\
		\hline
		0.1 & 0.15  & 0.481972 & [0.479559, 0.484385]&  0.481373 & [0.480778, 0.481968]  &  0.481499 & [0.480937, 0.482061] \\
		\hline
	0.2&	0.25 & 0.446163 & [0.443761, 0.448566] & 0.445241 & [0.444661, 0.445821]  &  0.445228 & [0.444679, 0.445778]  \\
		\hline
	0.3&	0.4 &  0.407307 & [0.40493, 0.409684]  & 0.407842 & [0.407286, 0.408398] & 0.407422 & [0.40689, 0.407954] \\
		\hline
	0.4&	 0.5 & 0.379459 & [0.37711, 0.381808]& 0.380161 & [0.379604, 0.380717] & 0.379669 & [0.379146, 0.380193]  \\
		\hline
	\end{tabular}
	} 
	\caption{ Comparison between the unbiased Monte Carlo estimation for the price of a digital Call option in the model with $\sigma_S(x) = \sigma_1 \cos(x) + \sigma_2$ for different values of the parameters $\sigma_1$ and $\sigma_2$.}
	\label{table:cosinus:price:digital}
\end{table}

\begin{table}[htp]
	\centering
	\resizebox{\textwidth}{23mm}{
	\begin{tabular}{|c|c|c|c|c|c|c|c|} 
		\hline  
		\multicolumn{1}{|c|}{\multirow{2}*{ $ \sigma_1 $ }}  &
		\multicolumn{1}{c|}{\multirow{2}*{ $ \sigma_2$ }}  &
		\multicolumn{2}{c|}{ Euler Scheme  } & 
		  \multicolumn{2}{c|}{   Exponential  sampling  }& 
		  \multicolumn{2}{c|}{   Beta  sampling  }  \\
		   \cline{3-8} 
		&  & Delta &    95\% CI  &    Delta  &  95\% CI &    Delta  &  95\% CI   \\
		\hline
		0 &  0.3  & 1.24454 & [1.21372, 1.27535]  & 1.24282 & [1.24031, 1.24533] & 1.24579 & [1.24335, 1.24822]  \\
		\hline
		0.1 & 0.15  & 1.5064 & [1.47264, 1.54016]& 1.51062 & [1.50751, 1.51373] & 1.51253 & [1.50958, 1.51549]  \\
		\hline
	0.2&	0.25 & 0.846474 & [0.820904, 0.872044] & 0.835956 & [0.834251, 0.83766] & 0.834907 & [0.833284, 0.836531]  \\
		\hline
	0.3&	0.4 &  0.512589 & [0.492588, 0.53259]&0.528498 & [0.527433, 0.529563]  &   0.527844 & [0.526825, 0.528863] \\
		\hline
	0.4&	 0.5 & 0.399506 & [0.381818, 0.417194]&  0.40319 & [0.402367, 0.404012] &  0.402428 & [0.401642, 0.403214]  \\
		\hline
	\end{tabular}
	} 
	\caption{ Comparison between the unbiased Monte Carlo estimation for the Delta of a digital Call option in the model with $\sigma_S(x) = \sigma_1 \cos(x) + \sigma_2$ for different values of the parameters $\sigma_1$ and $\sigma_2$.}
	\label{table:cosinus:delta:digital}
\end{table}

\begin{table}[htp]
	\centering
	\resizebox{\textwidth}{23mm}{
	\begin{tabular}{|c|c|c|c|c|c|c|c|} 
		\hline  
		\multicolumn{1}{|c|}{\multirow{2}*{ $ \sigma_1 $ }}  &
		\multicolumn{1}{c|}{\multirow{2}*{ $ \sigma_2 $ }}  &
		\multicolumn{2}{c|}{ Euler Scheme  } & 
		  \multicolumn{2}{c|}{   Exponential  sampling  }& 
		  \multicolumn{2}{c|}{   Beta  sampling  }  \\
		   \cline{3-8} 
		&  & Vega &    95\% CI  &    Vega   &  95\% CI &    Vega   &  95\% CI   \\
		\hline
		0 &  0.3  & 0 & [0, 0] & 0.00122669 & [-0.00426466, 0.00671804] & -0.00228675 & [-0.00742253, 0.00284903]  \\
		\hline
		0.1 & 0.15  & 0.00769619 & [-0.000285705, 0.0156781]&  0.00900717 & [0.00339831, 0.014616]  & 0.00730275 & [0.0021288, 0.0124767]  \\
		\hline
	0.2&	0.25 & 0.0138531 & [0.00480268, 0.0229036] & 0.0146584 & [0.00906018, 0.0202566] & 0.0131819 & [0.00819735, 0.0181665] \\
		\hline
	0.3&	0.4 &0.0107747 & [0.00279286, 0.0187565] & 0.00807909 & [0.00292412, 0.0132341] &  0.0116897 & [0.00689717, 0.0164823] \\
		\hline
	0.4&	 0.5 &  0.0153924 & [0.00585238, 0.0249324] & 0.0152859 & [0.0102097, 0.0203621] &  0.0164134 & [0.0117414, 0.0210853]  \\
		\hline
	\end{tabular}
	} 
	\caption{Comparison between the unbiased Monte Carlo estimation for the Vega of a digital Call option in the model with $\sigma_S(x) = \sigma_1 \cos(x) + \sigma_2$ for different values of the parameters $\sigma_1$ and $\sigma_2$.}
	\label{table:cosinus:vega:digital}
\end{table}

   \vspace{10\baselineskip}

\bibliographystyle{abbrv}
\bibliography{biblio}

%
%

\appendix

\section{Proof of Theorem \ref{theorem:probabilistic:formulation} and Lemma \ref{lem:transfer:derivative}}\label{section:appendix}

\subsection{Proof of Theorem \ref{theorem:probabilistic:formulation}}\label{proof:probabilistic:representation:marginal:law}

The proof is divided into three steps. In the first part, we establish the probabilistic representation for a bounded and continuous function $h$. We then provide the extension to measurable maps satisfying the growth condition \ref{growth:assumption:test:function}. We eventually conclude by establishing the $L^{p}$-moments when the jump times are distributed according to the Beta law.

Denote by $\mathcal{L}$ and $(\bar{\mathcal{L}}_t)_{t\geq0}$ the infinitesimal generators of $(P_t)_{t\geq0}$ and $(\bar{P}_t)_{t\geq0}$ respectively given by 
\begin{align*}
\mathcal{L} f(x, y) & = (r-\frac12 a_S(y))\partial_x f(x, y) + \frac12 a_S(y) \partial^2_x f(x, y) + b_Y(y) \partial_y f(x, y) + \frac12 a_Y(y)  \partial^2_y f(x, y) + \rho (\sigma_S \sigma_Y)(y) \partial^2_{x, y} f(x, y), \\
\bar{\mathcal{L}}_{t} f(x, y) & = (r-\frac12 a_{S}(m_t(y_0))) \partial_x f(x, y) + \frac12 a_{S}(m_t(y_0)) \partial^2_x f(x, y) + b_Y(m_t(y_0))\partial_y f(x, y) + \frac12 a_Y(m_t(y_0))  \partial^2_y f(x, y) \\
& \quad + \rho (\sigma_S \sigma_{Y})(m_t(y_0)) \partial^2_{x, y} f(x, y)
\end{align*}
\noindent for any $f \in \mathcal{C}^2_b(\mathbb{R}^2)$.

\smallskip

\emph{Step 1: Probabilistic representation for a bounded and continuous map $h$}\\

We establish a first order expansion of the Markov semigroup $(P_t)_{t\geq0}$ around $(\bar{P}_t)_{t\geq0}$. We apply It\^o's rule to the map $[0,t] \times \mathbb{R}^2 \ni (s, x, y) \mapsto P_{t-s}h(x, y) \in \mathcal{C}^{1,2}_b([0,t] \times \mathbb{R}^2)$ for $h \in \mathcal{C}^{\infty}_b(\mathbb{R}^2)$, observing that $\partial_s P_{t-s}h(x, y) = -\mathcal{L} P_{t-s}h(x, y)$. We obtain
\begin{align*}
h(\bar{X}_t, \bar{Y}_t) & = P_t h(x_0,y_0) + \int_0^t \Big(\partial_s P_{t-s} h(\bar{X}_s, \bar{Y}_s) + \bar{\mathcal{L}}_sP_{t-s} h(\bar{X}_s, \bar{Y}_s)  \Big) \, ds + \, M_t \\
& = P_t h(x_0, y_0) + \int_0^t  (\bar{\mathcal{L}}_{s} - \mathcal{L}) P_{t-s}h(\bar{X}_s, \bar{Y}_s) \, ds + M_t
\end{align*}

\noindent where $M:=(M_t)_{t\geq0}$ is a square integrable martingale. We then take expectation in the previous expression, make use of Fubini's theorem and finally let $t\uparrow T$ by dominated convergence theorem so that
\begin{align}
P_T h(x_0,y_0) & = \bar{P}_T h(x_0, y_0) + \int_0^T \mathbb{E}[(\mathcal{L}-\bar{\mathcal{L}}_{s})P_{T-s}h(\bar{X}_s, \bar{Y}_s)] \, ds \nonumber \\
 &= \mathbb{E}[h(\bar{X}^{x_0}_T, \bar{Y}^{y_0}_T)] + \int_0^T \mathbb{E}\Big[\frac12 (a_S(\bar{Y}^{y_0}_s) - a_S(m_s(y_0)) [\partial^2_x P_{T-s}h(\bar{X}^{x_0}_s, \bar{Y}^{y_0}_s) -\partial_x P_{T-s}h(\bar{X}^{x_0}_s, \bar{Y}^{y_0}_s) )]\Big] \, ds \nonumber \\
 & + \int_0^T \mathbb{E}\Big[\frac12 (a_Y(\bar{Y}^{y_0}_s) - a_Y(m_s(y_0))) \partial^2_y P_{T-s}h(\bar{X}^{x_0}_s, \bar{Y}^{y_0}_s) + (b_Y(\bar{Y}^{y_0}_s) - b_Y(m_s(y_0))) \partial_y P_{T-s}h(\bar{X}^{x_0}_s, \bar{Y}^{y_0}_s) \Big] \, ds\nonumber\\
 & + \int_0^T \mathbb{E}\Big[\rho((\sigma_S\sigma_Y)(\bar{Y}^{x_0}_s)-(\sigma_S\sigma_Y)(m_s(y_0))) \partial^2_{x, y} P_{T-s}h(\bar{X}^{x_0}_s, \bar{Y}^{y_0}_s)\Big] \, ds. \label{first:step:expansion:before:prob:interpretation}
\end{align}

We now rewrite the previous first order expansion using the Markov chain $(\bar{X}_i, \bar{Y}_i)_{0\leq i \leq N_T+1}$ and the renewal process $N$. From the previous identity, the definition of $\theta_{N_T+1}$ in \eqref{definition:theta:theorem:probabilistic:representation} and the identity \eqref{probabilistic:representation:time:integrals}, we directly obtain
\begin{align}
P_T h(x_0, y_0)&  = \mathbb{E}[h(\bar{X}_{N_T+1}, \bar{Y}_{N_T+1}) \theta_{N_T+1} \I_\seq{N_T = 0}]  \\
& \quad + \mathbb{E}\left[((1-F(T-\zeta_1)) f(\zeta_1))^{-1} \I_\seq{N_T=1} \left[ \frac12 (a_S(\bar{Y}_1) - a_S(m_{0})) \mathcal{D}^{(1,1)}_1 P_{T-\zeta_1}h(\bar{X}_1, \bar{Y}_1) \right. \right. \nonumber \\
& \quad \left.\left. - \frac12 (a_S(\bar{Y}_1) - a_S(m_{0})) \mathcal{D}^{(1)}_1 P_{T-\zeta_1}h(\bar{X}_1, \bar{Y}_1)  + \frac12 (a_Y(\bar{Y}_1) - a_Y(m_{0})) \mathcal{D}^{(2,2)}_1 P_{T-\zeta_1}h(\bar{X}_1, \bar{Y}_1) \right. \right. \nonumber  \\
& \quad \left. \left. +  (b_Y(\bar{Y}_1) - b_Y(m_{0})) \mathcal{D}^{(2)}_1 P_{T-\zeta_1}h(\bar{X}_1, \bar{Y}_1) +  \rho ((\sigma_S\sigma_Y)(\bar{Y}_1) - (\sigma_S\sigma_Y)(m_{0})) \mathcal{D}^{(1,2)}_1 P_{T-\zeta_1}h(\bar{X}_1, \bar{Y}_1) \right] \right] \nonumber \\
&  = \mathbb{E}[h(\bar{X}_{N_T+1}, \bar{Y}_{N_T+1}) \theta_{N_T+1} \I_\seq{N_T = 0}]  + \mathbb{E}\Big[((1-F(T-\zeta_1)) f(\zeta_1))^{-1} \I_\seq{N_T=1} \Big[ c^1_S \mathcal{D}^{(1,1)}_1 P_{T-\zeta_1}h(\bar{X}_1, \bar{Y}_1) \nonumber\\
& \quad - c^1_S \mathcal{D}^{(1)}_1 P_{T-\zeta_1}h(\bar{X}_1, \bar{Y}_1)  + c^1_Y \mathcal{D}^{(2,2)}_1 P_{T-\zeta_1} h(\bar{X}_1, \bar{Y}_1) + b^1_Y \mathcal{D}^{(2)}_1 P_{T-\zeta_1}h(\bar{X}_1, \bar{Y}_1) +  c^1_{S,Y} \mathcal{D}^{(1,2)}_1 P_{T-\zeta_1}h(\bar{X}_1, \bar{Y}_1) \Big] \Big]. \label{first:step:expansion:before:ibp}
\end{align}

Next, we apply the IBP formula \eqref{duality:formula} with respect to the random vector $(\bar{X}_1, \bar{Y}_1)$ in the above expression. In order to do that rigorously, one first has to take the conditional expectation $\mathbb{E}_{0,1}[.]$ in the second term of the above equality. We thus obtain
\begin{align*}
\mathbb{E}_{0,1}\Big[& c^1_S \mathcal{D}^{(1,1)}_1 P_{T-\zeta_1} h(\bar{X}_1, \bar{Y}_1)  - c^1_S \mathcal{D}^{(1)}_1 P_{T-\zeta_1} h(\bar{X}_1, \bar{Y}_1)  + c^1_Y \mathcal{D}^{(2,2)}_1 P_{T-\zeta_1} h(\bar{X}_1, \bar{Y}_1) +  b^1_Y \mathcal{D}^{(2)}_1 P_{T-\zeta_1} h(\bar{X}_1, \bar{Y}_1)  \\
& +  c^1_{S,Y} \mathcal{D}^{(1,2)}_1 P_{T-\zeta_1} h(\bar{X}_1, \bar{Y}_1)  \Big]=\mathbb{E}_{0,1}\Big[\Big[\mathcal{I}^{(1,1)}_1(c^1_S) - \mathcal{I}^{(1)}_1(c^1_S) + \mathcal{I}^{(2,2)}_1( c^1_Y) + \mathcal{I}^{(2)}_1( b^1_Y) + \mathcal{I}^{(1,2)}_1(c^1_{S, Y})\Big] P_{T-\zeta_1} h(\bar{X}_1, \bar{Y}_1)\Big].
\end{align*}

From Lemma \ref{estimate:theta:time:loca:density} and the estimate \eqref{conditional:lp:moment:space:M:I:N}, we get
\begin{equation}
\mathbb{E}_{0,1}\Big[\Big|\Big[\mathcal{I}^{(1,1)}_1(c^1_S) - \mathcal{I}^{(1)}_1(c^1_S) + \mathcal{I}^{(2,2)}_1( c^1_Y) + \mathcal{I}^{(2)}_1( b^1_Y) + \mathcal{I}^{(1,2)}_1(c^1_{S, Y})\Big]\Big| \Big| P_{T-\zeta_1} h(\bar{X}_1, \bar{Y}_1)\Big| \Big] \leq C_T |h|_\infty \zeta^{-1/2}_1 \label{first:time:degeneracy:estimate}
\end{equation}

\noindent for some positive constant $C_T$ such that $T\mapsto C_T$ is non-decreasing. The previous estimate yields an integrable time singularity. Indeed, from the previous estimate and \eqref{probabilistic:representation:time:integrals}, one directly gets
\begin{align*}
\mathbb{E}\Big[((1-F(T-\zeta_1)) f(\zeta_1))^{-1}&  \I_\seq{N_T=1} \Big|\mathbb{E}_{0,1}\Big[\Big[\mathcal{I}^{(1,1)}_1(c^1_S) - \mathcal{I}^{(1)}_1(c^1_S) + \mathcal{I}^{(2,2)}_1( c^1_Y) + \mathcal{I}^{(2)}_1( b^1_Y) + \mathcal{I}^{(1,2)}_1(c^1_{S, Y})\Big] P_{T-\zeta_1} h(\bar{X}_1, \bar{Y}_1)\Big]\Big|\Big]\\
& \leq C \mathbb{E}[((1-F(T-\zeta_1)) f(\zeta_1))^{-1}   \zeta^{-1/2}_1 \I_\seq{N_T=1}] \\
& \leq C \int_0^T s^{-1/2}_1 \, ds_1 < \infty.
\end{align*}

Coming back to \eqref{first:step:expansion:before:ibp}, we thus derive
\begin{align}
P_T h(x_0, y_0) &  = \mathbb{E}[h(\bar{X}_{N_T+1}, \bar{Y}_{N_T+1}) \theta_{N_T+1} \I_\seq{N_T = 0}] \nonumber\\
&  + \mathbb{E}\Big[((1-F(T-\zeta_1)) f(\zeta_1))^{-1} \I_\seq{N_T=1} \nonumber\\
& \quad \times \Big[\mathcal{I}^{(1,1)}_1(c^1_S) - \mathcal{I}^{(1)}_1(c^1_S) + \mathcal{I}^{(2,2)}_1( c^1_Y) + \mathcal{I}^{(2)}_1( b^1_Y) +  \mathcal{I}^{(1,2)}_1(c^1_{S, Y})\Big] P_{T-\zeta_1} h(\bar{X}_1, \bar{Y}_1) \Big] \nonumber\\
& = \mathbb{E}[h(\bar{X}_{N_T+1}, \bar{Y}_{N_T+1}) \theta_{N_T+1} \I_\seq{N_T = 0}] + \mathbb{E}\Big[P_{T-\zeta_1}h(\bar{X}_1, \bar{Y}_1) \theta_2 \theta_1 \I_\seq{N_T=1} \Big]. \label{first:order:expansion:with:prob:interpretation}
\end{align}

Our aim now is to iterate the above first order expansion. We prove by induction the following formula: for any positive integer $n$, one has
\begin{align}
P_Th(x_0,y_0)= \sum_{j=0}^{n-1} \E\left[ h(\bar{X}_{N_T+1}, \bar{Y}_{N_T+1})\prod_{i=1}^{N_T+1}  \theta_i  \I_\seq{N_T=j}\right ] +  \E\left[ P_{T-\zeta_{n}}h(\bar{X}_{n}, \bar{Y}_n)\prod_{i=1}^{n+1}  \theta_i \I_\seq{N_T=n}\right ].\label{induction:step:probabilistic:representation}
 \end{align} 
 The case $n=1$ corresponds to \eqref{first:order:expansion:with:prob:interpretation}. We thus assume that \eqref{induction:step:probabilistic:representation} holds at step $n$. We expand the last term appearing in the right-hand side of the previous equality using again \eqref{first:step:expansion:before:prob:interpretation}, by then applying Lemma \ref{lem:add:new:jumps} and by finally performing IBPs as before.
 
 To be more specific, using the notations introduced in Subsection \ref{subsection:choice:approximation:process}, from \eqref{first:step:expansion:before:prob:interpretation} and a change of variable, for any $\zeta\in [0,T]$, one has
\begin{align*}
P_{T-\zeta} h(x, y) & = \mathbb{E}[h(\bar{X}^{\zeta, x}_{T}, \bar{Y}^{\zeta, x}_T)]  \nonumber \\
& + \int_\zeta^T \mathbb{E}\Big[\frac12 (a_S(\bar{Y}^{\zeta, y}_s) - a_S(m_{s-\zeta}(y)) [\partial^2_x P_{T-s}h(\bar{X}^{\zeta, x}_s, \bar{Y}^{\zeta, y}_s) - \partial_x P_{T-s}h(\bar{X}^{\zeta, x}_s, \bar{Y}^{\zeta, y}_s) )]\Big] \, ds \nonumber \\
 & + \int_\zeta^T \mathbb{E}\Big[\frac12 (a_Y(\bar{Y}^{\zeta, y}_s) - a_Y(m_{s-\zeta}(y))) \partial^2_y P_{T-s}h(\bar{X}^{\zeta, x}_s, \bar{Y}^{\zeta, y}_s) + (b_Y(\bar{Y}^{\zeta, y}_s) - b_Y(m_{s-\zeta}(y))) \partial_y P_{T-s}h(\bar{X}^{\zeta, x}_s, \bar{Y}^{\zeta, y }_{s}) \Big] \, ds\nonumber\\
 & + \int_\zeta^T \mathbb{E}\Big[\rho((\sigma_S\sigma_Y)(\bar{Y}^{\zeta, x}_s)-(\sigma_S\sigma_Y)(m_{s-\zeta}(y))) \partial^2_{x, y} P_{T-s}h(\bar{X}^{\zeta, x}_s, \bar{Y}^{\zeta, y}_s)\Big] \, ds. 
\end{align*}

We take $\zeta= \zeta_{n}$, $(x, y) = (\bar{X}_{N_T}, \bar{Y}_{N_T})$ in the previous equality, then multiply it by $\prod_{i=1}^{n+1} \theta_i \I_\seq{N_T=n}$ and finally take expectation. We obtain
\begin{align}
& \mathbb{E}\Big[P_{T-\zeta_n} h(\bar{X}_n, \bar{Y}_n) \prod_{i=1}^{n+1} \theta_i \I_\seq{N_T=n}\Big]\nonumber\\
& =  \mathbb{E}\Big[h(\bar{X}^{\zeta_n, \bar{X}_n}_{T}, \bar{Y}^{\zeta_n, \bar{Y}_n}_{T})  \prod_{i=1}^{n+1}  \theta_i \I_\seq{N_T=n}\Big] \nonumber\\
& + \mathbb{E}\Big[\prod_{i=1}^{n+1}  \theta_i \I_\seq{N_T=n} \int_{\zeta_n}^{T} \frac12 (a_S(\bar{Y}^{\zeta_n, \bar{Y}_n}_s) - a_S(m_{s-\zeta_n}(\bar{Y}_n)) [\partial^2_x P_{T-s}h(\bar{X}^{\zeta_n, \bar{X}_{n}}_s,\bar{Y}^{\zeta_n, \bar{Y}_n}_s) - \partial_x P_{T-s}h(\bar{X}^{\zeta_n, \bar{X}_{n}}_s, \bar{Y}^{\zeta_n, \bar{Y}_n}_s) )]ds \Big] \nonumber \\
& + \mathbb{E}\Big[\prod_{i=1}^{n+1}  \theta_i \I_\seq{N_T=n} \int_{\zeta_n}^{T} \frac12 (a_Y(\bar{Y}^{\zeta_n, \bar{Y}_n}_s) - a_Y(m_{s-\zeta_n}(\bar{Y}_n))) \partial^2_y P_{T-s}h(\bar{X}^{\zeta_n, \bar{X}_n}_s, \bar{Y}^{\zeta_n, \bar{Y}_n}_s) ds\Big] \nonumber  \\
& + \mathbb{E}\Big[\prod_{i=1}^{n+1}  \theta_i \I_\seq{N_T=n} \int_{\zeta_n}^{T}  (b_Y(\bar{Y}^{\zeta_n, \bar{Y}_n}_s) - b_Y(m_{s-\zeta_n}(\bar{Y}_n))) \partial_y P_{T-s}h(\bar{X}^{\zeta_n, \bar{X}_n}_s, \bar{Y}^{\zeta_n, \bar{Y}_n}_s)ds \Big]\nonumber\\
 & + \mathbb{E}\Big[\prod_{i=1}^{n+1}  \theta_i \I_\seq{N_T=n} \int_{\zeta_n}^{T} \rho((\sigma_S\sigma_Y)(\bar{Y}^{\zeta_n, \bar{Y}_n}_s)-(\sigma_S\sigma_Y)(m_{s-\zeta_n}(\bar{Y}_n))) \partial^2_{x,y} P_{T-s}h(\bar{X}^{\zeta_n, \bar{X}_n}_s, \bar{Y}^{\zeta_n, \bar{Y}_n}_s) \, ds\Big]. \label{first:step:induction:hypothesis}
\end{align}

Now, from the very definition of the Markov chain $(\bar{X}_i, \bar{Y}_i)_{0\leq i\leq N_T+1}$ and of the weight sequence $(\theta_i)_{1\leq i \leq N_T+1}$ of Theorem \ref{theorem:probabilistic:formulation}, the first term of the above equality can be written as
\begin{equation}
\mathbb{E}\Big[h(\bar{X}^{\zeta_n, \bar{X}_n}_{T}, \bar{Y}^{\zeta_n, \bar{Y}_n}_{T}) \prod_{i=1}^{n+1} \theta_i \I_\seq{N_T=n}\Big] = \mathbb{E}\Big[h(\bar{X}_{N_T+1}, \bar{Y}_{N_T+1}) \prod_{i=1}^{N_T+1}  \theta_{i} \I_\seq{N_T=n}\Big]. \label{second:expansion:first:term} 
\end{equation}

  We now look at the second, third, fourth and fifth terms. Let us deal with the third and fourth terms. The others are treated in a similar manner and we will omit some technical details. We first take its conditional expectation w.r.t $\left\{ \zeta_1 =t_1, \cdots, \zeta_n = t_n, N_T = n\right\}$ and introduce the measurable function 
 \begin{align*}
& G(t_1, \cdots, t_n, s, T)  :=\mathbb{E}\Big[\prod_{i=1}^{n+1}  \theta_i  \int_{\zeta_{n}}^{T} \Big[ \frac12 (a_Y(\bar{Y}^{\zeta_{n}, \bar{Y}_{n}}_s) - a_Y(m_{s-\zeta_{n}}(\bar{Y}_{n}))) \partial^2_y P_{T-s}h(\bar{X}^{\zeta_{n}, \bar{X}_{n}}_s, \bar{Y}^{\zeta_{n}, \bar{Y}_{n}}_s) \\
& \quad+ (b_Y(\bar{Y}^{\zeta_{n}, \bar{Y}_{n}}_s) - b_Y(m_{s-\zeta_{n}}(\bar{Y}_{n}))) \partial_y P_{T-s}h(\bar{X}^{\zeta_{n}, \bar{X}_{n}}_s, \bar{Y}^{\zeta_{n}, \bar{Y}_{n}}_s) \Big]ds\Big|\zeta_1=t_1, \cdots, \zeta_n = t_n, N_T=n\Big]
 \end{align*}
 
 \noindent which satisfies 
 \begin{align*}
 |G(t_1, \cdots, t_n, s, T)| & \leq C \E\Big[\prod_{i=1}^{n+1} |\theta_i| \Big(1+ |W_{s}-W_{\zeta_{n}}| + | \widetilde{W}_{s}-\widetilde{W}_{\zeta_n}|\Big) \Big| \zeta_1=t_1,\cdots, \zeta_n=t_n, N_T=n\Big] \\
 & \leq C \mathbb{E}\Big[\prod_{i=1}^{n+1} |\theta_i| | \zeta_1=t_1,\cdots, \zeta_n=t_n, N_T=n \Big]
 \end{align*}
 \noindent where we used the boundedness of $a_Y$, the Lipschitz regularity of $b_Y$, the inequalities $\sup_{0\leq t \leq T} |\partial^{\ell}_y P_t h|_\infty \leq C$ for $\ell=1,\,2$ and, for the last inequality the fact that, conditionally on $\left\{ \zeta_1 =t_1, \cdots, \zeta_n = t_n, N_T = n\right\}$, the random variables $(W_{s}-W_{\zeta_{n}},  \widetilde{W}_{s}-\widetilde{W}_{\zeta_n})$ are independent of the sigma-field $\mathcal{G}_n$. Recall now that $\P(N_T=n, \zeta_1\in dt_1, \cdots, \zeta_n \in dt_n) = (1- F(T-t_{n})) \prod_{j=0}^{n-1} f(t_{j+1}-t_j) dt_1, \cdots, dt_n$ on the set $\Delta_n(T)$ so that from Lemma \ref{estimate:theta:time:loca:density} and the estimate \eqref{conditional:lp:moment:space:M:I:N}, we obtain 
 $$
 \E\Big[\int_{\zeta_n}^T|G(\zeta_1, \cdots, \zeta_n, s, T)| \I_\seq{N_T=n} ds\Big] \leq C \int_{t_n}^T \int_{\Delta_n(T)}  \prod_{i=1}^{n} (t_i-t_{i-1})^{-1/2} dt_1\cdots dt_n dt_{n+1} < \infty.
 $$
  Hence, by Lemma \ref{lem:add:new:jumps}, it holds
 \begin{align*}
 & \mathbb{E}\Big[\I_\seq{N_T=n}\prod_{i=1}^{n+1}  \theta_i  \int_{\zeta_{n}}^{T} \Big[ \frac12 (a_Y(\bar{Y}^{\zeta_{n}, \bar{Y}_{n}}_s) - a_Y(m_{s-\zeta_{n}}(\bar{Y}_{n}))) \partial^2_y P_{T-s}h(\bar{X}^{\zeta_{n}, \bar{X}_{n}}_s, \bar{Y}^{\zeta_{n}, \bar{Y}_{n}}_s) \\
& \quad+ (b_Y(\bar{Y}^{\zeta_{n}, \bar{Y}_{n}}_s) - b_Y(m_{s-\zeta_{n}}(\bar{Y}_{n}))) \partial_y P_{T-s}h(\bar{X}^{\zeta_{n}, \bar{X}_{n}}_s, \bar{Y}^{\zeta_{n}, \bar{Y}_{n}}_s) \Big]ds \Big]\\
& = \mathbb{E}\Big[ \I_\seq{N_T=n}\int_{\zeta_n}^T G(\zeta_1,\cdots, \zeta_n, s, T) \, ds \Big]\\
& = \mathbb{E}\Big[\prod_{i=1}^{n} \theta_i  (1-F(T-\zeta_{n+1}))^{-1} (f(\zeta_{n+1}-\zeta_{n}))^{-1}  \Big[\frac12 (a_Y(\bar{Y}_{n+1}) - a_Y(m_{n})) \mathcal{D}^{(2,2)}_{n+1} P_{T-\zeta_{n+1}}h(\bar{X}_{n+1}, \bar{Y}_{n+1}) \\
& \quad+ (b_Y(\bar{Y}_{n+1}) - b_Y(m_{n})) \mathcal{D}^{(2)}_{n+1} P_{T-s}h(\bar{X}_{n+1}, \bar{Y}_{n+1})\Big] \I_\seq{N_T=n+1}\Big].
 \end{align*}
 
 Finally, we take the conditional expectation $\mathbb{E}_{n, n+1}[.]$ inside the above expectation and then employ the IBP formula \eqref{duality:formula}, two times w.r.t. the diffusion coefficient and one time w.r.t the drift coefficient as done before. We obtain
 \begin{align*}
 \mathbb{E}\Big[& \prod_{i=1}^{n} \theta_i (1-F(T-\zeta_{n+1}))^{-1} (f(\zeta_{n+1}-\zeta_{n}))^{-1}  \Big[\frac12 (a_Y(\bar{Y}_{n+1}) - a_Y(m_{n})) \mathcal{D}^{(2,2)}_{n+1} P_{T-\zeta_{n+1}}h(\bar{X}_{n+1}, \bar{Y}_{n+1}) \\
& \quad+ (b_Y(\bar{Y}_{n+1}) - b_Y(m_{n})) \mathcal{D}^{(2)}_{n+1} P_{T-\zeta_{n+1}}h(\bar{X}_{n+1}, \bar{Y}_{n+1})\Big] \I_\seq{N_T=n+1}\Big] \\
& = \mathbb{E}\Big[ \prod_{i=1}^{n} \theta_i  (1-F(T-\zeta_{n+1}))^{-1}  (f(\zeta_{n+1}-\zeta_{n}))^{-1} [\mathcal{I}^{(2,2)}_{n+1}(c^{n+1}_Y) +\mathcal{I}^{(2)}_{n+1}(b^{n+1}_Y) ] P_{T-\zeta_{n+1}}h(\bar{X}_{n+1}, \bar{Y}_{n+1})  \I_\seq{N_T=n+1}\Big].
 \end{align*}
 
 In a completely analogous manner, we derive
 \begin{align*}
 \mathbb{E}\Big[& \prod_{i=1}^{n+1}  \theta_i \I_\seq{N_T=n} \int_{\zeta_n}^{T} \frac12 (a_S(\bar{Y}^{\zeta_n, \bar{Y}_n}_s) - a_S(m_{s-\zeta_n}(\bar{Y}_n)) [\partial^2_x P_{T-s}h(\bar{X}^{\zeta_n, \bar{X}_{n}}_s,\bar{Y}^{\zeta_n, \bar{Y}_n}_s) - \partial_x P_{T-s}h(\bar{X}^{\zeta_n, \bar{X}_{n}}_s, \bar{Y}^{\zeta_n, \bar{Y}_n}_s) )]ds \Big] \\
 & = \mathbb{E}\Big[ \prod_{i=1}^{n} \theta_i  (1-F(T-\zeta_{n+1}))^{-1}  (f(\zeta_{n+1}-\zeta_{n}))^{-1} [\mathcal{I}^{(1,1)}_{n+1}(c^{n+1}_S) - \mathcal{I}^{(1)}_{n+1}(c^{n+1}_S) ] P_{T-\zeta_{n+1}}h(\bar{X}_{n+1}, \bar{Y}_{n+1})\I_\seq{N_T=n+1} \Big]
 \end{align*}
 \noindent and
 \begin{align*}
 \mathbb{E}\Big[& \prod_{i=1}^{n+1}  \theta_i \I_\seq{N_T=n} \int_{\zeta_n}^{T} \rho((\sigma_S\sigma_Y)(\bar{Y}^{\zeta_n, \bar{Y}_n}_s)-(\sigma_S\sigma_Y)(m_{s-\zeta_n}(\bar{Y}_n))) \partial^2_{x,y} P_{T-s}h(\bar{X}^{\zeta_n, \bar{X}_n}_s, \bar{Y}^{\zeta_n, \bar{Y}_n}_s) \, ds\Big] \\
 & = \mathbb{E}\Big[ \prod_{i=1}^{n} \theta_i (1-F(T-\zeta_{n+1}))^{-1}  (f(\zeta_{n+1}-\zeta_{n}))^{-1} \mathcal{I}^{(1,2)}_{n+1}(c^{n+1}_{Y,S})P_{T-\zeta_{n+1}}h(\bar{X}_{n+1}, \bar{Y}_{n+1})  \I_\seq{N_T=n+1} \Big].
 \end{align*}
 
 Summing the three previous identities, we obtain that the sum of the second, third, fourth and fifth term in the right-hand side of \eqref{first:step:induction:hypothesis} is equal to
 \begin{align*}
 \mathbb{E}\Big[ & \prod_{i=1}^{n} \theta_i  (1-F(T-\zeta_{n+1}))^{-1}  (f(\zeta_{n+1}-\zeta_{n}))^{-1} \\
 & \quad \times \Big[ \mathcal{I}^{(1,1)}_{i+1}(a^{n+1}_S) -  \mathcal{I}^{(1)}_{n+1}(a^{n+1}_S)  + \mathcal{I}^{(2,2)}_{n+1}(a^{n+1}_Y) +  \mathcal{I}^{(2)}_{n+1}(b^{n+1}_Y) +\mathcal{I}^{(1,2)}_{n+1}(a^{n+1}_{Y, S}) \Big]  P_{T-\zeta_{n+1}}h(\bar{X}_{n+1}, \bar{Y}_{n+1}) \I_\seq{N_T=n+1}  \Big]\\
 & =  \mathbb{E}\Big[ \prod_{i=1}^{n+2} \theta_i P_{T-\zeta_{n+1}}h(\bar{X}_{n+1}, \bar{Y}_{n+1}) \I_\seq{N_T=n+1}  \Big]
 \end{align*}
 \noindent where we used the very definitions \eqref{definition:first:theta:theorem:probabilistic:representation} and \eqref{definition:theta:theorem:probabilistic:representation} of the weights $(\theta_{i})_{1\leq i \leq N_T+1}$ on the set $\left\{ N_T = n+1\right\}$. This concludes the proof of \eqref{induction:step:probabilistic:representation} at step $n+1$. 
 
 To conclude it remains to prove the absolute convergence of the first sum and the convergence to zero of the last term in \eqref{induction:step:probabilistic:representation}. These two results follow directly from the boundedness of $h$ and the general estimates on the product of weights established in Lemma \ref{estimate:theta:time:loca:density}.
 
Indeed, from Lemma \ref{estimate:theta:time:loca:density}, the estimate \eqref{conditional:lp:moment:space:M:I:N}, the tower property of conditional expectation and the identity \eqref{probabilistic:representation:time:integrals}, we obtain
\begin{align*}
\E\Big[  \left|h(\bar{X}_{N_T+1}, \bar{Y}_{N_T+1})\right| \prod_{i=1}^{N_T+1}  \left|\theta_i\right|  \I_\seq{N_T=j} \Big]\Big| & \leq  C^j |h|_\infty \mathbb{E}\Big[ (1-F(T-\zeta_{j}))^{-1}\prod_{i=1}^{j} (f(\zeta_{i}-\zeta_{i-1}))^{-1} (\zeta_{i}-\zeta_{i-1})^{-\frac12}  \I_\seq{N_T=j}  \Big] \\
 & = C^j |h|_{\infty} \int_{\Delta_j(T)} \prod_{i=1}^{j}(s_{i}-s_{i-1})^{-\frac12} \, d\s_j \\
 & = C^{j} |h|_{\infty} T^{\frac{j}{2}} \frac{\Gamma^{j}(1/2)}{\Gamma(1+j/2)}
\end{align*}
\noindent which in turn yields
$$
\sum_{j=0}^{n-1} \E\Big[ \left|h(\bar{X}_{N_T+1}, \bar{Y}_{N_T+1})\right| \prod_{i=1}^{N_T+1}  \left| \theta_i \right|  \I_\seq{N_T=j}\Big] \Big| \leq |h|_\infty \sum_{n\geq0} \frac{(CT^{1/2})^{n}}{\Gamma(1+n/2)} = |h|_\infty E_{1/2, 1}(C T^{1/2})
$$

\noindent so that the series converge absolutely. Similarly, 
$$
\Big|\E\Big[ P_{T-\zeta_{n}}h(\bar{X}_{n}, \bar{Y}_n)\prod_{i=1}^{n+1}  \theta_i \I_\seq{N_T=n}\Big] \Big| \leq C^{n} |h|_\infty  T^{\frac{n}{2}} \frac{\Gamma^{n}(1/2)}{\Gamma(1+n/2)}
$$
\noindent so that the remainder indeed vanishes as $n$ goes to infinity. We thus conclude
\begin{equation}
P_Th(x_0,y_0)= \sum_{n\geq0} \E\Big[ h(\bar{X}_{N_T+1}, \bar{Y}_{N_T+1})\prod_{i=1}^{N_T+1}  \theta_i  \I_\seq{N_T=n}\Big] = \mathbb{E}\Big[ h(\bar{X}_{N_T+1}, \bar{Y}_{N_T+1})\prod_{i=1}^{N_T+1}  \theta_i\Big] \label{semigroup:representation:before:step:2}
\end{equation}

\noindent for any $h\in \mathcal{C}^2_b(\mathbb{R}^2)$. We eventually extend the above representation formula to any bounded and continuous function $h$ using a standard approximation argument. The remaining technical details are omitted.
 
 \bigskip
 
\emph{Step 2: Extension to measurable map $h$ satisfying the growth assumption \eqref{growth:assumption:test:function}}\\

We first extend the previous result to any bounded and measurable $h$. This follows from a monotone class argument that we now detail.
 
 Let us first consider $h \in \mathcal{C}_b(\mathbb{R}^2)$. From Fubini's theorem, it holds
\begin{align*}
\mathbb{E}\Big[ h(\bar{X}_{n+1}, \bar{Y}_{n+1})& \prod_{i=1}^{n+1} \theta_{i} \Big| N_T=n, \zeta^{n} \Big] = \int_{\mathbb{R}^2} h(x, y) \, \mathbb{E}\Big[\bar{p}(T-\zeta_n, \bar{X}_{n}, \bar{Y}_{n}, x, y)  \prod_{i=1}^{n+1} \theta_{i}(\bar{X}_{{i-1}}, \bar{Y}_{{i-1}}, \bar{X}_{{i}}, \bar{Y}_{{i}}, \zeta^{n})  \Big| N_T=n, \zeta^{n} \Big] \, dx dy
\end{align*}

\noindent which can be justified as follows. From Lemma \ref{estimate:theta:time:loca:density}, Lemma \ref{lemma:semigroup:property} and the upper-bound estimate \eqref{upper:bound:transition:density:approximation:semigroup}, it holds
\begin{align*}
\Big| \mathbb{E}\Big[& \bar{p}(T-\zeta_n, \bar{X}_{n}, \bar{Y}_{n}, x, y)   \prod_{i=1}^{n+1} \theta_{i}(\bar{X}_{{i-1}}, \bar{Y}_{{i-1}}, \bar{X}_{{i}}, \bar{Y}_{{i}}, \zeta^{n}) \Big| N_T=n, \zeta^{n} \Big] \Big|\\
& \leq \int_{(\mathbb{R}^2)^{n}} \bar{p}(T-\zeta_n, x_n, y_n, x, y) \prod_{i=1}^{n}| \theta_{i}(x_{i-1}, y_{i-1}, x_i, y_i, \zeta^{n})|\bar{p}(\zeta_i-\zeta_{i-1}, x_{i-1},y_{i-1}, x_i, y_i) \,  d\x_n d\y_n \\
& \leq (C_T)^n \int_{(\mathbb{R}^2)^{n}}  \bar{q}_{4\kappa}(T-\zeta_n, x_n, y_n, x, y) \prod_{i=1}^{n} (f(\zeta_i -\zeta_{i-1}))^{-1}(\zeta_i -\zeta_{i-1})^{-\frac12} \bar{q}_{4\kappa}(\zeta_{i}-\zeta_{i-1}, x_{i-1}, y_{i-1}, x_i, y_i) \,   d\x_n d\y_n \\
&\leq  (C_T)^n  \prod_{i=1}^{n} (f(\zeta_i -\zeta_{i-1}))^{-1}(\zeta_i -\zeta_{i-1})^{-\frac12} \bar{q}_{c}(T, x_0, y_0, x, y)
\end{align*}

\noindent for some $c:=c(T, b_Y, \kappa)\geq 4\kappa$. Hence, from \eqref{semigroup:representation:before:step:2} and again Fubini's theorem, justified by the previous estimate and the fact that $\mathbb{E}\Big[ (C_T)^{N_T} \prod_{i=1}^{N_T} (f(\zeta_i -\zeta_{i-1}))^{-1}(\zeta_i - \zeta_{i-1})^{-\frac12} \Big] < \infty$, one has\
\begin{align}
P_T h(x_0, y_0) & = \int_{\mathbb{R}^2} h(x, y)  \mathbb{E}\Big[\bar{p}(T-\zeta_{N_T}, \bar{X}_{N_T}, \bar{Y}_{N_T}, x, y)  \prod_{i=1}^{N_T+1} \theta_i\Big] \, dx dy \label{semigroup:write:as:integral:density}
\end{align}

\noindent for any $h \in \mathcal{C}_b(\mathbb{R}^2)$. Moreover, from the previous computations, the following upper-bound holds
\begin{align}
 \Big| & \mathbb{E}\Big[\bar{p}(T-\zeta_{N_T}, \bar{X}_{N_T}, \bar{Y}_{N_T}, x, y)  \prod_{i=1}^{N_T+1} \theta_{i}\Big] \Big|\nonumber \\
 & = \Big| \sum_{n\geq0} \int_{\Delta_n(T)} \mathbb{E}\Big[\bar{p}(T-s_{n}, \bar{X}_{n}, \bar{Y}_{n}, x, y)  \prod_{i=1}^{n+1} \theta_{i} \, \Big| N_T=n, \zeta^n=(s_1, \cdots, s_n)\Big] \, (1-F(T-s_n)) \prod_{i=1}^{n} f(s_{i}-s_{i-1})\, d\s_n \Big| \nonumber\\
 & \leq \Big(\sum_{n\geq0}  \int_{\Delta_n(T)} (C_T)^n \prod_{i=1}^{n} (s_i -s_{i-1})^{-\frac12} \, d\s_n\Big)   \bar{q}_{c}(T, x_0, y_0, x, y) \nonumber \\
 & = E_{1/2, 1}(CT^{1/2}) \bar{q}_{c}(T, x_0, y_0, x, y). \label{upper:bound:density:qc}
\end{align}

 It now follows from \eqref{semigroup:write:as:integral:density} and a monotone class argument that the probabilistic representation formula \eqref{probabilistic:formulation:theorem} is valid for any real-valued bounded and measurable map $h$ defined over $\mathbb{R}^2$. The extension to any measurable map $h$ satisfying the growth assumption: $|h(x, y)|\leq C \exp(\gamma (|x|^2 + |y|^2))$ for any $0\leq \gamma < (2\kappa)^{-1}$, follows from the integral representation \eqref{semigroup:write:as:integral:density}, the upper-bound \eqref{upper:bound:density:qc} combined with a standard approximation argument. Remaining technical details are omitted. 

\bigskip

\emph{Step 3: Finite $L^p(\mathbb{P})$-moment for the probabilistic representation}\\

If $N$ is a renewal process with $Beta(\alpha,1)$ jump times then $f(s_{i}-s_{i-1}) = \frac{1-\alpha}{\bar{\tau}^{1-\alpha}} \frac{1}{(s_i-s_{i-1})^{\alpha}} \textbf{1}_{[0, \bar{\tau}]}(s_i-s_{i-1})$ and $1-F(T-s_n) = 1-\Big(\frac{T-s_n}{\bar{\tau}}\Big)^{1-\alpha} \geq 1- (\frac{T}{\bar{\tau}})^{1-\alpha}$, similarly to step 2, by Fubini's theorem, we get
\begin{align*}
\mathbb{E}\Big[ |h(\bar{X}_{n+1}, \bar{Y}_{n+1})|^p& \prod_{i=1}^{n+1} |\theta_{i}|^p \Big| N_T=n, \zeta^{n} \Big] \\
& = \int_{\mathbb{R}^2} |h(x, y)|^p \, \mathbb{E}\Big[\bar{p}(T-\zeta_n, \bar{X}_{n}, \bar{Y}_{n}, x, y)  \prod_{i=1}^{n+1} |\theta_{i}(\bar{X}_{{i-1}}, \bar{Y}_{{i-1}}, \bar{X}_{{i}}, \bar{Y}_{{i}}, \zeta^{n})|^p  \Big| N_T=n, \zeta^{n} \Big] \, dx dy.
\end{align*}
 The above formula is justified by Lemma \ref{estimate:theta:time:loca:density} and Lemma \ref{lemma:semigroup:property} which yield
 \begin{align*}
 \mathbb{E}\Big[& \bar{p}(T-\zeta_n, \bar{X}_{n}, \bar{Y}_{n}, x, y)   \prod_{i=1}^{n+1} |\theta_{i}(\bar{X}_{{i-1}}, \bar{Y}_{{i-1}}, \bar{X}_{{i}}, \bar{Y}_{{i}}, \zeta^{n})|^p \Big| N_T=n, \zeta^{n} \Big] \\
& \leq \int_{(\mathbb{R}^2)^{n}} \bar{p}(T-\zeta_n, x_n, y_n, x, y) \prod_{i=1}^{n}| \theta_{i}(x_{i-1}, y_{i-1}, x_i, y_i, \zeta^{n})|^p \bar{p}(\zeta_i-\zeta_{i-1}, x_{i-1},y_{i-1}, x_i, y_i) \,  d\x_n d\y_n \\
& \leq C^n \int_{(\mathbb{R}^2)^{n}}  \bar{q}_{4\kappa}(T-\zeta_n, x_n, y_n, x, y) \prod_{i=1}^{n} (f(\zeta_i -\zeta_{i-1}))^{-p}(\zeta_i -\zeta_{i-1})^{-\frac{p}{2}} \bar{q}_{4\kappa}(\zeta_{i}-\zeta_{i-1}, x_{i-1}, y_{i-1}, x_i, y_i) \,   d\x_n d\y_n \\
&\leq  C^n  \prod_{i=1}^{n} (\zeta_i -\zeta_{i-1})^{\alpha p}(\zeta_i -\zeta_{i-1})^{-\frac{p}{2}} \bar{q}_{c}(T, x_0, y_0, x, y)
\end{align*}

\noindent for some $c:=c(T, b_Y, \kappa)\geq 4\kappa$. Now, using the fact that $\mathbb{E}\Big[C^n \prod_{i=1}^{N_T} (\zeta_i -\zeta_{i-1})^{\alpha p } (\zeta_i - \zeta_{i-1})^{-\frac{p}{2}} \Big] < \infty$ as soon as $p(\frac{1}{2}-\alpha) < 1-\alpha $ and that $h\in \mathcal{B}_\gamma(\mathbb{R}^2)$, from the previous computation, we obtain
\begin{align*}
 & \mathbb{E}[|h(\bar{X}_{N_T+1}, \bar{Y}_{N_T+1})|^p   \prod_{i=1}^{N_T+1} |\theta_{i}|^p] \\
 & = \mathbb{E}\Big[\mathbb{E}\Big[ |h(\bar{X}_{n+1}, \bar{Y}_{n+1})|^p \prod_{i=1}^{n+1} |\theta_{i}|^p \Big| N_T, \zeta^{n} \Big]\Big]\\ 
 & \leq \mathbb{E}\Big[C^n \prod_{i=1}^{N_T} (\zeta_i -\zeta_{i-1})^{\alpha p } (\zeta_i - \zeta_{i-1})^{-\frac{p}{2}} \Big] \int_{\mathbb{R}^2} e^{\gamma p (|x|^2 + |y|^2)} \bar{q}_{c}(T, x_0, y_0, x, y)\, dx dy.
\end{align*}

 To conclude the proof, it suffices to note that the above space integral is finite as soon as $0\leq \gamma p < c^{-1}$.


\subsection{Proof of Lemma \ref{lem:transfer:derivative}}\label{proof:lem:transfer:derivative} 

Since $h\in \mathcal{C}^{1}_{p}(\mathbb{R}^2)$ and $\E_{i,n}\Big[|\bar{X}_{i+1}|^q + |\bar{Y}_{i+1}|^q\Big]< \infty$ a.s., for any $q\geq1$, under \A{AR}, one may differentiate under the (conditional) expectation and deduce that $(x, y) \mapsto \mathbb{E}_{i, n}\Big[h(\bar{X}_{i+1}, \bar{Y}_{i+1}) \theta_{i+1} | (\bar{X}_i, \bar{Y}_i)=(x, y)\Big] \in \mathcal{C}^1_p(\mathbb{R}^2)$ for any $i \in \left\{0, \cdots, n\right\}$. The rest of the proof is divided into three parts.

\bigskip

\emph{Step 1: proofs of \eqref{transfer:derivative:x:first:interval} and \eqref{transfer:derivative:x:next:interval}} \\

The transfer of derivatives formulae \eqref{transfer:derivative:x:first:interval} and \eqref{transfer:derivative:x:next:interval} are easily obtained by differentiating under expectation (which is allowed by the polynomial growth at infinity of $h$) noting from the definition of the Markov chain $\bar{X}$ that $\partial_{s_0} \bar{X}_1 = \partial_{s_0} \ln(s_0) = \frac{1}{s_0}$ and $\partial_{\bar{X}_i} h(\bar{X}_{i+1}, \bar{Y}_{i+1}) = \partial_{\bar{X}_{i+1}} h(\bar{X}_{i+1}, \bar{Y}_{i+1}) \partial_{\bar{X}_i} \bar{X}_{i+1} = \partial_{\bar{X}_{i+1}} h(\bar{X}_{i+1}, \bar{Y}_{i+1})$. Observe as well that from \eqref{chain:rule:formula:integral:operators:first:coordinate:deriv}, the fact that $\partial_{\bar{X}_i} c^{i+1}_S = \partial_{\bar{X}_i} c^{i+1}_Y = \partial_{\bar{X}_i} b^{i+1}_Y  = \partial_{\bar{X}_i} c^{i+1}_{Y, S} = \partial_{\bar{X}_i}  \mathcal{I}_{i+1}^{(1)}(1) = \partial_{\bar{X}_i}  \mathcal{I}_{i+1}^{(2)}(1)  =0 $ and the very definition of the random variables $(\theta_{i})_{1\leq i \leq n+1}$, one has $\partial_{\bar{X}_i} \theta_{i+1} = 0$. This gives the identities \eqref{transfer:derivative:x:first:interval} and \eqref{transfer:derivative:x:next:interval}.

\bigskip

\emph{Step 2: proofs of \eqref{transfer:derivative:y:next:interval} and \eqref{transfer:derivative:y:last:interval}} \\

The proofs of \eqref{transfer:derivative:y:next:interval} and \eqref{transfer:derivative:y:last:interval} are more involved. Let us prove \eqref{transfer:derivative:y:next:interval}. We proceed by considering the difference between the term appearing on the left-hand side and the first two terms appearing on the right-hand side of \eqref{transfer:derivative:y:next:interval}. On the one hand, using the IBP formula \eqref{duality:formula} and \eqref{partial:flow:deriv:Y:bar:process}, we get
\begin{align*}
\partial_{\bar{Y}_i} \mathbb{E}_{i,n}\Big[h(\bar{X}_{i+1},\bar{Y}_{i+1}) \theta_{i+1}\Big] & = \mathbb{E}_{i,n}\Big[ \partial_{\bar{X}_{i+1}}h(\bar{X}_{i+1}, \bar{Y}_{i+1}) \partial_{\bar{Y}_i} \bar{X}_{i+1} \theta_{i+1}\Big]  + \mathbb{E}_{i, n}\Big[ \partial_{\bar{Y}_{i+1}}h(\bar{X}_{i+1}, \bar{Y}_{i+1}) \partial_{\bar{Y}_i}\bar{Y}_{i+1} \theta_{i+1} \Big] \\
& \quad + \mathbb{E}_{i, n}\Big[h(\bar{X}_{i+1}, \bar{Y}_{i+1}) \partial_{\bar{Y}_{i}}\theta_{i+1} \Big] \\
& = \mathbb{E}_{i,n}\Big[ h(\bar{X}_{i+1}, \bar{Y}_{i+1}) \Big[\mathcal{I}^{(1)}_{i+1}(\partial_{\bar{Y}_i}\bar{X}_{i+1} \theta_{i+1}) + \mathcal{I}^{(2)}_{i+1}(m'_i \theta_{i+1}) \Big] \Big] \\
& \quad  + \mathbb{E}_{i,n}\Big[h(\bar{X}_{i+1}, \bar{Y}_{i+1}) \mathcal{I}^{(2)}_{i+1}\Big(\sigma'_{Y,i} \Big(\rho_i Z^{1}_{i+1} + \sqrt{1-\rho_i^2} Z^2_{i+1}\Big)  \theta_{i+1}\Big) \Big] \\
& \quad + \mathbb{E}_{i,n}\Big[h(\bar{X}_{i+1}, \bar{Y}_{i+1}) \mathcal{I}^{(2)}_{i+1}\Big(\sigma_{Y,i}\frac{\rho'_i}{1-\rho_i^2} \Big(\sqrt{1-\rho^2_i} Z^{1}_{i+1} -\rho_i Z^2_{i+1}\Big)  \theta_{i+1}\Big) \Big] \\
& \quad +  \mathbb{E}_{i, n}\Big[h(\bar{X}_{i+1}, \bar{Y}_{i+1}) \partial_{\bar{Y}_{i}}\theta_{i+1} \Big].  
\end{align*}

On the other hand, again from the IBP formula \eqref{duality:formula}, we obtain
\begin{align*}
\mathbb{E}_{i, n}\Big[ \partial_{\bar{X}_{i+1}}h(\bar{X}_{i+1}, \bar{Y}_{i+1}) \rtheta^{e,X}_{i+1} \Big] &+ \mathbb{E}_{i, n}\Big[ \partial_{\bar{Y}_{i+1}}h(\bar{X}_{i+1}, \bar{Y}_{i+1}) \rtheta^{e,Y}_{i+1} \Big] + \mathbb{E}_{i, n}\Big[h(\bar{X}_{i+1}, \bar{Y}_{i+1}) \rtheta^{c}_{i+1} \Big] \\
& = \mathbb{E}_{i, n}\Big[h(\bar{X}_{i+1}, \bar{Y}_{i+1}) \Big[\mathcal{I}^{(1)}_{i+1}(\rtheta^{e,X}_{i+1}) + \mathcal{I}^{(2)}_{i+1}(\rtheta^{e,Y}_{i+1})+ \rtheta^{c}_{i+1}\Big] \Big].
\end{align*}

Combining the two previous identities, we see that the difference 
$$
\partial_{\bar{Y}_i} \mathbb{E}_{i,n}\Big[h(\bar{X}_{i+1},\bar{Y}_{i+1}) \theta_{i+1}\Big] - \Big(\mathbb{E}_{i, n}\Big[ \partial_{\bar{X}_{i+1}}h(\bar{X}_{i+1}, \bar{Y}_{i+1}) \rtheta^{e,X}_{i+1} \Big]  +  \mathbb{E}_{i, n}\Big[ \partial_{\bar{Y}_{i+1}}h(\bar{X}_{i+1}, \bar{Y}_{i+1}) \rtheta^{e,Y}_{i+1} \Big] +\mathbb{E}_{i, n}\Big[h(\bar{X}_{i+1}, \bar{Y}_{i+1}) \rtheta^{c}_{i+1} \Big]  \Big)
$$
can be written as
\begin{align}
&\mathbb{E}\Big[ h(\bar{X}_{i+1}, \bar{Y}_{i+1}) \mathcal{I}^{(2)}_{i+1}\Big(m'_i \theta_{i+1} - \rtheta^{e, Y}_{i+1}\Big) \Big] +  \mathbb{E}\Big[h(\bar{X}_{i+1}, \bar{Y}_{i+1})\partial_{\bar{Y}_i} \theta_{i+1}\Big] - \mathbb{E}_{i,n}\Big[ h(\bar{X}_{i+1}, \bar{Y}_{i+1}) \mathcal{I}^{(1)}_{i+1}( \rtheta^{e,X}_{i+1})\Big] \nonumber \\
& +  \mathbb{E}_{i,n}\Big[ h(\bar{X}_{i+1}, \bar{Y}_{i+1}) \mathcal{I}^{(1)}_{i+1}(\partial_{\bar{Y}_i}\bar{X}_{i+1} \theta_{i+1})\Big]  + \mathbb{E}\Big[h(\bar{X}_{i+1}, \bar{Y}_{i+1})\mathcal{I}^{(2)}_{i+1}\Big(\sigma'_{Y,i} \Big(\rho_i Z^{1}_{i+1} + \sqrt{1-\rho_i^2} Z^2_{i+1}\Big)  \theta_{i+1}\Big) \Big] \nonumber \\
& + \mathbb{E}\Big[h(\bar{X}_{i+1}, \bar{Y}_{i+1})\mathcal{I}^{(2)}_{i+1}\Big(\sigma_{Y,i} \frac{\rho'_i}{\sqrt{1-\rho_i^2}} \Big(\sqrt{1-\rho^2_i} Z^{1}_{i+1} - \rho_i Z^2_{i+1}\Big)  \theta_{i+1}\Big) \Big] \label{difference:terms:after:ibp}\\
&-  \mathbb{E}\Big[h(\bar{X}_{i+1}, \bar{Y}_{i+1}) \rtheta^{c}_{i+1}\Big].\nonumber
\end{align}

Before proceeding, we provide the explicit expression for the quantity $\partial_{\bar{Y}_i} \theta_{i+1}$. Using the chain rule formula of Lemma \ref{lem:chain:rule:formula}, after some standard but cumbersome computations, we obtain

\begin{align*}
	\partial_{\bar{Y}_i} \theta_{i+1} & = (f(\zeta_{i+1}-\zeta_i))^{-1}\Big[   \mathcal{I}^{(1,1)}_{i+1}(\partial_{\bar{Y}_i}c^{i+1}_S)  -   \mathcal{I}^{(1)}_{i+1}(\partial_{\bar{Y}_i}c^{i+1}_S) + \mathcal{I}^{(2,2)}_{i+1}(\partial_{\bar{Y}_i}c^{i+1}_Y) +  \mathcal{I}^{(2)}_{i+1}(\partial_{\bar{Y}_i}b^{i+1}_Y) +  \mathcal{I}^{(1,2)}_{i+1}(\partial_{\bar{Y}_i}c^{i+1}_{Y, S}) \\
	& -  \Big(  \frac{\sigma'_{S,i}}{\sigma_{S,i}} \Big( 2\mathcal{I}^{(1,1)}_{i+1}(c^{i+1}_S)  - \mathcal{I}^{(1)}_{i+1}(c^{i+1}_S) +  \mathcal{I}^{(1,2)}_{i+1}(c^{i+1}_{Y, S})  \Big)  +  \Big(\frac{\sigma'_{Y,i}}{\sigma_{Y,i}} - \frac{\rho'_i \rho_i}{1-\rho_i^2}\Big) \Big(2 \mathcal{I}^{(2,2)}_{i+1}(c^{i+1}_Y)   +  \mathcal{I}^{(2)}_{i+1}(b^{i+1}_Y) + \mathcal{I}^{(1,2)}_{i+1}(c^{i+1}_{Y, S})  \Big) \Big) \\
	&  + \frac{\rho'_i}{1-\rho_i^2} \frac{\sigma_{Y,i}}{\sigma_{S,i}}\Big( \mathcal{I}^{(1,2)}_{i+1}(c^{i+1}_{S}) + \mathcal{I}^{(2,1)}_{i+1}(c^{i+1}_{S}) - \mathcal{I}^{(2)}_{i+1}(c^{i+1}_{S}) \Big) \Big].
\end{align*}

Also, after some simple algebraic simplifications using the definitions of $\rtheta^{e,Y}_{i+1}$ and $\rtheta^{e,X}_{i+1}$ in \eqref{transfer:derivative:y:next:interval}, one obtains
$$
\mathcal{I}^{(2)}_{i+1}\Big(m'_i \theta_{i+1} - \rtheta^{e, Y}_{i+1}\Big) = -(f(\zeta_{i+1}-\zeta_i))^{-1}  \Big[ \mathcal{I}^{(2,2)}_{i+1}\Big(\partial_{\bar{Y}_i} c^{i+1}_Y\Big) + \mathcal{I}^{(1,2)}_{i+1}\Big(\partial_{\bar{Y}_i} c^{i+1}_{Y,S}\Big) \Big]
$$ 
\noindent and
$$
\mathcal{I}^{(1)}_{i+1}( \rtheta^{e,X}_{i+1}) = (f(\zeta_{i+1}-\zeta_i))^{-1} \mathcal{I}^{(1,1)}_{i+1}\Big(\partial_{\bar{Y}_i}c^{i+1}_S \Big).
$$

Combining the three previous identities and gathering similar terms, we obtain

\begin{align}\label{one:part:theta:c}
&\qquad\qquad 	\mathcal{I}^{(2)}_{i+1}\Big(m'_i \theta_{i+1} - \rtheta^{e, Y}_{i+1}\Big) + \partial_{\bar{Y}_i} \theta_{i+1}   - \mathcal{I}^{(1)}_{i+1}( \rtheta^{e,X}_{i+1}) = (f(\zeta_{i+1}-\zeta_i))^{-1}\Big[   -   \mathcal{I}^{(1)}_{i+1}(\partial_{\bar{Y}_i}c^{i+1}_S) +  \mathcal{I}^{(2)}_{i+1}(\partial_{\bar{Y}_i}b^{i+1}_Y) \nonumber \\
 &\qquad \;\;\;    -  \Big(  \frac{\sigma'_{S,i}}{\sigma_{S,i}} \Big( 2\mathcal{I}^{(1,1)}_{i+1}(c^{i+1}_S)  - \mathcal{I}^{(1)}_{i+1}(c^{i+1}_S) +  \mathcal{I}^{(1,2)}_{i+1}(c^{i+1}_{Y, S})  \Big)  +  \Big(\frac{\sigma'_{Y,i}}{\sigma_{Y,i}} - \frac{\rho'_i \rho_i}{1-\rho_i^2}\Big) \Big(2 \mathcal{I}^{(2,2)}_{i+1}(c^{i+1}_Y)   +  \mathcal{I}^{(2)}_{i+1}(b^{i+1}_Y) + \mathcal{I}^{(1,2)}_{i+1}(c^{i+1}_{Y, S})  \Big) \Big) \\
	& \qquad \;\;\;   + \frac{\rho'_i}{1-\rho_i^2} \frac{\sigma_{Y,i}}{\sigma_{S,i}}\Big( \mathcal{I}^{(1,2)}_{i+1}(c^{i+1}_{S}) + \mathcal{I}^{(2,1)}_{i+1}(c^{i+1}_{S}) - \mathcal{I}^{(2)}_{i+1}(c^{i+1}_{S}) \Big) \Big].\nonumber
\end{align}

The previous identity will be used in the next step of the proof. Coming back to \eqref{difference:terms:after:ibp} and using the definition of the weight $\rtheta^{c}_{i+1}$ allows to conclude the proof of the identity \eqref{transfer:derivative:y:next:interval}.

\bigskip

\noindent \emph{Step 3:  The weight sequences $(\rtheta^{e, Y}_{i})_{1\leq i \leq n+1}$, $(\rtheta^{e, X}_{i})_{1\leq i \leq n+1}$ and $(\rtheta^{c}_{i})_{1\leq i \leq n+1}$ and the related spaces $\mathbb{M}_{i,n}(\bar{X}, \bar{Y}, \ell/2)$, $\ell \in \mathbb{Z}$.}

\bigskip

In this last step, we prove the last statement of Lemma \ref{lem:transfer:derivative} concerning the weight sequences $(\rtheta^{e, Y}_{i})_{1\leq i \leq n+1}$, $(\rtheta^{e, X}_{i})_{1\leq i \leq n+1}$ and $(\rtheta^{c}_{i})_{1\leq i \leq n+1}$. 

Following similar lines of reasonings as those used in the proof of Lemma \ref{estimate:theta:time:loca:density}, namely using the fact that $d^{i+1}_S, \, d^{i+1}_Y, \, d^{i+1}_{Y, S} \in \mathbb{M}_{i,n}(\bar{X}, \bar{Y}, 1/2)$ and $\mathcal{D}^{(1)}_{i+1} d^{i+1}_S, \, \mathcal{D}^{(1,1)}_{i+1} d^{i+1}_S, \, \mathcal{D}^{(2)}_{i+1} d^{i+1}_Y, \, \mathcal{D}^{(2,2)}_{i+1} d^{i+1}_Y, \, \mathcal{D}^{(1)}_{i+1} d^{i+1}_Y, \, \mathcal{D}^{(2)}_{i+1} d^{i+1}_Y, \, \mathcal{D}^{(1,2)}_{i+1} d^{i+1}_Y$, $e^{Y, i+1}_S$, $\mathcal{D}^{(1)}_{i+1} e^{Y, i+1}_S$,  $e^{Y, i+1}_Y$, $\mathcal{D}^{(2)}_{i+1} e^{Y, i+1}_Y$, $\mathcal{D}^{(1)}_{i+1} e^{Y, i+1}_S \in \mathbb{M}_{i,n}(\bar{X}, \bar{Y}, 0)$ as well as Lemma \ref{lem:time:degeneracy:property}, we conclude 
$$
 f(\zeta_{i+1}-\zeta_{i}) \rtheta^{e, Y}_{i+1} \in \mathbb{M}_{i, n}(\bar{X}, \bar{Y}, -1/2),\quad i \in \left\{0, \cdots, n-1\right\}.
 $$

Note also that 
\begin{align*}
e^{X, i+1}_S  & = \partial_{\bar{Y}_i} c_S^{i+1} \\
& = \frac12 (a'_S(\bar{Y}_{i+1}) \partial_{\bar{Y}_i} \bar{Y}_{i+1} - a'_S(m_i) m'_i)\\
& = \frac12 (a'_S(\bar{Y}_{i+1})-   a'_S(m_i)) \partial_{\bar{Y}_i}\bar{Y}_{i+1} + \frac12 a'_S(m_i) (\partial_{\bar{Y}_i} \bar{Y}_{i+1} - m'_i)
\end{align*}
so that, using on the one hand the Lipschitz regularity of $a'_S$ and on the other hand \eqref{partial:flow:deriv:Y:bar:process}, from similar arguments as those used in the proof of Lemma \ref{estimate:theta:time:loca:density}, we conclude that
$$ 
\frac12 a'_S(\bar{Y}_{i+1})-   a'_S(m_i)) \partial_{\bar{Y}_i}\bar{Y}_{i+1}, \quad  \frac12 a'_S(m_i) (\partial_{\bar{Y}_i} \bar{Y}_{i+1} - m'_i) \in \mathbb{M}_{i, n}(\bar{X}, \bar{Y}, 1/2)
$$
\noindent which in turn implies that $e^{X, i+1}_S \in \mathbb{M}_{i,n}(\bar{X}, \bar{Y}, 1/2)$. Moreover, standard computations that we omit show that $\mathcal{D}^{(1)}_{i+1} e^{X, i+1}_S \in \mathbb{M}_{i, n}(\bar{X}, \bar{Y}, 0)$ so that by Lemma \ref{lem:time:degeneracy:property} we deduce
$$
 f(\zeta_{i+1}-\zeta_{i})\rtheta^{e, X}_{i+1}\in \mathbb{M}_{i, n}(\bar{X}, \bar{Y}, 0).
$$
We now prove that $f(\zeta_{i+1}-\zeta_{i}) \rtheta^{c}_{i+1} \in \mathbb{M}_{i, n}(\bar{X}, \bar{Y}, -1/2)$ for any $i \in \left\{0, \cdots, n-1\right\}$. We use the decomposition
\begin{align*}
f(\zeta_{i+1}-\zeta_{i}) \rtheta^{c}_{i+1} & = f(\zeta_{i+1}-\zeta_i) \Big(\mathcal{I}^{(2)}_{i+1}\Big(m'_i \theta_{i+1} - \rtheta^{e, Y}_{i+1}\Big) + \partial_{\bar{Y}_i} \theta_{i+1} - \mathcal{I}^{(1)}_{i+1}( \rtheta^{e,X}_{i+1})\Big) + \mathcal{I}^{(1)}_{i+1}\Big(\partial_{\bar{Y}_i} \bar{X}_{i+1} f(\zeta_{i+1}-\zeta_i)\theta_{i+1}\Big)  \\
& + \mathcal{I}^{(2)}_{i+1}\Big(\Big(\sigma'_{Y, i} \Big(\rho_i Z^{1}_{i+1}+ \sqrt{1-\rho_i^2}Z^{2}_{i+1}\Big) + \sigma_{Y,i} \frac{\rho'_i}{\sqrt{1-\rho^2_i}}\Big(\sqrt{1-\rho_i^2} Z^{1}_{i+1} - \rho_i Z^{2}_{i+1}\Big) \Big) f(\zeta_{i+1}-\zeta_i) \theta_{i+1}\Big).
\end{align*}
 We first prove that $f(\zeta_{i+1}-\zeta_i) \Big(\mathcal{I}^{(2)}_{i+1}\Big(m'_i \theta_{i+1} - \rtheta^{e, Y}_{i+1}\Big) + \partial_{\bar{Y}_i} \theta_{i+1} - \mathcal{I}^{(1)}_{i+1}( \rtheta^{e,X}_{i+1})\Big)  \in \mathbb{M}_{i, n}(\bar{X}, \bar{Y}, -1/2)$. We investigate each term appearing on the right-hand side of \eqref{one:part:theta:c}.

In particular, we first use the fact that $c^{i+1}_S,\, c^{i+1}_Y, \, b^{i+1}_Y, \, c^{i+1}_{Y,S} \in \mathbb{M}_{i, n}(\bar{X}, \bar{Y}, 1/2)$, $\partial_{\bar{Y}_i} c^{i+1}_S, \partial_{\bar{Y}_i} b^{i+1}_Y \in \mathbb{M}_{i, n}(\bar{X}, \bar{Y}, 0)$ and the fact that when one applies the differential operators $\mathcal{D}^{(\alpha_1)}_{i+1}, \, \mathcal{D}^{(\alpha_1, \alpha_2)}_{i+1}$ to these elements the resulting random variables belong to $\mathbb{M}_{i, n}(\bar{X}, \bar{Y}, 0)$ for any $(\alpha_1, \alpha_2) \in \left\{1, 2\right\}^2$. From Lemma \ref{lem:time:degeneracy:property}, we thus conclude that the elements $ \mathcal{I}^{(1)}_{i+1}(\partial_{\bar{Y}_i}c^{i+1}_S)$, $ \mathcal{I}^{(1,1)}_{i+1}(c^{i+1}_S)$, $\mathcal{I}^{(1,2)}_{i+1}(c^{i+1}_S)$, $\mathcal{I}^{(2,1)}_{i+1}(c^{i+1}_S)$, $\mathcal{I}^{(2)}_{i+1}(\partial_{\bar{Y}_i}b^{i+1}_Y)$, $\mathcal{I}^{(2,2)}_{i+1}(c^{i+1}_Y)$, $\mathcal{I}^{(1, 2)}_{i+1}(c^{i+1}_{Y,S})$ belong to $\mathbb{M}_{i, n}(\bar{X}, \bar{Y}, -1/2)$ and that $ \mathcal{I}^{(1)}_{i+1}(c^{i+1}_S),  \mathcal{I}^{(2)}_{i+1}(b^{i+1}_Y)$ belong to $\mathbb{M}_{i, n}(\bar{X}, \bar{Y}, 0)$. Moreover, using \A{ND}, one gets that there exists $C>0$ such that for any $i \in \left\{0, \cdots, n-1\right\}$, $|\sigma'_{S, i}/\sigma_{S, i}| +|\sigma'_{Y, i}/\sigma_{Y, i}| + |\sigma_{Y, i}/\sigma_{S, i}| + |\rho'_i/(1-\rho_i^2)| + |\rho'_i \rho_i/ (1-\rho_i^2)| \leq C $.  We thus conclude that $f(\zeta_{i+1}-\zeta_i) \Big(\mathcal{I}^{(2)}_{i+1}\Big(m'_i \theta_{i+1} - \rtheta^{e, Y}_{i+1}\Big) + \partial_{\bar{Y}_i} \theta_{i+1} - \mathcal{I}^{(1)}_{i+1}( \rtheta^{e,X}_{i+1})\Big)  \in \mathbb{M}_{i, n}(\bar{X}, \bar{Y}, -1/2)$.

It thus suffices to prove $\mathcal{I}^{(1)}_{i+1}\Big(\partial_{\bar{Y}_i} \bar{X}_{i+1} f(\zeta_{i+1}-\zeta_i) \theta_{i+1} \Big)$, $ \mathcal{I}^{(2)}_{i+1}\Big( \sigma'_{Y,i} (\rho_iZ^{1}_{i+1}+ \sqrt{1-\rho_i^2}Z^{2}_{i+1}) f(\zeta_{i+1}-\zeta_i)\theta_{i+1} \Big)$ and $\mathcal{I}^{(2)}_{i+1}\Big( \sigma_{Y,i} \frac{\rho'_i}{\sqrt{1-\rho^2_i}}\Big(\sqrt{1-\rho_i^2} Z^{1}_{i+1} - \rho_i Z^{2}_{i+1}\Big) \Big) f(\zeta_{i+1}-\zeta_i) \theta_{i+1}\Big)$ belong to $\mathbb{M}_{i, n}(\bar{X}, \bar{Y}, -1/2)$. In order to do this, we remark that
\begin{align*}
&\partial_{\bar{Y}_i} \bar{X}_{i+1}   = -\frac12 a'_{S, i} + \sigma'_{S, i} Z^{1}_{i+1} \in \mathbb{M}_{i,n}(\bar{X}, \bar{Y}, 1/2), \\
& \mathcal{D}^{(1)}_{i+1}( \partial_{\bar{Y}_i} \bar{X}_{i+1} ) = \frac{\sigma'_{S,i}}{\sigma_{S,i}} \in \mathbb{M}_{i,n}(\bar{X}, \bar{Y}, 0), \\
&  \sigma'_{Y,i}  (\rho_i Z^{1}_{i+1} + \sqrt{1-\rho_i^2} Z^{2}_{i+1})  \in \mathbb{M}_{i,n}(\bar{X}, \bar{Y}, 1/2), \\
&  \mathcal{D}^{(2)}_{i+1}\Big(\sigma'_{Y,i} (\rho_i Z^{1}_{i+1}+ \sqrt{1-\rho_i^2} Z^{2}_{i+1}\Big)  = \frac{\sigma'_{Y, i}}{\sigma_{Y,i}} \in  \mathbb{M}_{i,n}(\bar{X}, \bar{Y}, 0), \\
&   \sigma_{Y,i} \frac{\rho'_i}{\sqrt{1-\rho^2_i}}\Big(\sqrt{1-\rho_i^2} Z^{1}_{i+1} - \rho_i Z^{2}_{i+1}\Big)\in \mathbb{M}_{i,n}(\bar{X}, \bar{Y}, 1/2), \\
& \mathcal{D}^{(2)}_{i+1}\Big(\sigma_{Y,i} \frac{\rho'_i}{\sqrt{1-\rho^2_i}}\Big(\sqrt{1-\rho_i^2} Z^{1}_{i+1} - \rho_i Z^{2}_{i+1}\Big) \Big) = - \frac{\rho_i\rho'_i }{1- \rho_i^2}  \in \mathbb{M}_{i,n}(\bar{X}, \bar{Y}, 0)
 \end{align*}
 \noindent and from Lemma \ref{estimate:theta:time:loca:density}, $f(\zeta_{i+1}-\zeta_i) \theta_{i+1} \in  \mathbb{M}_{i,n}(\bar{X}, \bar{Y}, -1/2)$. From Lemma \ref{lem:time:degeneracy:property}, it follows that $f(\zeta_{i+1}-\zeta_i)\theta_{i+1}  \partial_{\bar{Y}_i} \bar{X}_{i+1} \in   \mathbb{M}_{i,n}(\bar{X}, \bar{Y}, 0) $ and $f(\zeta_{i+1}-\zeta_i)\theta_{i+1}  \sigma'_Y(m_i) m'_i (\rho(W_{\zeta_{i+1}}-W_{\zeta_i})+ \sqrt{1-\rho^2}(\tilde{W}_{\zeta_{i+1}}-\tilde{W}_{\zeta_i})) \in  \mathbb{M}_{i,n}(\bar{X}, \bar{Y}, 0) $. Now following similar computations as those employed in the proof of Lemma \ref{estimate:theta:time:loca:density} and omitting some technical details we obtain $\mathcal{D}^{(\alpha)}_{i+1}(f(\zeta_{i+1}-\zeta_i) \theta_{i+1}) \in \mathbb{M}_{i, n}(\bar{X}, \bar{Y}, -1)$ so that from the chain rule formula and Lemma \ref{lem:time:degeneracy:property}, the random variables $\mathcal{D}^{(\alpha)}_{i+1}( f(\zeta_{i+1}-\zeta_i)\theta_{i+1}  \partial_{\bar{Y}_i} \bar{X}_{i+1} ) $, $\mathcal{D}^{(\alpha)}_{i+1}(\sigma'_{Y,i} (\rho_i Z^{1}_{i+1}+ \sqrt{1-\rho_i^2} Z^2_{i+1}) f(\zeta_{i+1}-\zeta_i)\theta_{i+1}   )$ and $\mathcal{D}^{(\alpha)}_{i+1}(\sigma_{Y,i} \frac{\rho'_i}{\sqrt{1-\rho^2_i}} (\sqrt{1-\rho_i^2} Z^{1}_{i+1} - \rho_i Z^{2}_{i+1})f(\zeta_{i+1}-\zeta_i)\theta_{i+1}   )$ belong to $\mathbb{M}_{i, n}(\bar{X}, \bar{Y}, -1/2)$. From Lemma \ref{lem:time:degeneracy:property}, we thus conclude that $\mathcal{I}^{(1)}_{i+1}\Big(\partial_{\bar{Y}_i} \bar{X}_{i+1}f(\zeta_{i+1}-\zeta_i) \theta_{i+1} \Big)$, $ \mathcal{I}^{(2)}_{i+1}\Big( \sigma'_{Y, i} (\rho_i Z^{1}_{i+1} + \sqrt{1-\rho_i^2} Z^2_{i+1}) f(\zeta_{i+1}-\zeta_i) \theta_{i+1} \Big)$ and $ \mathcal{I}^{(2)}_{i+1}\Big(\sigma_{Y,i} \frac{\rho'_i}{\sqrt{1-\rho^2_i}} (\sqrt{1-\rho_i^2} Z^{1}_{i+1} - \rho_i Z^{2}_{i+1}) f(\zeta_{i+1}-\zeta_i) \theta_{i+1} \Big)$ belong to $ \mathbb{M}_{i, n}(\bar{X}, \bar{Y}, -1/2)$. From the preceding arguments, we eventually deduce that $f(\zeta_{i+1}-\zeta_{i}) \rtheta^{c}_{i+1} \in \mathbb{M}_{i, n}(\bar{X}, \bar{Y}, -1/2)$ for any $i \in \left\{0, \cdots, n-1\right\}$.
 
Finally, from the very definition of the weights on the last time interval $\rtheta^{e, Y}_{n+1}$ and $\rtheta^{e, X}_{n+1}$ one directly gets that
$$
(1-F(T-\zeta_n)) \rtheta^{e, Y}_{n+1} =  m'_{n} + \sigma'_{Y,n}  \Big(\rho_n Z^{1}_{n+1}+ \sqrt{1-\rho_n^2} Z^{2}_{n+1}\Big) + \sigma_{Y, n} \frac{\rho'_n}{\sqrt{1-\rho^2_n}} \Big(\sqrt{1-\rho^2_n} Z^{1}_{n+1} -\rho_n Z^{2}_{n+1} \Big)
$$
\noindent belongs to $\mathbb{M}_{n, n}(\bar{X}, \bar{Y}, 0)$ and that
$$
(1-F(T-\zeta_n)) \rtheta^{e, X}_{n+1} = -\frac12 a'_{S, n} + \sigma'_{S, n} Z^{1}_{n+1} 
$$
\noindent belongs to $\mathbb{M}_{n, n}(\bar{X}, \bar{Y}, 1/2)$. The proof is now complete.

\section{Some technical results}

\subsection{Emergence of jumps in the renewal process $N$}\label{add:jumps:renewal:process}
The first result is used in the proof of the probabilistic representation in Theorem \ref{theorem:probabilistic:formulation} and is used to express that time integrals add jumps to the renewal process $N$. In what follows, $N$ is a renewal process in the sense of Definition \ref{counting:process}.
\begin{lem}
	\label{lem:add:new:jumps}
	Let $ n \in \N $ and $ G:\{(t_1,\dots,t_{n+2}): 0<t_1<\dots<t_{n+1}<t_{n+2}:=T\}\rightarrow \mathbb{R}$ be a measurable function such that
	$\E\Big[\int_{\zeta _n}^T |G(\zeta_1,\dots,\zeta_{n},s,T)| \I_\seq{N_T=n}ds\Big]<\infty  $. Then, it holds
	\begin{align*}
	& \E\Big[\int_{\zeta_n}^TG(\zeta_1,\dots,\zeta_{n},s,T)\I_\seq{N_T=n}ds\Big] \\
	& \quad =
	\E\Big[G(\zeta_1,\dots,\zeta_{n},\zeta_{n+1},T) (1-F(T-\zeta_{n+1}))^{-1} (1-F(T-\zeta_n)) (f(\zeta_{n+1}-\zeta_n))^{-1}\I_\seq{N_T=n+1}\Big].
	\end{align*}
\end{lem}
\begin{proof}
The proof follows by rewriting the above expectations using \eqref{probabilistic:representation:time:integrals}. We rewrite the expectation on the right-hand side in integral form. By Fubini's theorem, we obtain
\begin{align*}
& \E\Big[G(\zeta_1,\dots,\zeta_{n},\zeta_{n+1},T) (1-F(T-\zeta_{n+1}))^{-1} (1-F(T-\zeta_n)) (f(\zeta_{n+1}-\zeta_n))^{-1}\I_\seq{N_T=n+1}\Big] \\
& = \int_{\Delta_{n+1}(T)} G(s_1, \cdots, s_{n+1}, T) (1-F(T-s_{n+1}))^{-1} (1-F(T-s_n))(f(s_{n+1}-s_n))^{-1} \\
& \quad \times (1-F(T-s_{n+1})) \prod_{j=0}^{n} f(s_{j+1}-s_{j}) \, d\s_{n+1}\\
& = \int_{\Delta_n(T)} \int_{s_n}^{T} G(s_1, \cdots, s_{n+1}, T)  \, ds_{n+1} \, (1-F(T-s_n))\prod_{j=0}^{n-1} f(s_{j+1}-s_{j}) \, d\s_{n}.
\end{align*}

This completes the proof.
\end{proof}

\begin{lem}\label{estimate:theta:time:loca:density} Let $n\in \mathbb{N}$. On the set $\left\{ N_T= n\right\}$, the sequence of weights $(\theta_i)_{1\leq i \leq n+1}$ defined by \eqref{definition:first:theta:theorem:probabilistic:representation} and \eqref{definition:theta:theorem:probabilistic:representation} satisfy:
\begin{align}
\forall i \in \left\{1,\cdots, n\right\}, \, f(\zeta_i - \zeta_{i-1}) \theta_{i} \in \mathbb{M}_{i-1,n}(\bar{X}, \bar{Y}, -1/2), \quad (1-F(T- \zeta_{n})) \theta_{n+1} \in \mathbb{M}_{n,n}(\bar{X}, \bar{Y}, 0).
\end{align}
\end{lem}

\begin{proof} We investigate each term appearing in the definition of $\theta_i \in \mathbb{S}_{i-1, n}(\bar{X}, \bar{Y})$ and seek to apply Lemma \ref{lem:time:degeneracy:property}. From the Lipschitz property of $a_S$ and the space-time inequality \eqref{space:time:inequality}, for any $c>0$ and any $c'>c$, the map $(x_{i-1}, y_{i-1}, x_i, y_i, \s_{n+1}) \mapsto c^{i}_S(x_{i-1}, y_{i-1}, x_i, y_i, \s_{n+1})$ satisfies 
\begin{align*}
|c^{i}_S(x_{i-1}, y_{i-1}, x_i, y_i, \s_{n+1})|^p \bar{q}_{c}(s_{i}-s_{i-1}, x_{i-1}, y_{i-1}, x_i, y_i)& \leq C |y_i-m_{i-1}(y_{i-1})|^p  \bar{q}_{c}(s_{i}-s_{i-1}, x_{i-1}, y_{i-1}, x_i, y_i)\\
&  \leq C (s_{i}-s_{i-1})^{p/2}  \bar{q}_{c'}(s_{i}-s_{i-1}, x_{i-1}, y_{i-1}, x_i, y_i)
\end{align*}

\noindent so that, the random variables $c^{i}_S \in \mathbb{M}_{i-1,n}(\bar{X}, \bar{Y}, 1/2)$, for any $i \in \left\{1, \cdots, n+1\right\}$. Moreover, from the boundedness of the first and second derivatives of $a_S$, it follows
$$
\mathcal{D}^{(1)}_{i} c^{i}_S = \partial_{\bar{Y}_i} c^{i}_S(\bar{X}_{i-1}, \bar{Y}_{i-1}, \bar{X}_i, \bar{Y}_i, \zeta^{n+1}) = \partial_4 c^{i}_S(\bar{X}_{i-1}, \bar{Y}_{i-1}, \bar{X}_i, \bar{Y}_i, \zeta^{n+1}) = \frac12 a'_S(\bar{Y}_i) \in \mathbb{M}_{i-1,n}(\bar{X}, \bar{Y}, 0)
$$
\noindent and
$$
\mathcal{D}^{(1,1)}_{i} c^{i}_S = \partial^2_{\bar{Y}_i} c^{i}_S(\bar{X}_{i-1}, \bar{Y}_{i-1}, \bar{X}_i, \bar{Y}_i, \zeta^{n+1}) = \partial^2_4 c^{i}_S(\bar{X}_{i-1}, \bar{Y}_{i-1}, \bar{X}_i, \bar{Y}_i, \zeta^{n+1}) = \frac12 a''_S(\bar{Y}_i) \in \mathbb{M}_{i-1,n}(\bar{X}, \bar{Y}, 0).
$$

From Lemma \ref{lem:time:degeneracy:property}, we thus conclude 
$$
\mathcal{I}^{(1)}_i(c^i_S) \in \mathbb{M}_{i,n}(\bar{X}, \bar{Y}, 0) \quad \mbox{and} \quad \mathcal{I}^{(1,1)}_i(c^i_S) \in \mathbb{M}_{i,n}(\bar{X}, \bar{Y}, -1/2), \quad  i \in \left\{1, \cdots, n\right\}.
$$

In a completely analogous manner, omitting some technical details, we derive
$$
 \mathcal{I}^{(2)}_{i}(b^{i}_{Y}) \in \mathbb{M}_{i,n}(\bar{X}, \bar{Y}, 0), \quad \mbox{and} \quad  \mathcal{I}^{(1,2)}_{i}(c^{i}_{Y,S}), \,  \mathcal{I}^{(2,2)}_{i}(c^{i}_Y)  \in  \mathbb{M}_{i,n}(\bar{X}, \bar{Y}, -1/2).
$$

Hence, we obtain $f(\zeta_i - \zeta_{i-1}) \theta_{i} \in \mathbb{M}_{i-1,n}(\bar{X}, \bar{Y}, -1/2)$, for any $i \in \left\{1,\cdots, n\right\}$. We finally observe that $(1-F(T- \zeta_{n})) \theta_{n+1}=1 \in  \mathbb{M}_{ n ,n}(\bar{X}, \bar{Y}, 0)$. The proof is now complete.
\end{proof}

\begin{lem}\label{lemma:semigroup:property}Let $T>0$ and $n$ a positive integer. For any $\s_n=(s_1, \cdots, s_n) \in \Delta_{n}(T)$, for any $(x,y) \in \mathbb{R}^2$, for any positive constant $c$ there exist two positive constants $C$ and $c'>c$ such that the transition density $(t, x, y) \mapsto \bar{q}_{c}(t, x_0, y_0, x, y)$ defined by \eqref{def:kernel:bar:q} satisfies the following semigroup property:
$$
\int_{(\mathbb{R}^2)^{n}} \bar{q}_{c}(T-s_n, x_{n}, y_n, x, y) \times \bar{q}_{c}(s_{n}-s_{n-1}, x_{n-1}, y_{n-1}, x_n, y_n) \times  \cdots  \times \bar{q}_{c}(s_1,x_0, y_0, x_1, y_1) \, d\x_n d\y_n \leq C^n\bar{q}_{c'}(T, x_0, y_0, x, y).
$$
\end{lem}

\begin{proof}
The $dx_1 \cdots dx_n$ integrals are treated using the standard semigroup property of Gaussian kernels so that from the very definition of $\bar{q}_{c}$, it directly follows
\begin{align*}
& \int_{(\mathbb{R}^2)^{n}} \bar{q}_{c}(T-s_n, x_{n}, y_n, x, y) \times \bar{q}_{c}(s_{n}-s_{n-1}, x_{n-1}, y_{n-1}, x_n, y_n) \times  \cdots  \times \bar{q}_{c}(s_1,x_0, y_0, x_1, y_1) \, d\x_n d\y_n\\
& \leq  \frac{1}{\sqrt{2\pi cT}} e^{-\frac{(x-x_0)^2}{2c T}} \int_{\mathbb{R}^n} \frac{1}{\sqrt{2\pi c(T-s_n)}}e^{-\frac{(y-m_{T-s_n}(y_n))^2}{2c (T-s_{n})}}\times \cdots \times  \frac{1}{\sqrt{2\pi c s_1}} e^{-\frac{(y_1-m_{s_1}(y_0))^2}{2c s_{1}}} \, d\y_n
\end{align*}

We now perform the change of variables $y_1 = m_{s_1}(z_1), y_2 = m_{s_2}(z_2), \cdots, y_n=m_{s_n}(z_n)$. Observe that since $b_Y$ admits a bounded first derivative the determinants of the Jacobians $J_{s_1}(z_1):= \partial_{x} m_{s_1}(z_1), \cdots, J_{T-s_n}(z_n) = \partial_x m_{T-s_n}(z_ n)$ are (locally) uniformly bounded for any $(s_1, \cdots, s_n) \in \Delta_{n}(T)$. Remark also that from the semigroup property $m_{s_{i+1}-s_{i}}(m_{s_{i}}(z_{i})) = m_{s_{i+1}}(z_{i})$, for $1\leq  i\leq n$ with the convention $s_{n+1}=T$. Hence, for some positive constants $C$ and $c'>c$ that may change from line to line, we get 
\begin{align*}
& \frac{1}{\sqrt{2\pi c T}} e^{-\frac{(x-x_0)^2}{2c T}} \int_{\mathbb{R}^n} \frac{1}{\sqrt{2\pi c(T-s_n)}}e^{-\frac{(y-m_{T-s_n}(y_n))^2}{2c (T-s_{n})}}\times \cdots \times  \frac{1}{\sqrt{2\pi c s_1}} e^{-\frac{(y_1-m_{s_1}(y_0))^2}{2c s_{1}}} \, d\y_n \\
& \leq C^{n}\frac{1}{\sqrt{2\pi cT}} e^{-\frac{(x-x_0)^2}{2c T}}   \int_{\mathbb{R}^n} \frac{1}{\sqrt{2\pi c(T-s_n)}}e^{-\frac{(y-m_{T}(z_n))^2}{2c (T-s_{n})}}\times \cdots \times  \frac{1}{\sqrt{2\pi c s_1}} e^{-\frac{(m_{s_1}(z_1)-m_{s_1}(y_0))^2}{2c s_{1}}} \, d\z_n  \\
& \leq C^{n}\frac{1}{\sqrt{2\pi cT}} e^{-\frac{(x-x_0)^2}{2c T}} \int_{\mathbb{R}^n} \frac{1}{\sqrt{2\pi c(T-s_n)}}e^{-C^{-1}\frac{(m^{-1}_T(y)-z_n)^2}{2c (T-s_{n})}}\times \cdots \times  \frac{1}{\sqrt{2\pi c s_1}} e^{-C^{-1}\frac{(z_1-y_0)^2}{2c s_{1}}} \, d\z_n  \\
&  \leq C^{n}\frac{1}{2\pi c CT} e^{-\frac{(x-x_0)^2}{2c T}}  e^{- \frac{(m^{-1}_T(y)-y_0)^2}{2c C t}}\\
& \leq C^n \bar{q}_{c'}(T x_0, y_0, x_1, y_1)
\end{align*}

\noindent where we first used the bi-Lipschitz property of the flow $(s, x)\mapsto m_s(x)$ which yields 
$$
\forall t\in [0,T], \, C^{-1} |x-z|^2 \leq |m_{t}(x) - m_{t}(z)|^2 \leq C |x-z|^2
$$
\noindent for some positive constant $C\geq1$ and then the semigroup property satisfied by Gaussian kernels. This completes the proof.
\end{proof}

\section{Some useful formulas}
We here provide some useful formulas in order to device the unbiased Monte Carlo algorithms based on Theorem \ref{theorem:probabilistic:formulation} and Theorem \ref{thm:ibp:formulae}. Their proofs follow from standard computations as those used in subsection \ref{sec:tailor:mall:calculus} and are omitted.

The following formulae are required in order to compute the weights $(\theta_i)_{1\leq i \leq N_T}$ appearing in the identity \eqref{probabilistic:formulation:theorem}. Note that in our examples since $a_Y(.)$ is constant, one has $c^{i}_Y(.) \equiv 0$ for $i \in \left\{1, \cdots, N_T\right\}$. Hence, for $i \in \left\{1, \cdots, N_T\right\}$, one has:
{\small
\begin{align*}
	\mathcal{I}_i^{(1)}(c_S^i) &= c_S^i \mathcal{I}_i^{(1)}(1), \\
	   	      	\mathcal{I}_i^{(1, 1)}(c_S^i) &= c_S^i ((\mathcal{I}_i^{(1)}(1))^2 - \mathcal{D}_i^{(1)} \mathcal{I}_i^{(1)}(1)), \\
	\mathcal{I}_i^{(2)}(b_Y^i)  &= b_Y^i \mathcal{I}_i^{(2)}(1) - b'_Y(\bar{Y}_i) , \\
	\mathcal{I}_i^{(1, 2)}(c_{Y, S}^i) &= \mathcal{I}_i^{(2)}(c_{Y, S}^i \mathcal{I}_i^{(1)}(1)  ) =  c_{Y, S}^i \mathcal{I}_i^{(1)}(1)\mathcal{I}_i^{(2)}(1) - \mathcal{I}_i^{(1)}(1)  \mathcal{D}_i^{(2)}(c_{Y,S}^i) - c^i_{Y,S}  \mathcal{D}_i^{(2)} \mathcal{I}_i^{(1)}(1).
   \end{align*}
}   
   
The following formulae are needed in order to compute the weights for the Delta appearing in the identity \eqref{IBP:s0}, for $i \in \left\{1, \cdots, N_T\right\}$ one has:
{\small
\begin{align*}
	\mathcal{D}_{i}^{(1)}(\mathcal{I}_i^{(1, 1)}(c_S^i) ) &= 2 c_S^i \mathcal{I}_i^{(1)}(1) \mathcal{D}_{i}^{(1)}\mathcal{I}_i^{(1)}(1)  ,\\
	\mathcal{D}_{i}^{(1)}( \mathcal{I}_i^{(1)}(c_S^i) ) &= c_S^i \mathcal{D}_{i}^{(1)} \mathcal{I}_i^{(1)}(1),\\
	\mathcal{D}_{i}^{(1)}( \mathcal{I}_i^{(2)}(b_Y^i) ) &=  b_Y^i \mathcal{D}_{i}^{(1)}\mathcal{I}_i^{(2)}(1)  ,\\
	\mathcal{D}_{i}^{(1)}( \mathcal{I}_i^{(1, 2)}(c_{Y, S}^i)) &= c_{Y,S}^i  \mathcal{D}_{i}^{(1)}\mathcal{I}_i^{(1)}(1) \mathcal{I}_i^{(2)}(1) +  c_{Y,S}^i  \mathcal{D}_{i}^{(1)} \mathcal{I}_i^{(2)}(1) \mathcal{I}_i^{(1)}(1)  -   \mathcal{D}_{i}^{(1)} \mathcal{I}_i^{(1)}(1) \mathcal{D}_i^{(2)}(c_{Y,S}^i).
	\end{align*}
	
	\begin{align*}
	 \mathcal{D}^{(1)}_i \theta_i  &= (f(\zeta_{i}-\zeta_{i-1}))^{-1}\Big[  \mathcal{D}_i^{(1)} \mathcal{I}^{(1,1)}_{i}(c^{i}_S) -  \mathcal{D}_i^{(1)}  \mathcal{I}^{(1)}_{i}(c^{i}_S)  +  \mathcal{D}_i^{(1)} \mathcal{I}^{(2)}_{i}(b^{i}_Y) + \mathcal{D}_i^{(1)} \mathcal{I}^{(1,2)}_{i}(c^{i}_{Y, S}) \Big], \\
	 \mathcal{I}^{(1)}_k (\theta_k) &=  \mathcal{I}^{(1)}_k (1)  \theta_k - \mathcal{D}^{(1)}_k \theta_k , \, k \leq N_T, \\
	\mathcal{I}^{(1)}_{N_T+1} (\theta_{N_T+1}) &= \theta_{N_T+1} \mathcal{I}^{(1)}_{N_T+1} (1) - \mathcal{D}^{(1)}_{N_T+1} \theta_{N_T+1} = \theta_{N_T+1} \mathcal{I}^{(1)}_{N_T+1} (1).
\end{align*}
}
The following formulae are required for the computation of the weights for the Vega appearing in the identity \eqref{IBP:y0}, for $i \in \left\{1, \cdots, N_T\right\}$ it holds:
{\small
\begin{align*}
\mathcal{I}_{i+1}^{(1, 1)}(d_S^{i+1}) &= m'_i \mathcal{I}_{i+1}^{(1, 1)}(c_S^{i+1}) ,\\
    \mathcal{I}_{i+1}^{(1)}(e_S^{Y, i+1}) & =  -m'_i \mathcal{I}_{i+1}^{(1)}(c_S^{i+1})  +  \mathcal{D}_{i}^{(2)}(c_{Y,S}^{i+1}) \mathcal{I}_{i+1}^{(1)}(1) -   \mathcal{D}_{i+1}^{(1)} \mathcal{D}_{i}^{(2)}(c_{Y,S}^{i+1})  ,\\
     \mathcal{I}_{i+1}^{(2)}(e_Y^{Y, i+1}) &= m'_i \mathcal{I}_{i+1}^{(2)}(b_Y^{i+1}) ,\\
     \mathcal{I}_{i+1}^{(1, 2)}(d_{Y, S}^{i+1}) &= m'_i  \mathcal{I}_{i+1}^{(1, 2)}(c_{Y, S}^{i+1})  ,\\
     \mathcal{I}_{i+1}^{(1)}( e_S^{X, i+1}) &=   e_S^{X, i+1} \mathcal{I}_{i+1}^{(1)}(1) - \mathcal{D}_{i+1}^{(1)} e_S^{X, i+1} = \mathcal{D}_{i}^{(2)}  (c_S^{i+1}) \mathcal{I}_{i+1}^{(1)}(1) -  \mathcal{D}_{i+1}^{(1)}  \mathcal{D}_{i}^{(2)}  (c_S^{i+1}) ,\\
\nonumber    \\
	\mathcal{D}_{i}^{(2)}( 	\mathcal{I}_i^{(1, 1)}(c_S^i) ) &= \mathcal{D}_{i}^{(2)}(  c_S^i) (\mathcal{I}_i^{(1)}(1))^2 + 2 c_S^i \mathcal{I}_i^{(1)}(1) \mathcal{D}_{i}^{(2)} \mathcal{I}_i^{(1)}(1)   - \mathcal{D}_i^{(2)}(c_S^i) \mathcal{D}_{i}^{(1)} \mathcal{I}_i^{(1)}(1) ,\\
	\mathcal{D}_{i}^{(2)}( \mathcal{I}_i^{(1)}(c_S^i) ) &=  \mathcal{D}_{i}^{(2)}(c_S^i) \mathcal{I}_i^{(1)}(1) + c_S^i   \mathcal{D}_{i}^{(2) } \mathcal{I}_i^{(1)}(1),\\
	\mathcal{D}_{i}^{(2)}( \mathcal{I}_i^{(2)}(b_Y^i) ) &=  \mathcal{D}_{i}^{(2)} b_Y^i  \mathcal{I}_i^{(2)}(1) + b_Y^i \mathcal{D}_{i}^{(2)} \mathcal{I}_i^{(2)}(1) - b''_Y(\bar{Y}_i)  ,\\
	\mathcal{D}_{i}^{(2)}( \mathcal{I}_i^{(1, 2)}(c_{Y, S}^i)) &=  \mathcal{D}_{i}^{(2)}(c_{Y,S}^i ) \mathcal{I}_i^{(1)}(1)\mathcal{I}_i^{(2)}(1) +  c_{Y,S}^i  \mathcal{D}_{i}^{(2)} \mathcal{I}_i^{(2)}(1)\mathcal{I}_i^{(1)}(1)  +  c_{Y,S}^i \mathcal{I}_i^{(2)}(1)  \mathcal{D}_{i}^{(2)} \mathcal{I}_i^{(1)}(1)   \\
	& \quad -   \mathcal{I}_i^{(1)}(1)  \mathcal{D}_i^{(2, 2)}(c_{Y,S}^i) - 2 \mathcal{D}_i^{(2)}(c_{Y,S}^i)  \mathcal{D}_i^{(2)} \mathcal{I}_i^{(1)}(1)  ,\nonumber \\
\nonumber  	\\
	\mathcal{D}_{i-1}^{(2)}( 	\mathcal{I}_i^{(1, 1)}(c_S^i) ) &=  \mathcal{D}_{i-1}^{(2)}(  c_S^i) (\mathcal{I}_i^{(1)}(1))^2 + 2 c_S^i \mathcal{I}_i^{(1)}(1) \mathcal{D}_{i-1}^{(2)} \mathcal{I}_i^{(1)}(1) \\ 
	&\quad - \mathcal{D}_{i-1}^{(2)}(c_S^i) \mathcal{D}_{i}^{(1)} \mathcal{I}_i^{(1)}(1) -  c_S^i  \mathcal{D}_{i-1}^{(2)} (  \mathcal{D}_{i}^{(1)} \mathcal{I}_i^{(1)}(1) )  , \nonumber \\
	\mathcal{D}_{i-1}^{(2)}( \mathcal{I}_i^{(1)}(c_S^i) ) &=  \mathcal{D}_{i-1}^{(2)}(c_S^i) \mathcal{I}_i^{(1)}(1) + c_S^i   \mathcal{D}_{i-1}^{(2) } \mathcal{I}_i^{(1)}(1) ,\\
	\mathcal{D}_{i-1}^{(2)}( \mathcal{I}_i^{(2)}(b_Y^i) ) &= b_Y^i \mathcal{D}_{i-1}^{(2)} \mathcal{I}_i^{(2)}(1) + \mathcal{I}_i^{(2)}(1) (b'_Y(\bar{Y}_i) \partial_{\bar{Y}_{i-1}} \bar{Y}_{i} - b'_Y(m_{i-1}) m_{i-1}') - b''_Y(\bar{Y}_i)  \partial_{\bar{Y}_{i-1}} \bar{Y}_{i}   ,\\
	\mathcal{D}_{i-1}^{(2)}( \mathcal{I}_i^{(1, 2)}(c_{Y, S}^i)) &=  \mathcal{D}_{i-1}^{(2)}(c_{Y,S}^i ) \mathcal{I}_i^{(1)}(1)\mathcal{I}_i^{(2)}(1)  +  c_{Y,S}^i ( \mathcal{I}_i^{(2)}(1)  \mathcal{D}_{i-1}^{(2)} \mathcal{I}_i^{(1)}(1)  +   \mathcal{I}_i^{(1)}(1)  \mathcal{D}_{i-1}^{(2)} \mathcal{I}_i^{(2)}(1) )
	 \\
	 & \quad -   \mathcal{I}_i^{(1)}(1)  \mathcal{D}_{i-1}^{(2)}( \mathcal{D}_i^{(2)}(c_{Y,S}^i) )   -   \mathcal{D}_{i-1}^{(2)}(\mathcal{I}_i^{(1)}(1))  \mathcal{D}_i^{(2)}(c_{Y,S}^i)  - c_{Y, S}^i \mathcal{D}_{i-1}^{(2)} \mathcal{D}_i^{(2)} \mathcal{I}_i^{(1)}(1) \nonumber  \\
	 &\quad - \mathcal{D}_{i-1}^{(2)}(c_{Y,S}^i)  \mathcal{D}_i^{(2)} \mathcal{I}_i^{(1)}(1) . \nonumber
 \end{align*}
}

{\small
\begin{align*}
\mathcal{D}_{i}^{(2)} \theta_{i} =& (f(\zeta_{i}-\zeta_{i-1}))^{-1} \Big[ \mathcal{D}_{i}^{(2)}( 	\mathcal{I}_i^{(1, 1)}(c_S^i) ) - \mathcal{D}_{i}^{(2)}( 	\mathcal{I}_i^{(1)}(c_S^i) )  +  \mathcal{D}_{i}^{(2)} (\mathcal{I}_i^{(2)}(b_Y^i) ) + \mathcal{D}_{i}^{(2)}(\mathcal{I}_i^{(1, 2)}(c_{Y, S}^i)  ) \Big],  \\
\mathcal{D}_{i-1}^{(2)} \theta_{i} =&  (f(\zeta_{i}-\zeta_{i-1}))^{-1} \Big[ \mathcal{D}_{i-1}^{(2)}( 	\mathcal{I}_i^{(1, 1)}(c_S^i) ) - \mathcal{D}_{i-1}^{(2)}( 	\mathcal{I}_i^{(1)}(c_S^i) )  +  \mathcal{D}_{i-1}^{(2)} (\mathcal{I}_i^{(2)}(b_Y^i) ) + \mathcal{D}_{i-1}^{(2)}(\mathcal{I}_i^{(1, 2)}(c_{Y, S}^i)  ) \Big],  
 \end{align*}

\begin{align*}
	\rtheta_{i}^{e, Y} 	&= m_{i-1}' \theta_i   + (f(\zeta_{i}-\zeta_{i-1}))^{-1} ( \mathcal{I}_{i}^{(1)}(1) \mathcal{D}_{i-1}^{(2)}(c_{Y,S}^{i}) - \mathcal{D}_{i}^{(1)} \mathcal{D}_{i-1}^{(2)}(c_{Y,S}^{i}) ) , \\
	\mathcal{D}_{i}^{(2)}\rtheta_{i}^{e, Y} 	&= m'_{i-1} \mathcal{D}_{i}^{(2)} \theta_{i}  +   (f(\zeta_{i}-\zeta_{i-1}))^{-1}  \Big[\mathcal{D}_{i}^{(2)}( \mathcal{D}_{i-1}^{(2)}(c_{Y,S}^{i}))   \mathcal{I}_{i}^{(1)}(1) \\
	& \quad  +  \mathcal{D}_{i-1}^{(2)}(c_{Y,S}^{i}) \mathcal{D}_{i}^{(2)} \mathcal{I}_{i}^{(1)}(1)  - \mathcal{D}_{i}^{(2)}\mathcal{D}_{i}^{(1)} \mathcal{D}_{i-1}^{(2)}(c_{Y,S}^{i}) \Big] ,   \nonumber \\
	 \mathcal{I}^{(2)}_{i}(\rtheta^{e, Y}_i)  &=  \rtheta^{e, Y}_i  \mathcal{I}^{(2)}_{i}(1) -  \mathcal{D}_{i}^{(2)}  \rtheta^{e, Y}_i  ,\\ 
	  \mathcal{I}^{(2)}_{i}(m'_{i-1} \theta_i - \rtheta^{e, Y}_i)  &  = - (f(\zeta_{i}-\zeta_{i-1}))^{-1} \mathcal{I}^{(1, 2)}_{i} (  \mathcal{D}_{i-1}^{(2)}(c_{Y,S}^{i})  ) \\
	  &=  - (f(\zeta_{i}-\zeta_{i-1}))^{-1} \Big(  \mathcal{D}_{i-1}^{(2)}(c_{Y,S}^{i})\mathcal{I}^{(1)}_{i} (1) \mathcal{I}^{(2)}_{i} (1) -  \mathcal{I}^{(1)}_{i} (1) \mathcal{D}_{i}^{(2)} \mathcal{D}_{i-1}^{(2)}(c_{Y,S}^{i}) \nonumber  \\
	  & \quad -   \mathcal{D}_{i-1}^{(2)}(c_{Y,S}^{i}) \mathcal{D}_{i}^{(1)}  \mathcal{I}^{(2)}_{i} (1) -  \mathcal{D}_{i}^{(1)} \mathcal{D}_{i-1}^{(2)}(c_{Y,S}^{i}) \mathcal{I}^{(2)}_{i}(1) + \mathcal{D}_i^{(2)} \mathcal{D}_{i}^{(1)} \mathcal{D}_{i-1}^{(2)}(c_{Y,S}^{i})    \Big) , \nonumber  \\
\end{align*}

\begin{align*}
	\rtheta_{i}^{e, X} &=  (f(\zeta_{i}-\zeta_{i-1}))^{-1}  \mathcal{I}_{i}^{(1)}( e_S^{X, i}) =  (f(\zeta_{i}-\zeta_{i-1}))^{-1} ( \mathcal{I}_{i}^{(1)}(1) \mathcal{D}_{i-1}^{(2)}(c_{S}^{i}) -  \mathcal{D}_{i}^{(1)}  \mathcal{D}_{i-1}^{(2)}  (c_S^{i}) ) ,\\
	 \mathcal{D}_{i}^{(1)} \rtheta_{i}^{e, X}   &=  (f(\zeta_{i}-\zeta_{i-1}))^{-1} e_S^{X, i}  \mathcal{D}_{i}^{(1)}  \mathcal{I}_{i}^{(1)}(1),\\
	   \mathcal{I}_{i}^{(1)}( \rtheta_{i}^{e, X} ) &=  \rtheta_{i}^{e, X} \mathcal{I}_{i}^{(1)}(1)  -   \mathcal{D}_{i}^{(1)} \rtheta_{i}^{e, X}   ,\\
	    \mathcal{I}_{i}^{(1)}(  \theta_i \mathcal{D}_{i-1}^{(2)} \bar{X}_{i} ) &=  ( \theta_{i}  \mathcal{I}_{i}^{(1)} (1)  -  \mathcal{D}_{i}^{(1)}  \theta_{i} )   \mathcal{D}_{i-1}^{(2)} \bar{X}_{i}  -  \mathcal{D}_{i}^{(1)}  ( \mathcal{D}_{i-1}^{(2)} \bar{X}_{i}) \theta_i ,
\end{align*}

\begin{align*}
	   \mathcal{I}^{(2)}_{i} \Big( (\sqrt{1-\rho_{i-1}^2} Z^{1}_{i} -& \rho_{i-1} Z^{2}_{i} ) \theta_{i} \Big)   =  \left(\sqrt{1-\rho_{i-1}^2} Z^{1}_{i} - \rho_{i-1} Z^{2}_{i} \right) ( \mathcal{I}^{(2)}_{i} (1) \theta_{i} - \mathcal{D}_{i}^{(2)} \theta_{i}  ) + \frac{\rho_{i-1} \theta_i }{ \sigma_{Y, i-1} \sqrt{1-\rho^2_{i-1}}}    ,
\end{align*}

\begin{align*}
	\rtheta_{i}^{c} 	=&  \mathcal{I}^{(2)}_{i}(m'_{i-1} \theta_i - \rtheta^{e, Y}_i)  + \sigma_{Y,i - 1} \frac{\rho'_{i-1}}{\sqrt{1-\rho^2_{i-1}}}  \mathcal{I}^{(2)}_{i} \Big( (\sqrt{1-\rho_{i-1}^2} Z^{1}_{i} - \rho_{i-1} Z^{2}_{i} ) \theta_{i} \Big)  \\
	& +   \mathcal{I}_{i}^{(1)}(  \theta_i \mathcal{D}_{i-1}^{(2)} \bar{X}_{i} )  -    \mathcal{I}_{i}^{(1)}( \rtheta_{i}^{e, X} )  + \mathcal{D}_{i-1}^{(2)}  \theta_{i}, \nonumber  \\
 \nonumber	\\
	\rtheta_{N_T+1}^{e, Y} =&  \theta_{N_T +1} \Big(m'_{N_T} + \sigma_{Y, N_T} \frac{\rho'_{N_T}}{\sqrt{1 - \rho^2_{N_T}}} (\sqrt{1 - \rho^2_{N_T}}Z^1_{N_T+1} - \rho_{N_T} Z^2_{N_T+1}) \Big) , \\
	\rtheta_{N_T+1}^{e, X} =& \theta_{N_T +1} \Big( -\frac{1}{2} a'_{S, N_T} + \sigma'_{S, N_T} Z^1_{N_T+1} \Big), \\
	\rtheta_{N_T+1}^{c} 	=& 0,  \\
	 \mathcal{I}^{(2)}_{N_T+1}(\rtheta^{e, Y}_{N_T+1})  =& \rtheta^{e, Y}_{N_T+1}  \mathcal{I}^{(2)}_{N_T+1}(1)  - \mathcal{D}^{(2)}_{N_T+1}  \rtheta^{e, Y}_{N_T+1} =  \rtheta^{e, Y}_{N_T+1}  \mathcal{I}^{(2)}_{N_T+1}(1) + \theta_{N_T +1} \frac{\rho'_{N_T} \rho_{N_T}}{1 - \rho^2_{N_T}} ,\\
	   \mathcal{I}^{(2)}_{N_T+1}(\rtheta^{e, X}_{N_T+1})  =& \rtheta^{e, X}_{N_T+1}  \mathcal{I}^{(2)}_{N_T+1}(1)  - \mathcal{D}^{(2)}_{N_T+1}  \rtheta^{e, X}_{N_T+1} =  \rtheta^{e, X}_{N_T+1}  \mathcal{I}^{(2)}_{N_T+1}(1) - \theta_{N_T +1} \frac{ \sigma'_{S, N_T}  }{\sigma_{S, N_T} } .
\end{align*}

}

\end{document}